\documentclass[reqno]{amsart}

\usepackage[all]{xy}
\usepackage{graphicx}

\usepackage{amssymb}
\usepackage{amsmath}
\usepackage{mathrsfs}
\usepackage{epsfig}
\usepackage{amscd}
\usepackage{graphicx, color}

\def\E{\ifmmode{\mathbb E}\else{$\mathbb E$}\fi} 
\def\N{\ifmmode{\mathbb N}\else{$\mathbb N$}\fi} 
\def\R{\ifmmode{\mathbb R}\else{$\mathbb R$}\fi} 
\def\Q{\ifmmode{\mathbb Q}\else{$\mathbb Q$}\fi} 
\def\C{\ifmmode{\mathbb C}\else{$\mathbb C$}\fi} 
\def\H{\ifmmode{\mathbb H}\else{$\mathbb H$}\fi} 
\def\Z{\ifmmode{\mathbb Z}\else{$\mathbb Z$}\fi} 
\def\P{\ifmmode{\mathbb P}\else{$\mathbb P$}\fi} 
\def\T{\ifmmode{\mathbb T}\else{$\mathbb T$}\fi} 
\def\SS{\ifmmode{\mathbb S}\else{$\mathbb S$}\fi} 
\def\DD{\ifmmode{\mathbb D}\else{$\mathbb D$}\fi} 

\newcommand{\del}{\partial}
\newcommand{\Cont}{{\operatorname{Cont}}}

\newcommand{\ben}{\begin{enumerate}}
\newcommand{\een}{\end{enumerate}}
\newcommand{\be}{\begin{equation}}
\newcommand{\ee}{\end{equation}}
\newcommand{\bea}{\begin{eqnarray}}
\newcommand{\eea}{\end{eqnarray}}
\newcommand{\beastar}{\begin{eqnarray*}}
\newcommand{\eeastar}{\end{eqnarray*}}
\newcommand{\bc}{\begin{center}}
\newcommand{\ec}{\end{center}}

\theoremstyle{theorem}
\newtheorem{thm}{Theorem}[section]
\newtheorem{cor}[thm]{Corollary}
\newtheorem{lem}[thm]{Lemma}
\newtheorem{prop}[thm]{Proposition}

\theoremstyle{definition}
\newtheorem{defn}[thm]{Definition}
\newtheorem{rem}[thm]{Remark}

\newtheorem{hypo}[thm]{Hypothesis}

\newtheorem{notation}[thm]{\rm\bfseries{Notation}}

\newtheorem*{thm*}{Theorem}

\numberwithin{equation}{section}

\def\R{{\mathbb R}}

\def\Crit{{\hbox{Crit}}}

\def\E{{\mathbb E}}
\def\Z{{\mathbb Z}}
\def\C{{\mathbb C}}
\def\R{{\mathbb R}}
\def\P{{\mathbb P}}

\def\N{{\mathbb N}}

\def\11{{\mathbb I}}

\def\delbar{{\overline \partial}}

\def\C{\mathbb{C}}
\def\Z{\mathbb{Z}}

\def\T{\mathbb{T}}

\def\Q{\mathbb{Q}}

\def\E{\ifmmode{\mathbb E}\else{$\mathbb E$}\fi} 
\def\N{\ifmmode{\mathbb N}\else{$\mathbb N$}\fi} 
\def\R{\ifmmode{\mathbb R}\else{$\mathbb R$}\fi} 
\def\Q{\ifmmode{\mathbb Q}\else{$\mathbb Q$}\fi} 
\def\C{\ifmmode{\mathbb C}\else{$\mathbb C$}\fi} 
\def\H{\ifmmode{\mathbb H}\else{$\mathbb H$}\fi} 
\def\Z{\ifmmode{\mathbb Z}\else{$\mathbb Z$}\fi} 
\def\P{\ifmmode{\mathbb P}\else{$\mathbb P$}\fi} 
\def\SS{\ifmmode{\mathbb S}\else{$\mathbb S$}\fi} 
\def\DD{\ifmmode{\mathbb D}\else{$\mathbb D$}\fi} 

\def\R{{\mathbb R}}

\def\Crit{{\hbox{Crit}}}
\def\E{{\mathbb E}}
\def\Z{{\mathbb Z}}
\def\C{{\mathbb C}}
\def\R{{\mathbb R}}

\def\N{{\mathbb N}}
\def\LL{{\mathcal L}}
\def\CC{{\mathcal C}}

\def\delbar{{\overline \partial}}

\def\CA{{\mathcal A}}

\def\CC{{\mathcal C}}

\def\CE{{\mathcal E}}
\def\CF{{\mathcal F}}

\def\CH{{\mathcal H}}

\def\CJ{{\mathcal J}}

\def\CL{{\mathcal L}}
\def\CM{{\mathcal M}}

\def\CP{{\mathcal P}}

\def\CP{{\mathcal P}}
\def\CS{{\mathcal S}}

\def\CV{{\mathcal V}}
\def\CW{{\mathcal W}}

\def\darr#1{\raise1.5ex\hbox{$\leftrightarrow$}
\mkern-16.5mu #1}

\def\roughly#1{\raise.3ex\hbox{$#1$\kern-.75em
\lower1ex\hbox{$\sim$}}}

\def\opname#1{\mathop{\kern0pt{\rm #1}}\nolimits}

\def\Im{\opname{Im}}

\def\dim{\opname{dim}}

\def\supp{\operatorname{supp}}

\def\span{\operatorname{span}}

\def\Cont{\operatorname{Cont}}
\def\Crit{\operatorname{Crit}}
\def\Spec{\operatorname{Spec}}
\def\Sing{\operatorname{Sing}}
\def\GFQI{\mathfrak{G}}
\def\Index{\operatorname{Index}}

\def\Image{\operatorname{Image}}
\def\ev{\operatorname{ev}}
\def\Int{\operatorname{Int}}

\begin{document}

\quad \vskip1.375truein

\def\mq{\mathfrak{q}}
\def\mp{\mathfrak{p}}
\def\mH{\mathfrak{H}}
\def\mh{\mathfrak{h}}
\def\ma{\mathfrak{a}}
\def\ms{\mathfrak{s}}
\def\mm{\mathfrak{m}}
\def\mn{\mathfrak{n}}
\def\mz{\mathfrak{z}}
\def\mw{\mathfrak{w}}
\def\Hoch{{\tt Hoch}}
\def\mt{\mathfrak{t}}
\def\ml{\mathfrak{l}}
\def\mT{\mathfrak{T}}
\def\mL{\mathfrak{L}}
\def\mg{\mathfrak{g}}
\def\md{\mathfrak{d}}
\def\mr{\mathfrak{r}}
\def\Cont{\operatorname{Cont}}
\def\Crit{\operatorname{Crit}}
\def\Spec{\operatorname{Spec}}
\def\Sing{\operatorname{Sing}}
\def\GFQI{\text{\rm GFQI}}
\def\Index{\operatorname{Index}}
\def\Cross{\operatorname{Cross}}
\def\Ham{\operatorname{Ham}}
\def\Fix{\operatorname{Fix}}
\def\Graph{\operatorname{Graph}}
\def\ric{\operatorname{Ric}}
\def\lcs{\text{\rm lcs}}

\title[Pseudoholomorphic curves on symplectization]
{Analysis of pseudoholomorphic curves on symplectization:
Revisit via contact instantons}
\author{Yong-Geun Oh, Taesu Kim}
\address{Center for Geometry and Physics, Institute for Basic Science (IBS),
77 Cheongam-ro, Nam-gu, Pohang-si, Gyeongsangbuk-do, Korea 790-784
\& POSTECH, Gyeongsangbuk-do, Korea}
\email{yongoh1@postech.ac.kr}
\address{Center for Geometry and Physics, Institute for Basic Science (IBS),
77 Cheongam-ro, Nam-gu, Pohang-si, Gyeongsangbuk-do, Korea 790-784}
\email{tkim@ibs.re.kr}
\thanks{This work is supported by the IBS project \# IBS-R003-D1.
Both authors would like to also acknowledge MATRIX and the Simons Foundation for their support and funding through the MATRIX-Simons Collaborative Fund of the IBS-CGP and MATRIX workshop on Symplectic Topology.
}


\begin{abstract} In this survey article, we present the analysis of pseudoholomorphic curves
$u : (\dot \Sigma,j) \to (Q \times \R, \widetilde J)$
on the symplectization of contact manifold $(Q,\lambda)$ as a subcase of the analysis of contact
instantons $w:\dot \Sigma \to Q$, i.e., of the maps $w$ satisfying the equation
$$
\delbar^\pi w = 0, \, d(w^*\lambda \circ j) = 0
$$
on the contact manifold $(Q,\lambda)$, which has been carried out by a coordinate-free covariant tensorial
calculus. The latter was initiated by Wang and the first author of the present survey in
\cite{oh-wang:CR-map1,oh-wang:CR-map2} for the closed string case.
More recently the first author
has extended the machinery to the open string case and applied it to the problems
of quantitative contact topology and contact dynamics \cite{oh:contacton-Legendrian-bdy}, \cite{oh:shelukhin-conjecture},
\cite{oh-yso:spectral}.
When the analysis is applied to that of pseudoholomorphic curves
$u = (w,f)$  with $w = \pi_Q \circ u$, $f = s\circ u$ on symplectization,
the outcome is generally stronger  and more accurate
than the common results on the regularity presented in the literature in that all of our a priori estimates can be written purely in terms $w$ not involving $f$.
The a priori elliptic estimates for $w$, especially $W^{2,2}$-estimate,
 are largely consequences of
various Weitzenb\"ock-type formulae \emph{with respect
to the contact triad connection} introduced by Wang and the first
author in \cite{oh-wang:connection}, and the estimate for $f$ is a consequence thereof
by simple integration of the equation $df = w^*\lambda \circ j$.

We also derive a simple precise tensorial formulae for the linearized
operator and for the asymptotic operator that
admit a perturbation theory of the operators with respect to (adapted) almost complex structures:
The latter has been missing in the analysis of pseudoholomorphic curves on symplectization
in the existing literature.
\end{abstract}


\maketitle

\tableofcontents

\section{Introduction and overview}

Let $(Q, \xi)$ be a contact manifold. Assume $\xi$ is coorientable.
Then we can choose a contact form $\lambda$ with $\ker \lambda = \xi$.
With $\lambda$ given, we have the Reeb vector field $R_\lambda$
uniquely determined by the equation $R_\lambda \rfloor d\lambda = 0, \, R_\lambda \rfloor \lambda = 1$.
Then we have decomposition $TM = \xi \oplus \R \{R_\lambda\}$. We denote by $\pi: TM \to \xi$ the associated projection and
$\Pi = \Pi_\lambda: TM \to TM$ the associated idempotent whose image
is $\xi$.

A \emph{contact triad} is a triple $(Q,\lambda, J)$
where $\lambda$ is a contact form of $\xi$, i.e., $\ker \lambda = \xi$ and
$J$ is an endomorphism $J: TM \to TM$ that satisfies the following.

\begin{defn}[$\lambda$-adapted CR almost complex structure]
\label{defn:adapted-J}
A \emph{CR almost complex structure $J$} is an endomorphism
$J: TM \to TM$ satisfying $J^2 = - \Pi$, or more explicitly
$$
(J|_\xi)^2 = - id|_\xi, \quad J(R_\lambda) = 0.
$$
We say $J$ is \emph{adapted to} $\lambda$ if $d\lambda(Y, J Y) \geq 0$ for all $Y \in \xi$ with equality only when $Y = 0$.
In this case, we all the pair $(\lambda, J)$ an \emph{adapted pair} of the contact manifold $(Q,\xi)$.
\end{defn}
The associated contact triad metric is given by
\be\label{eq:triad-metric}
g_\lambda := d\lambda(\cdot, J\cdot) +\lambda\otimes\lambda.
\ee
The symplectization of $(Q,\lambda)$ is the symplectic manifold $(Q \times \R, d(e^s\lambda))$ with $\R$-coordinate
$s$ also called the radial coordinate.
We equip the symplectization with the $s$-translation invariant almost complex structure
$$
\widetilde J = J \oplus J_0
$$
where $J_0$ is the almost
complex structure on the plane $\R \{\frac{\del}{\del s}, R_\lambda\}$ satisfying $J_0(\frac{\del}{\del s}) = R_\lambda$.

The main purpose of the present survey is to advertise
the \emph{covariant} tensorial approach to the analysis of pseudoholomorphic curves
on symplectization via the study of contact instantons which was initiated by
Wang and the first author of the present survey in \cite{oh-wang:CR-map1,oh-wang:CR-map2},
further developed by
the first author in \cite{oh:contacton}--\cite{oh:shelukhin-conjecture} and
also by the present authors in \cite{kim-oh:asymp-analysis}.

\begin{rem} When we say `covariant tensorial', it means that we follow the way how Riemannian geometers
and physicists do their tensor calculations. More specifically, we fix a
`best'  connection $\nabla$ on $Q$ once and for all that optimizes tensor calculations.
In our case, the contact triad connection is such a connection as illustrated by
\cite{oh-wang:CR-map1,oh-wang:CR-map2}, \cite{oh:contacton} and \cite{oh-yso:index}.
Then for a given smooth map $w: \dot \Sigma \to Q$ or
$u: \dot \Sigma \to Q \times \R$, we take covariant derivatives of any induced tensorial quantities only in terms of
the induced connection of $\nabla$ and the Levi-Civita connection of the domain $\dot \Sigma$.
\end{rem}

Also the relevant Fredholm theory and
compactification results are developed by the first author
in \cite{oh:contacton} for the closed string case.
More recently the theory of contact instantons
has been extended in two different directions. On the one hand, in
the joint work by Savelyev and the first author
\cite{oh-savelyev}, they lifted the theory of contact instantons to
the theory of pseudoholomorphic curves on the $\mathfrak{lcs}$-fication
$(Q \times S^1_\rho, \omega_\lambda)$, Banyaga's locally conformal symplectification (which they call the $\mathfrak{lcs}$-fication) of
contact manifold $(Q,\lambda)$ \cite{banyaga:lcs} on
$$
(Q \times S^1_\rho, d\lambda + d\theta \wedge \lambda)
$$
with the canonical angular form $d\theta$ satisfying $\int_{S^1_\rho} d\theta = 1$. According to the terminology adopted in \cite{oh-savelyev},
the authors call them the $\mathfrak{lcs}$-fication `of nonzero temperature'
on which the theory of pseudoholomorphic curves is
developed. Here `lcs' stands for the standard abbreviation of `locally conformal symplectic'. The authors of ibid. call the relevant pseudoholomorphic
curves \emph{lcs instantons}.
This family can be augmented by  including
 the case of the product $Q \times \R$ as the `zero temperature limit'
 with $1/\rho \to 0$,
\be\label{eq:lcsfication}
(Q \times \R, \omega_\lambda), \quad  \omega_\lambda : = d\lambda + ds \otimes \lambda
\ee
(Here $\rho$ represents the radius of the circle $S^1$.)

On the other hand,  in \cite{oh:contacton-Legendrian-bdy}, the first
author of the present paper also extended the theory to the open-string case. It is further
developed in \cite{oh-yso:index}  and applied for the construction of
Legendrian contact instanton cohomology and the associated spectral invariants on the one-jet bundle
in \cite{oh-yso:spectral} jointly with Seungook Yu. Then
the present  authors have carried out precise asymptotic analyses near the punctures
of finite energy contact instantons and of finite energy pseudoholomorphic curves
in \cite{kim-oh:asymp-analysis} in preparation by developing a generic perturbation theory
of asymptotic operators over the change of adapted pairs $(\lambda,J)$.
The tensorial approach also clarifies the relationship between
the background geometries of
the contact triad $(Q,\lambda,J)$, the symplectization
$$
(SQ, d(e^s\lambda)) = (Q \times \R, e^s \omega_\lambda)
$$
and the lcs manifold \eqref{eq:almost-Hermitian-triple}.
Consider the decomposition
\be\label{eq:TM-splitting}
TM \cong \xi \oplus \R\left\{\frac{\del}{\del s},R_\lambda\right\}.
\ee
Let $\widetilde \nabla: = \nabla^{\text{\rm can}}$
be the canonical connection of
this almost Hermitian manifold
\be\label{eq:almost-Hermitian-triple}
(Q \times \R, \widetilde g_\lambda, \widetilde J), \quad  \widetilde g_\lambda : = g_\lambda + ds \otimes ds
\ee
i.e., the unique Riemannian connection whose torsion
$T$ satisfies $T(X,\widetilde J X) = 0$ for all $X \in T(Q \times \R)$
(See \cite{gauduchon,kobayashi,oh:book1} for its definition and basic properties.)
We note that this almost Hermitian structure on $Q\times \R$ is translational invariant in the $s$-direction.

The \emph{first upshot} of our tensorial approach utilizing the contact triad connection lies in our point of view of regarding the product
$$
(Q \times \R, \omega_\lambda), \quad \omega_\lambda: = d\lambda
+ ds \wedge \lambda
$$
as an $\mathfrak{lcs}$-fication of contact manifold $(Q,\lambda)$, and
Hofer's analysis of pseudoholomorphic curves as the analysis of
pseudoholomorphic curves on the \emph{$\mathfrak{lcs}$-fication
(in zero temperature)}
$$
(Q \times \R, \omega_\lambda, \widetilde J)
$$
 of contact triad  $(Q,\lambda,J)$. It is quite apparent that
in Hofer's global analysis of
pseudoholomorphic curves on symplectization
\emph{the symplectic form $d(e^s\lambda)$  plays little role}
but $d\lambda$ and Hofer's ingenious way of considering
translational invariant \emph{$\lambda$-energy} does.

The $\mathfrak{lcs}$-fication carries a canonical
translational invariant  metric
$$
\widetilde g: = \omega_\lambda(\cdot, J\cdot)
$$
so that the triple $(Q \times \R, \widetilde g,\widetilde J)$ becomes an
almost Hermitian (but not almost K\"ahler) manifold. The following geometric relationship between
the contact triad connection and its lcs lifting
has been implicitly exploited.

\begin{prop}[Canonical connection on $\mathfrak{lcs}$-fication]\label{prop:canonical-intro}
Let $\widetilde g = \widetilde g_\lambda $ be the almost Hermitan metric given above.
Let $\widetilde \nabla$ be the canonical connection of the almost Hermitian manifold
\eqref{eq:almost-Hermitian-triple}, and $\nabla$ be the contact triad connection for the triad $(Q,\lambda,J)$.
Then $\widetilde \nabla$ preserves the splitting \eqref{eq:TM-splitting}
and satisfies $\widetilde \nabla|_\xi = \nabla|_\xi$.
\end{prop}

The \emph{second upshot} is our utilization of an important property
of the Levi-Civita connection proved by Blair
\cite{blair:book} for the triad metric on $Q$. This important property has been completely unnoticed
(as far as the authors are aware) in the symplectic community around pseudoholomorphic curves, including the first author
until very recently at the time of preparing  the
paper \cite{kim-oh:asymp-analysis} and this survey. (Indeed this property had been already
mentioned in the paper \cite[Proposition 4]{oh-wang:connection}
by Wang and the first author himself!)

\begin{prop}[Lemma 6.1 \cite{blair:book}]\label{prop:blair} Let $\nabla^{\text{\rm LC}}$ be the Levi-Civita
connection of the triad metric associated to the triad $(Q,\lambda,J)$. Then
\be\label{eq:nablaLCRlambdaJ=0}
\nabla_{R_\lambda}^{\text{\rm LC}} J = 0.
\ee
\end{prop}
This property of the Levi-Civita connection together with the usage of canonical connection
on the $\mathfrak{lcs}$-fication $(Q\times \R, \widetilde J, \omega_\lambda)$
(or equivalently via contact triad connection on the triad $(Q,\lambda,J)$)
plays an important role in the present authors'
analysis of the asymptotic operators of finite energy contact instantons
in \cite{kim-oh:asymp-analysis}
and hence of finite energy pseudoholomorphic curves too.
\begin{rem}
in fact if we
let $\nabla_{R_\lambda}^{\text{\rm LC}}$ or $\nabla_{R_\lambda}$
acted upon $\Gamma(\xi) \subset \Gamma(TQ)$, then two connections coincide, i.e.,
$$
\pi \nabla_{R_\lambda}^{\text{\rm LC}}|_{\xi} = \pi \nabla_{R_\lambda}|_{\xi}.
$$
See \cite[Section 6]{oh-wang:connection} or more clearly
see its arXiv version 1212.4817(v2) Theorem 1.4 therein where
$\nabla$ is expressed as $\nabla = \nabla^{\text{\rm LC}} + B$
for some $(2,1)$ tensor $B$. From the expression of $B$ there,
we have $B(R_\lambda,Y) = 0$ for any $Y \in \xi$. However, while
we have $\nabla_{R_\lambda}^{\text{\rm LC}} J = 0$,  $\nabla_{R_\lambda}J \neq 0$
on the full tangent bundle $TQ$.
\end{rem}

Roughly speaking
Proposition \ref{prop:blair} enables us to compute the asymptotic operator
$$
A^\pi_{(\lambda,J,\nabla)}: \Gamma(\gamma^*\xi) \to \Gamma(\gamma^*\xi)
$$
in the covariant tensorial way
\emph{uniformly in terms of the pull-back connection of the Levi-Civita connection of
the triad metric} of the given adapted pair $(\lambda,J)$.

\begin{rem}
\begin{enumerate}
\item In the literature, the notion of `asymptotic operator' of a
pseudoholomorphic curve on symplectization has been used. For example,
Hofer-Wysocki-Zehnder use special coordinates followed by some
adjustment of the given almost complex structures along
the Reeb orbit in the analysis of asymptotic operators of finite energy planes, while Siefring \cite{siefring:relative} used a symmetric connection,
Wendl \cite{wendl:lecture}  and Pardon \cite{pardon:contacthomology}
a connection obtained by declaring the Lie derivative $\CL_{R_\lambda}$
to be the covariant differential along the Reeb orbit $\gamma$.
(Compare these practices with those given in \cite{kim-oh:asymp-analysis}
a summary of which is given in Section \ref{sec:asymptotic-analysis}
of the present survey.)
\item  Such an important property \eqref{eq:nablaLCRlambdaJ=0} of the Levi-Civita connection
has not attracted any attention from the symplectic community around
pseudoholomorphic curves, because contact Hamiltonian dynamics
has not attracted much attention from the researchers around pseudoholomorphic curves,
and there has been no serious investigation thereof up to the level of symplectic
Hamiltonian dynamics. As a consequence,
there might not have  been enough motivation for them to re-do the analysis
of pseudoholomorphic curves on symplectization \emph{from scratch starting from its starting place, the contact triad $(Q,\lambda,J)$}, especially when there is already the well-established Gromov's theory of pseudoholomorphic curves around.
\item
However in relation to thermodynamics, contact completely integrable systems
and new constructions of Sasaki-Einstein manifolds which is motivated by AdS/CFT correspondence and black-hole dynamics,
there has been a systematic development of contact Hamiltonian calculus by a group of
geometers and physicists. (See  \cite{BCT}, \cite{dMV} and \cite{lerman} and \cite{martelli-sparks-05,martelli-sparks-06}, \cite{boyer} and references therein).
Through their study, it has been becoming increasingly clearer
that contact Hamiltonian dynamics deserves much more attention than now in many respects.
We believe that the analysis and geometry
of (perturbed) contact instantons will provide a flexible geometro-analytical package
for the study of contact Hamiltonian dynamics and quantitative contact topology as
the symplectic Floer theory does. These are illustrated by
the first author and his collaborators' recent applications of the package
 to the problems of quantitative
contact topology and contact dynamics. (See \cite{oh:entanglement1,oh:shelukhin-conjecture} and
\cite{oh-yso:spectral}.)
\item
The common folklore practice
of lifting contact dynamics to the homogeneous Hamiltonian dynamics
to the symplectization, attempting to exploit the machinery of Floer's theory
on the symplectization and then extracting information on the original contact dynamics
therefrom does not produce optimal results because the lifting process is
not reversible. Because of this, the optimal results proved in
\cite{oh:entanglement1,oh:shelukhin-conjecture} or the Floer theoretic construction of Legendrian
spectral invariants on the one-jet bundle
given in \cite{oh-yso:spectral} have not been obtainable
by the existing machinery of pseudoholomorphic curves via symplectization at least by now.
It is an interesting open problem, if possible, to recover those results
belonging to purely contact realm by the symplectic machinery of pseudoholomorphic curves.
\end{enumerate}
\end{rem}

\subsection{Pseudoholomorphic curves on symplectization}

In the seminal work \cite{hofer:invent}, Hofer initiated the study of
pseudoholomorphic curves on the symplectization $Q \times \R=: M$
for the study of contact topology and developed the analysis thereof
in a series of papers \cite{HWZ:asymptotics}-\cite{HWZ:smallarea} with
applications to contact dynamics in 3 dimensions. Bourgeois \cite{bourgeois}
and Wendl \cite{wendl:lecture} extend their analysis to higher dimensions,
and Cant  to the relative case \cite{cant:thesis} .

Throughout the paper, we adopt the following notations.

\begin{notation} We denote by $(\Sigma,j)$ a closed Riemann surface,
$\dot \Sigma$ the associated punctured Riemann surface and
$\overline \Sigma$ the real blow-up of $\dot \Sigma$ along the
punctures.
\end{notation}

Note that in the presence of contact form $\lambda$, any smooth map
$u:\dot \Sigma \to Q \times \R$ has the form $u = (w,f)$ with
\be\label{eq:a-s}
f = s \circ u, \, w = \pi \circ u.
\ee
We have the splitting
$$
TM \cong \xi \oplus \R \{ R_\lambda \} \oplus \R \left \{
\frac{\del}{\del s} \right \}.
$$
We have a canonical almost complex structure
$$
J_0: \R\left\{\frac{\del}{\del s},R_\lambda\right\} \to
\R \left\{\frac{\del}{\del s}, R_\lambda \right\}
$$
defined by $J_0 \frac{\del}{\del s} = R_\lambda$.
We equip $(Q,\xi)$ with a triad $(Q,\lambda, J)$ and the cylindrical
almost complex structure $\widetilde J = J \oplus J_0$ which
is $s$-translation invariant.
\begin{rem}
As mentioned above, Hofer's analysis of pseudoholomorphic
curves on symplectization should be regarded as a special case
for the analysis of pseudoholomorphic curves on lcs manifolds
$(Q \times \R, d\lambda + ds \wedge \lambda)$ equipped with
the above mentioned almost Hermitian structure arising from
the contact triad $(Q, \lambda, J)$ as illustrated by Savelyev
and the first author in \cite{oh-savelyev}.
See \cite{oh:contacton}, \cite{oh-savelyev}
for further explanation on this
point of view. We would like to reiterate that the symplectic form $d(e^s\lambda)$ itself plays very little role in the compactification
of punctured  pseudoholomorphic curves but Hofer's translational
invariant energy $E(w) = E^\pi(w) + E^\perp(w)$ does.
\end{rem}

We have the decomposition of the derivative
$$
du = dw \oplus \left(df \otimes \frac{\del}{\del s}\right)
$$
viewed as a $TM$-valued one-form which
can be further decomposed to
\be\label{eq:dtildew}
du(z) = d^\pi w \oplus (w^*\lambda \otimes R_\lambda) \oplus
\left(df \otimes \frac{\del}{\del s}\right)
\ee
with respect to the splitting
$$
\operatorname{Hom}(T_z\dot \Sigma, T_{\widetilde w(z)}M)
= \operatorname{Hom}(T_z\dot \Sigma, HT_{\widetilde w(z)}M)
\oplus \operatorname{Hom}(T_z\dot \Sigma, VT_{\widetilde w(z)}M).
$$
(For the notational simplicity, we often omit `$\otimes$' except
the situation that could cause confusion to the readers without it.)

By definition, we have
$
d\pi du = dw.
$
It was observed by Hofer \cite{hofer:invent} that $u$ is $\widetilde J$-holomorphic if and only if $(w,f)$ satisfies
\be\label{eq:tildeJ-holo}
\begin{cases}
\delbar^\pi w = 0 \\
w^*\lambda \circ j= df.
\end{cases}
\ee

\subsection{Definitions of contact instantons and bordered contact instantons}

Let $\dot \Sigma$ a boundary punctured Riemann surface associated a bordered compact Riemann surface
$(\Sigma, j)$. Then for a given map $w: \dot \Sigma \to Q$, we can decompose its derivative
$du$, regarded as a $w^*TQ$-valued one-form on $\dot \Sigma$, into
\be\label{eq:du}
dw = d^\pi w + w^*\lambda \otimes R_\lambda
\ee
 where $d^\pi w := \pi dw$. Furthermore $d^\pi w$ is decomposed into
\be\label{eq:dpiu}
d^\pi w = \delbar^\pi w + \del^\pi w
\ee
where $\delbar^\pi w: = (dw^\pi)_J^{(0,1)}$ (resp. $\del^\pi w: = (dw^\pi)_J^{(1,0)}$) is
the anti-complex linear part (resp. the complex linear part) of $d^\pi w: (T\dot \Sigma, j) \to (\xi,J|_\xi)$.
(For the simplicity of notation, we will abuse our notation by often denoting $J|_\xi$ by $J$.
We also simply write $((\cdot)^\pi)_J^{(0,1)} = (\cdot)^{\pi(0,1)}$ and
 $((\cdot)^\pi)_J^{(1,0)} = (\cdot)^{\pi(1,0)}$ in general.)

A contact instanton is a map $w: \dot \Sigma \to M$ that satisfies
the system of nonlinear partial differential equation
\be\label{eq:contacton-intro}
\delbar^\pi w = 0, \quad d(w^*\lambda \circ j) = 0
\ee
on a contact triad $(M,\lambda, J)$. The equation itself had been first
mentioned by Hofer \cite[p.698]{hofer:gafa}, and some attempt
to utilize the equation to attack the Weinstein's conjecture for dimension 3
was made by Abbas \cite{abbas}, Abbas-Cieliebak-Hofer \cite{abbas-cieliebak-hofer} as well as by Bergmann
\cite{bergmann1,bergmann2}.

In a series of papers,
\cite{oh-wang:CR-map1,oh-wang:CR-map2} jointed with Wang and in \cite{oh:contacton}, the first named author systematically
developed analysis of contact instantons
(for the closed string case)
\emph{without taking symplectization} by the global covariant tensorial
calculations using the
notion of \emph{contact triad connection} which was introduced in \cite{oh-wang:connection}

More recently he also studied its open string counterpart of
the boundary value problem of \eqref{eq:contacton-intro} under the Legendrian boundary condition
whose explanation is now in order.  For the simplicity and
for the main purpose of the present paper, we focus on the genus zero case so that $\dot \Sigma$ is conformally the unit disc with boundary
punctures $z_0, \ldots, z_k \in \del D^2$ ordered
counterclockwise, i.e.,
$$
\dot \Sigma \cong D^2 \setminus \{z_0, \ldots, z_k\}
$$
Then, for a  $(k+1)$-tuple $\vec R = (R_0,R_1, \cdots, R_k)$ of Legendrian submanifolds, which we call
an (ordered) Legendrian link, we consider
the boundary value problem
\be\label{eq:contacton-Legendrian-bdy-intro}
\begin{cases}
\delbar^\pi w = 0, \quad d(w^*\lambda \circ j) = 0,\\
w(\overline{z_iz_{i+1}}) \subset R_i
\end{cases}
\ee
as an elliptic boundary value problem for a map $w: \dot \Sigma \to M$
by deriving the a priori coercive elliptic estimates. Here $\overline{z_iz_{i+1}} \subset \del D^2$
is the open arc between $z_i$ and $z_{i+1}$.

\subsection{$W^{2,2}$-estimates, Weitzenb\"ock formulae and contact triad connection}

Let us start with stating the general Weitzenb\"ock formula
in differential geometry. Let $(P,h)$ be a Riemannian manifold
and $E \to P$ is a
Euclidean vector bundle with inner product $\langle \cdot, \cdot \rangle$
and assume $\nabla$ is a connection compatible with $\langle \cdot,
\cdot \rangle$. ( A good exposition of general Weitzenb\"ock
formula is given in \cite[Appendix C]{freed-uhlen}.)

It is well established in the analysis of geometric PDE of the types, harmonic maps,
minimal surface equation and Yang-Mills equations and so on that all a priori
elliptic regularity results are based on suitable applications of the following general Weitzenb\"ock Formula one way or the other.
(See  \cite{schoen-yau:harmonicmap}, \cite{sacks-uhlen}, \cite{uhlen:removal},
\cite{schoen-uhlen:harmonicmap},  \cite{schoen},
\cite{parker-wolfson}, \cite{ruan-tian}, to name a few.)

\begin{thm}[Weitzenb\"ock Formula]\label{thm:weitzenbock-intro} Let $E \to P$ be a vector bundle equipped with
inner product $\langle \cdot,\cdot \rangle$ and an Euclidean connection $\nabla$.
Assume $\{e_i\}$ is an orthonormal frame of $P$, and $\{\alpha^i\}$ is the dual frame. Then
when applied to $E$-valued differential forms, we have
\beastar
\Delta^{\nabla}& =& -Tr\nabla^2+\sum_{i,j}\alpha^j\wedge (e_i\rfloor R(e_i,e_j)(\cdot))\\
&=& - \nabla^*\nabla+\sum_{i,j}\alpha^j\wedge (e_i\rfloor R(e_i,e_j)(\cdot)).
\eeastar
Here we denote $Tr\nabla^2=\sum_i \nabla^2_{e_i,e_i}= \nabla^*\nabla $ the trace Laplacian
and $R$ is the curvature tensor of the bundle $E$ with respect to the connection $\nabla$.
\end{thm}
For each given equation, to get the optimal regularity estimates, it is important to use the `best'
connection compatible with the given geometry such as the Chern connection or the canonical connection in the harmonic theory
of holomorphic vector bundles in complex geometry \cite{chern:connection}.
(See \cite{wells} for a nice exposition on the harmonic theory of holomorphic vector bundles on complex manifolds.
One may view that our calculations are largely the almost complex counterpart thereof on the Riemann surface $\dot \Sigma$.)

We specialize this general Weitzenb\"ock formula to our purpose of tensorial study of
contact instantons.  For each given contact triad $(M,\lambda,J)$, we consider
the vector bundles $E$ such as
$$
w^*TM, \quad w^*\xi, \quad \Lambda^1(w^*TM), \quad \Lambda^{(0,1)}(w^*TM)
$$
on punctured Riemann surface $(\dot \Sigma, j, h)$ equipped with K\"ahler metric $h$
that is cylindrical near each puncture. As the aforementioned `best' connection in this setting of
contact triads, Wang and the first author introduced the notion of
\emph{contact triad connection} which is unique for each given contact triad $(Q,\lambda,J)$.
(See \cite{oh-wang:connection} for its construction and full
properties. See also Section \ref{sec:connection} for a summary.)

In the geometric PDEs of minimal surface-type map $w: S \to M$,
such as pseudoholomorphic curves or contact instantons,
the regularity estimates usually starts with computing the formula for the Laplacian of the harmonic energy density function
$$
\Delta |dw|^2.
$$
A priori, this quantity involves degree 3 derivatives for general smooth maps (\emph{in the off-shell}),
 but which is hoped to be expressed as a sum of the
terms of degree less than 3 for maps satisfying the equation (\emph{on shell}).
It will then enable one to develop a priori boot-strap arguments.
We always regard $dw$ as a $w^*TM$-valued one-form on the domain $\dot \Sigma$.

Following this general practice and
using the decomposition \eqref{eq:du} and the properties $\xi \perp R_\lambda$ and $|R_\lambda| =1$ for the contact triad metric,
we  can decompose the (full) harmonic energy density into the sum
\be\label{eq:|dw|2}
|dw|^2 = |d^\pi w|^2 + |w^*\lambda|^2
\ee
where $|d^\pi w|^2$ is the contribution from the $\xi$-direction and $|w^*\lambda|^2$ is the energy density
in the Reeb direction: Recall the decomposition
$dw = d^\pi w + w^*\lambda \, \R_{\lambda}$
is orthogonal with respect to the induced metric from the triad metric of the target and
the K\"ahler metric $h$ of the domain Riemann surface $(\dot \Sigma, j)$.

The following differential identity for contact Cauchy-Riemann map plays a  fundamental
role for all the estimates needed for the contact instantons.
\begin{thm}[Fundamental Equation; Theorem 4.2 \cite{oh-wang:CR-map1}]\label{thm:fundamental-intro}
Let $w$ be any contact Cauchy--Riemann map,
 i.e., a solution of $\delbar^\pi w=0$. Then
\be\label{eq:Laplacian-w-intro}
d^{\nabla^\pi} (d^\pi w) = -w^*\lambda\circ j \wedge\left(\frac{1}{2} (\CL_{R_\lambda}J)\, d^\pi w\right).
\ee
\end{thm}

\begin{rem} When we apply similar calculation to a $J$-holomorphic map $u$ with respect to the canonical connection
on the almost Hermitian manifold $(M,\omega,J)$, then the corresponding equation is
the simple harmonic map equation $d^\nabla(du) = 0$ for $u$. (See \cite[Corollary 7.3.3]{oh:book1}.)
Together with the conformality of any $J$-holomorphic map, this provides a computational proof with
the well-known fact that the image of a $J$-holomorphic curve is a minimal surface with respect to the
compatible metric. This equation was the basis for the $W^{2,2}$-estimate for $J$-holomorphic map
equation $\delbar_Ju = 0$ on symplectic manifolds. (See \cite[Proposition 7.4.5]{oh:book1}.)
\end{rem}

Then we can derive the following differential identity for the
$\xi$-component $d^\pi u$ of the derivative $dw$,
utilizing the Weitzenb\"ock formula associated to the contact triad connection.

\begin{prop}[Equation (4.11) \cite{oh-wang:CR-map1}]
\label{prop:e-pi-weitzenbock-intro}
Let $w$ be a contact Cauchy-Riemann map. Then
\bea\label{eq:e-pi-weitzenbock-intro}
-\frac{1}{2}\Delta e^\pi(w)&=&|\nabla^\pi (d^\pi w)|^2+K|d^\pi w|^2+\langle \ric^{\nabla^\pi} (d^\pi w), d^\pi w\rangle\nonumber\\
&{}&+\langle \delta^{\nabla^\pi}[(w^*\lambda\circ j)\wedge (\CL_{R_\lambda}J)d^\pi w], d^\pi w\rangle.
\eea
\end{prop}

If $w$ is a contact instanton, i.e., if it satisfies $d(w^*\lambda \circ j) = 0$ in addition,
we also have the following identity for the Laplacian of the energy density along the Reeb direction.
\begin{prop}[Equation (5.4) \cite{oh-wang:CR-map1}; Proposition \ref{prop:Delta|w*lambda|2}]
\label{prop:-1/2Delta-w*lambda-intro}
Let $w$ be a contact instanton. Then
\be\label{eq:Delta|w*lamba|2-intro}
-\frac{1}{2}\Delta|w^*\lambda|^2=|\nabla w^*\lambda|^2+K|w^*\lambda|^2
+ \langle *\langle \nabla^\pi d^\pi w, d^\pi w\rangle,  w^*\lambda\rangle.
\ee
\end{prop}

By adding the two identities, we obtain a formula for the Laplacian $\Delta |dw|^2$
of the full harmonic energy density function $|dw|^2$ on shell.
Once these are established, the following local $W^{2,2}$-estimate
is obtained by a standard trick of multiplying cut-off function and
integrating by parts.
\begin{thm}[Theorem 1.6 \cite{oh-wang:CR-map1}]\label{thm:local-W12-intro}
Let $w: \dot \Sigma \to M$ satisfy \eqref{eq:contacton-Legendrian-bdy-intro}.
Then for any relatively compact domains $D_1$ and $D_2$ in
$\dot\Sigma$ such that $\overline{D_1}\subset D_2$, we have
$$
\|dw\|^2_{W^{1,2}(D_1)}\leq  C_4 \|dw\|^4_{L^4(D_2)}
$$
where $C_4$ is a constant depending only on $D_1$, $D_2$ and $(M,\lambda, J)$ and $R_i$'s.
\end{thm}

The same estimates is also proved in the similar spirit
by incorporating the Legendrian boundary condition which is a (nonlinear) elliptic
boundary value problem. (See \cite[Theorem 1.4]{oh:contacton-Legendrian-bdy} for the statement
and \cite{oh-yso:index} for the same statement with corrected proof.)
This boundary estimate is rather nontrivial unlike the closed string case.

As the first step towards the analytic study of the
above boundary value problem \eqref{eq:contacton-Legendrian-bdy-intro},
we first show that the Legendrian boundary condition for the contact
instanton is a free boundary value problem, i.e., it satisfies
$$
\frac{\del w}{\del \nu} \perp TR
$$
for any Legendrian submanifold. (See \cite{jost:free-bdy} for
the importance of the free boundary value condition for a general study of
elliptic estimates of the minimal surface type equations.) Then we
prove the elliptic $W^{2,2}$-estimate as an application of
Stokes' formula combined with the Legendrian boundary condition. The global tensorial calculation deriving
the a priori estimate in \cite{oh-yso:index} illustrates how well the Legendrian boundary condition
interacts with triad connection and the contact instanton equation.

\subsubsection{Higher $C^{k,\delta}$ H\"older estimates}

Once this $W^{2,2}$ estimates is established,
we proceed with the higher boundary regularity estimates.
Obviously the same estimates hold for the closed string case \eqref{eq:contacton-intro}
the corresponding statement of which had been established in \cite{oh-wang:CR-map1}.
Since this case is easier, we focus on the statements for the boundary estimates below.

Starting from Theorem \ref{thm:local-W12-intro} and using
the embedding $W^{2,2} \hookrightarrow C^{0,\delta}$ with $0 < \delta < 1/2$,
we also establish the following higher local $C^{k,\delta}$-estimates on punctured surfaces
$\dot \Sigma$ in terms of the $W^{2,2}$-norms with $\ell \leq k+1$.

\begin{thm}[Theorem 1.4 \cite{oh-yso:index}]\label{thm:local-regularity-intro} Let $w$ satisfy \eqref{eq:contacton-Legendrian-bdy-intro}.
Then for any pair of disk $D_1 \subset D_2 \subset \dot \Sigma$
of semi-disk domains $(D_1,\del D_1)\subset (D_2,\del D_2) \subset (\Sigma,\del \Sigma)$
such that $\overline{D_1}\subset D_2$,
$$
\|dw\|_{C^{k,\delta}} \leq C_\delta(\|dw\|_{W^{1,2}(D_2)})
$$
where $C_\delta = C_\delta(r) > 0$ is a function continuous at $r = 0$
and depends only on $J$, $\lambda$ and $D_1, \, D_2$ but independent of $w$.
\end{thm}

With some adjustment of the function $C_\delta$, combining the two theorems, we obtain
\begin{cor} Let $w$ satisfy \eqref{eq:contacton-Legendrian-bdy-intro}.
Then for any pair of domains $D_1 \subset D_2 \subset \dot \Sigma$ such that $\overline{D_1}\subset D_2$,
$$
\|dw\|_{C^{k,\delta}} \leq C_\delta(\|dw\|_{L^4(D_2)})
$$
where $C_\delta = C_\delta(r) > 0$ is a function continuous at $r = 0$
and depends only on $J$, $\lambda$ and $D_1, \, D_2$ but independent of $w$.
\end{cor}

In particular, we prove that
any weak solution of \eqref{eq:contacton-Legendrian-bdy-intro} in
$W^{1,4}_{\text{\rm loc}}$ automatically becomes a classical solution.
(Compare \cite[Theorem 8.5.5]{oh:book1} for a similar theorem for the
Lagrangian boundary condition in symplectic geometry.)

The proof of Theorem \ref{thm:local-regularity-intro} is carried out
by an \emph{alternating boot strap argument} by decomposing
$$
dw = d^\pi w + w^*\lambda\, R_\lambda
$$
as follows. Let $z = x+i y$ be any isothermal coordinates on $(D_2,\del D_2) \subset (\dot \Sigma,\del \dot \Sigma)$
adapted to the boundary, i.e., satisfying that $\frac{\del}{\del x}$ is tangent to $\del \dot \Sigma$.
We set
\beastar
\zeta & : = & d^\pi w(\del_x), \\
 \alpha &: = & \lambda\left(\frac{\del w}{\del y}\right)
+ \sqrt{-1} \lambda\left(\frac{\del w}{\del x}\right)
\eeastar
Then we show that the fundamental equation \eqref{eq:Laplacian-w-intro} is transformed into
the following system of equations for the pair $(\zeta,\alpha)$
\be\label{eq:equation-for-zeta0-intro}
\begin{cases}\nabla_x^\pi \zeta + J \nabla_y^\pi \zeta
+ \frac{1}{2} \lambda(\frac{\del w}{\del y})(\CL_{R_\lambda}J)\zeta - \frac{1}{2} \lambda(\frac{\del w}{\del x})(\CL_{R_\lambda}J)J\zeta =0\\
\zeta(z) \in TR_i \quad \text{for } \, z \in \del D_2
\end{cases}
\ee
and
\be\label{eq:equation-for-alpha-intro}
\begin{cases}
\delbar \alpha = \frac{1}{2}|\zeta|^2 \\
\alpha(z) \in \R \quad \text{for } \, z \in \del D_2
\end{cases}
\ee
for some $i = 0, \ldots, k$. With this coupled system of equations for $(\zeta,\alpha)$
 at our disposal, the proof of higher regularity results is carried out
by the alternating boot strap argument between $\zeta$ and $\alpha$ in \cite{oh:contacton-Legendrian-bdy,oh-yso:index}.

\subsection{Asymptotic convergence and vanishing of asymptotic charge}

Next we study the asymptotic convergence result of contact instantons
of finite energy $E(w) = E^\pi(w) + E^\perp(w) < \infty$
for the closed string case (resp. with Legendrian boundary
condition of pair $(R_0,R_1)$ for the open string case)
near the punctures of a Riemann surface $\dot \Sigma$.
(We refer to \cite{oh:contacton,oh:entanglement1,oh-yso:spectral}
for the precise definition of total energy.)

Let $\dot\Sigma$ be a punctured Riemann surface with punctures
$$
\{p^+_i\}_{i=1, \cdots, l^+}\cup \{p^-_j\}_{j=1, \cdots, l^-}
$$
equipped with a metric $h$ with cylinder-like ends
(resp. \emph{strip-like ends} for the open string case)
outside a compact subset $K_\Sigma$.
Let $w: \dot \Sigma \to M$ be any such smooth map.

As in \cite{oh-wang:CR-map1}, we define the total $\pi$-harmonic energy $E^\pi(w)$ is easy to define as
\be\label{eq:pienergy-w}
E^\pi(w) = E^\pi_{(\lambda,J;\dot\Sigma,h)}(w) = \frac{1}{2} \int_{\dot \Sigma} |d^\pi w|^2\, dA
\ee
where $dA$ is the associated area form and the norm is taken in terms of the given metric $h$ on $\dot \Sigma$ and the triad metric on $M$.

\subsubsection{The case of closed strings}

 Under the hypotheses of nondegeneracy $\lambda$ (resp. of the pair
 $(\lambda,(R_0,R_1)$ for the open string case)
 and of asymptotic convergence at the punctures, we can associate two
natural asymptotic invariants at each puncture defined as
\bea
T & := & \lim_{r \to \infty} \int_{\{r\}\times S^1}(w|_{\{r\}\times S^1})^*\lambda
\label{eq:TQ-T-intro}\\
Q & : = & \lim_{r \to \infty} \int_{\{r\}\times S^1}((w|_{\{r\}\times S^1 })^*\lambda\circ j)\label{eq:TQ-Q-intro}
\eea
at each puncture.
(Here we only look at positive punctures. The case of negative punctures is similar.)
As in \cite{oh-wang:CR-map1}, we call $T$ the \emph{asymptotic contact action}
and $Q$ the \emph{asymptotic contact charge} of the contact instanton $w$ at the given puncture.

The proof of the following subsequence convergence result
is given in \cite[Theorem 6.4]{oh-wang:CR-map1}.
A similar asymptotic convergence result
 in 3 dimension in the setting of pseudoholomorphic
curves on symplectization, i.e., the case of $Q = 0$ is proved in
\cite{HWZ:smallarea}.
(See also \cite{HWZ:asymptotics,HWZ:asymptotics-correction} and
compare their proofs with the proof given in \cite{oh-wang:CR-map1}.)

\begin{thm}[Subsequence Convergence,
Theorem 6.4 \cite{oh-wang:CR-map1}]
\label{thm:subsequence-intro}
Let $w:[0, \infty)\times S^1 \to M$ satisfy the contact instanton equations \eqref{eq:contacton-Legendrian-bdy-intro}
and Hypothesis \eqref{eq:hypo-basic-pt}.
Then for any sequence $s_k\to \infty$, there exists a subsequence, still denoted by $s_k$, and a
massless instanton $w_\infty(\tau,t)$ (i.e., $E^\pi(w_\infty) = 0$)
on the cylinder $\R \times [0,1]$  that satisfies the following:
\begin{enumerate}
\item $\delbar^\pi w_\infty = 0$ and
$$
\lim_{k\to \infty}w(s_k + \tau, t) = w_\infty(\tau,t)
$$
in the $C^l(K \times [0,1], M)$ sense for any $l$, where $K\subset [0,\infty)$ is an arbitrary compact set.
\item $w_\infty^*\lambda = -Q\, d\tau + T\, dt$
\end{enumerate}
\end{thm}

In general $Q  = 0$ does not necessarily hold for the closed string case.
When $Q \neq 0$ combined with $T = 0$ happens, we say $w$ has
the bad limit of \emph{appearance of spiraling instantons along the Reeb core}. It is also proven in \cite{oh:contacton} that If $Q = 0 = T$,
then the puncture is removable.

When $Q = 0$, which is always the case when contact instanton
is exact such as those arising from the symplectization case,
we have the following asymptotic convergence result.
\begin{cor} Assume that $\lambda$ is nondegenerate. Let
$w$ be as above and assume $Q = 0, \, T \neq 0$ and that
$w_\tau: S^1 \to Q$ converges as $|\tau| \to \infty$.
Then $w_\tau$ converges to a Reeb orbit of period $|T|$
exponentially fast.
\end{cor}

\subsubsection{The case of open strings}

Now we make the corresponding statement for
the open string case proved in \cite{oh:contacton-Legendrian-bdy}.

\begin{thm}[Subsequence Convergence; the case of open strings]\label{thm:subsequence-open-intro}
Let $w:[0, \infty)\times [0,1]\to M$ satisfy the contact instanton equations \eqref{eq:contacton-Legendrian-bdy-intro}.
Then for any sequence $s_k\to \infty$, there exists a subsequence, still denoted by $s_k$, and a
massless instanton $w_\infty(\tau,t)$ (i.e., $E^\pi(w_\infty) = 0$)
on the cylinder $\R \times [0,1]$  such that
$$
\lim_{k\to \infty}w(s_k + \tau, t) = w_\infty(\tau,t)
$$
in the $C^l(K \times [0,1], M)$ sense for any $l$, where $K\subset [0,\infty)$ is an arbitrary compact set.
Furthermore, $w_\infty$ has $Q = 0$ and the formula $w_\infty(\tau,t)= \gamma(T\, t)$  with
asymptotic action $T$, where $\gamma$ is some Reeb chord joining $R_0$ and $R_1$ of period $|T|$.
\end{thm}

\begin{cor}[Vanishing Charge] Assume the pair $(\lambda, \vec R)$ is nondegenerate.
Let $w$ be as above with finite energy. Suppose that $w(\tau,\cdot)$ converges as $\tau \to \infty$ in the
strip-like coordinate at a puncture $p \in \del \Sigma$  with associated Legendrian pair $(R,R')$.
 Then its  asymptotic charge $Q$ vanishes at $p$.
 \end{cor}

\subsection{Asymptotic operators and their analysis}
\label{subsec:admissiblepair}

We first mention a few differences between the way how we study the asymptotic operators and
those of \cite{HWZ:asymptotics} and of other literature such as \cite[Appendix C]{robbin-salamon:asymptotic}, \cite{siefring:relative,siefring:intersection},
\cite{wendl:lecture}, \cite{cant:thesis}.

\begin{rem}
In \cite[Appendix E]{robbin-salamon:asymptotic},
\cite{siefring:relative,siefring:intersection}, \cite[Section 3.3]{wendl:lecture}, there have been attempts to give a coordinate-free definition
of the asymptotic operator along the
associated asymptotic Reeb orbit for a pseudoholomorphic
curve $u = (w,f)$ on symplectization.
However both fall short of a seamless definition of the `asymptotic operator' of the Reeb orbits
because the Reeb orbit lives on $Q$ while the pseudoholomorphic curves live on the product $Q \times \R$
and the asymptotic limit of pseudoholomorphic curve live at infinity $Q \times \{\pm \infty\}$ \emph{where only the contact structure
makes sense, i.e., is canonically defined, but not the contact form itself}.
Cant studies the asymptotic operator for the relative context in \cite[Section 6.3]{cant:thesis}
by adapting Wendl's. What these literature (e.g.\cite[Section 3.3]{wendl:lecture}) are describing
is actually the asymptotic operator of the contact instanton $w$ but trying to describe it in terms of
the pseudoholomorphic curves which prevents them from being able to give
a seamless definition. (See our definition of the asymptotic operator of contact instantons given
in Definition \ref{defn:asymptotic-operator} and compare it therewith. See also \cite[Section 11.2 \& 11.5]{oh-wang:CR-map2}
for the precursor of our definition.)
\end{rem}

In their series of works \cite{HWZ:embedding-control} --
\cite{HWZ:smallarea} in 3 dimension,
Hofer-Wysocki-Zehnder carried out fundamental
analytic study of  pseudoholomorphic curves  on symplectization
utilizing \emph{special coordinates followed by some local adjustment of given
almost complex structure along the Reeb orbit of interest.}
This practice has been propagated to other literature, such as \cite{bourgeois}, \cite[Section 1.3]{hutchings:index-inequality}
giving rise to some unnecessary restrictions on the choice of
almost complex structure beyond the natural $\lambda$-adaptedness.
\footnote{Hutchings has informed the first author that this part of his paper
is now obsolete in that Siefring proved the asymptotics he needed without the extra assumptions in \cite{siefring:relative} quoted above.
Siefring's paper is  now quoted for this in Hutching's  later paper
\cite{hutchings:index-revisit} which is a kind of update to
\cite{hutchings:index-inequality}.}

Largely, thanks to  the property
$\nabla_{R_\lambda}^{\text{\rm LC}}J = 0$ from Proposition \ref{prop:blair} combined with
our usage of contact triad connection in the derivation of the formula for the asymptotic operator,
 our asymptotic analysis provided in \cite{kim-oh:asymp-analysis} does not
need any of those special coordinates and so Hutchings' assumption (or similar ansatz in other literature)
is really not needed.

\begin{thm}[Corollary \ref{cor:A-in-LC}]\label{thm:A-in-LC} Let $(\lambda, J)$ be any adapted pair
and let $\nabla^{\text{\rm LC}}$ be the Levi-Civita connection of the
triad metric of $(Q,\lambda,J)$. For given contact instanton $w$ with its action $\int \gamma^*\lambda =T$
at a puncture, let $A_{(\lambda, J, \nabla)}$ be the asymptotic operator of $w$ written in cylindrical coordinate $(\tau,t)$. Then
\begin{enumerate}
\item $[\nabla_t^{\text{\rm LC}},J] (= \nabla^{\text{\rm LC}}_t J) = 0$,
\item We have
\beastar
A^\pi_{(\lambda,J,\nabla)} & = & - J\nabla_t + \frac{T}{2}\CL_{R_\lambda}JJ \\
& = & - J\nabla_t^{\text{\rm LC}} - \frac{T}{2} Id +  \frac{T}{2}\CL_{R_\lambda}JJ.
\eeastar
\end{enumerate}
\end{thm}
We also have the following formula of $A^\pi_{(\lambda,J,\nabla)}$ that is 
independent of the choice of connections.

\begin{prop}[Proposition 1.7 \cite{kim-oh:asymp-analysis}]
\label{prop:A=T-intro} We have
\be\label{eq:A-in-CL2-intro}
A^\pi_{(\lambda,J,\nabla)}
= T\, \left(-\frac12 \CL_{R_\lambda}J - \CL_{R_\lambda}
+ \frac{1}{2}(\CL_{R_\lambda}J) J\right).
\ee
\end{prop}
Because the existing literature on the pseudoholomorphic curves on
symplectization lack  this kind of explicit formula of the asymptotic operator
(together with the commuting property $[\nabla_t^{\text{\rm LC}},J] = 0$),
it has been the case that the general abstract
perturbation theory of linear operators Kato \cite{kato} is just
quoted in their study of asymptotic operators  which prevents one
from making any statement on specific dependence on
the adapted almost complex structures.
(See \cite[Lemma 3.17 \& Theorem 3.35]{wendl:lecture},
for example, which  in turn follows the statements and arguments
given in \cite{HWZ:embedding-control}.)

\begin{rem} \begin{enumerate}
\item The formula of an asymptotic operator canonically applied to
every contact instanton $w$  of the adapted pair $(\lambda, J)$
(as well as that of a pseudoholomorphic curve
$u = (w,f)$ on symplectization) depends not only on the adapted pair
$(\lambda, J)$ but also on the connection that is used to compute
the linearization operator of $w$ (or $(u,f)$).
(See Definition \ref{defn:asymptotic-operator} and Remark
\ref{rem:asymptotic-operator} for the explanation for why.) In this regard, we denote
the asymptotic operator by
$$
A^\pi_{(\lambda,J, \nabla)}(u) := A^\pi_{(\lambda,J, \nabla)}(w)
$$
So it is conceivable to expect that a good choice of connection will
give rise to a formula of the asymptotic operator that is easier to analyze.
\item  In fact, the lack of precise definition together with
the non-commuting property $[\nabla_t,J] \neq  0$ in the literature
combined with the practice of
\emph{using special coordinates followed by adjusting the almost
complex structure along the Reeb orbits}
is bound to make the analysis of asymptotic operators very complicated
as seen from \cite{HWZ:smallarea}, \cite{siefring:relative},
 and prevents one from developing any perturbation
theory of asymptotic operators under the perturbation of $J$'s
such as those developed by the present authors in \cite{kim-oh:asymp-analysis}.
The latter is summarized in Section \ref{sec:asymptotic-analysis} of the present survey.
\end{enumerate}
\end{rem}

On the other hand, our explicit formula of the asymptotic operator given in Theorem
\ref{thm:A-in-LC}, which simultaneously applies to all closed Reeb orbits, enables us to prove
the following natural generic perturbation result \cite{kim-oh:asymp-analysis}.

\begin{thm}[Generic simpleness of eigenvalues; \cite{kim-oh:asymp-analysis}]
\label{thm:eigenvalue-intro}  Let $(Q,\xi)$ be a contact manifold.
Assume that $\lambda$ is
nondegenerate. For a generic choice of $\lambda$-adapted CR almost complex
structures $J$, all eigenvalues $\mu_i$ of the asymptotic
operator are simple for all
closed Reeb orbits of $\lambda$.
\end{thm}

See Section \ref{sec:asymptotic-analysis} for our derivation of the formula of
the asymptotic operator and a summary of our analysis in \cite{kim-oh:asymp-analysis}
of the asymptotic operators. We believe that this kind of perturbation result
will play some role in the construction of Kuranishi structures
on the moduli space of finite energy contact instantons so that certain
natural functor can be defined in our Fukaya-type category of contact manifolds
\cite{kim-oh:contacton-kuranishi}, \cite{oh:entanglement2} (See Remark \ref{rem:compatibility} (2) for
the relevant remark.)

\subsection{Comparison of compactifications of two moduli spaces}

(The materials in this subsection is borrowed from
\cite{kim-oh:contacton-kuranishi}, \cite{oh:entanglement2}
which are in preparation.)

Now let us consider contact instantons $w$ arising
from a pseudoholomorphic curves on
symplectization $(w,f)$. In particular all such $w$ has its the
\emph{charge class} vanishes $[w^*\lambda \circ j] = 0$
in $H^1(\dot \Sigma, \Z)$. (See \cite{oh-savelyev} or Subsection
\ref{subsec:prescribed-charge} for its definition.)

\begin{rem}
For the boundary punctured case,
the charge class can be lifted to
$H^1(\overline\Sigma, \del_\infty \dot \Sigma)$ where $\overline\Sigma$
is the real blow-up of $\dot \Sigma$ along the punctures.
See  Subsection \ref{subsec:prescribed-charge} for its definition.
\end{rem}

Let $\Sigma$ be a closed Riemann surface of genus $g$ and $\dot \Sigma$
be the associated punctured Riemann surface
$\dot \Sigma = \Sigma \setminus \{z_1,\cdots, z_\ell\}$.
We denote the moduli space of such contact instantons
$w: \dot \Sigma \to Q$ of finite energy by
$$
\widetilde \CM^{\text{\rm exact}}_{g,\ell}(Q,J;\vec \gamma^-, \vec \gamma^+)
$$
and
$$
\CM^{\text{\rm exact}}_{g,k,\ell}(Q,J;\vec \gamma^-, \vec \gamma^+)
: = \widetilde \CM^{\text{\rm exact}}_{g,\ell}( Q,J;\vec \gamma^-, \vec \gamma^+)/\text{\rm Aut}(\dot \Sigma)
$$
the set of isomorphism classes thereof. We have the
natural forgetful map $(w,f) \mapsto w$ which descends to
$$
\mathfrak{forget}: \CM_{g,\ell}( M,\widetilde J;\vec \gamma^-, \vec \gamma^+)
\to \CM^{\text{\rm exact}}_{g,\ell}( Q,J;\vec \gamma^-, \vec \gamma^+).
$$
By definition of the equivalence relation on
$\widetilde \CM^{\text{\rm exact}}_{g,\ell}( M,\widetilde J;\vec \gamma^-, \vec \gamma^+)$ defined in \cite{EGH},
\cite{behwz}, it follows that this forgetful map is a
bijective correspondence, \emph{provided $\dot \Sigma$ is connected}.

However when one considers the SFT compactification, one needs to also
consider the case of pseudoholomorphic curves with disconnected domains.
So let us consider such cases. Suppose that the Riemann surface $\dot \Sigma$ is the union
$$
\dot \Sigma = \bigsqcup_{i=1}^k \dot \Sigma_i
$$
of connected components with $k \geq 2$. We denote by
$$
\overline \CM_{g,\ell}(M,\widetilde J;\vec \gamma^-, \vec \gamma^+)
$$
and
$$
\overline \CM^{\text{\rm exact}}_{g,\ell}( Q,J;\vec \gamma^-, \vec \gamma^+)
$$
the stable map compactification respectively. The following proposition shows
the precise relationship between the two. We know that each story carries
at least one non-cylindrical component.

\begin{prop}[Proposition \ref{prop:Rk-1fibration}]
Let $1 \leq \ell \leq k$ be the number of connected components which are
not cylinderical. The forgetful map $\mathfrak{forget}$ is a
principle $\R^{\ell-1}$ fibration.
\end{prop}
So the above mentioned bijective correspondence still holds when there is
exactly one non-trivial component. (See \cite[p. 835]{behwz} for
a relevant remark.)

This proposition clearly shows that the compactification proposed in \cite{EGH}, \cite{behwz} have spurious strata that contract  in the compactification of (exact) contact instantons under the forgetful map.
In this regard, our compactification
of exact contact instanton moduli spaces is closely related to Pardon's
compactification  of moduli spaces of pseudoholomorphic curves on
symplectization \cite{pardon:contacthomology},
which is slightly different from that of \cite{EGH},
\cite{behwz}. We will make precise their relationship in
\cite{kim-oh:contacton-kuranishi} and \cite{oh:entanglement2}.
Pardon \cite{pardon:contacthomology} used his compactification
 for his construction of Kuranishi structure on the moduli space of
pseudoholomorphic curves on symplectization to define contact homology.

\subsection{Fredholm theory and the index formula}
\label{subsec:index-formula}

Next, we study another crucial component, the relevant Fredholm theory and the index formula
for the equation \eqref{eq:contacton-Legendrian-bdy-intro}
by adapting the one from \cite{oh:contacton}, \cite{oh-savelyev}
to the current case of contact instantons \emph{with boundary}. The relevant Fredholm theory in general
has been developed  by the first named author in \cite{oh:contacton} for the closed string case and
\cite{oh:contacton-transversality} for the case with boundary.

In the present paper, we state the index formula for general disk instanton $w$
with finite number of boundary punctures.

We recall  the following Fredholm property of the linearized operator that is proved in
\cite{oh:contacton-transversality}.

\begin{prop}[Proposition 3.18  \& 3.20 \cite{oh:contacton-transversality}] \label{prop:open-fredholm}Suppose that $w$ is a solution to \eqref{eq:contacton-Legendrian-bdy-intro}.
Consider the completion of $D\Upsilon(w)$,
which we still denote by $D\Upsilon(w)$, as a bounded linear map
from $\Omega^0_{k,p}(w^*TM,(\del w)^*T\vec R)$ to
$\Omega^{(0,1)}(w^*\xi)\oplus \Omega^2(\Sigma)$
for $k \geq 2$ and $p \geq 2$. Then
\begin{enumerate}
\item The off-diagonal terms of $D\Upsilon(w)$ are relatively compact operators
against the diagonal operator.
\item
The operator $D\Upsilon(w)$ is homotopic to the operator
\be\label{eq:diagonal}
\left(\begin{matrix}\delbar^{\nabla^\pi} + T_{dw}^{\pi,(0,1)}+ B^{(0,1)} & 0 \\
0 & -\Delta(\lambda(\cdot)) \,dA
\end{matrix}
\right)
\ee
via the homotopy
\be\label{eq:s-homotopy}
s \in [0,1] \mapsto \left(\begin{matrix}\delbar^{\nabla^\pi} + T_{dw}^{\pi,(0,1)} + B^{(0,1)}
& \frac{s}{2} \lambda(\cdot) (\CL_{R_\lambda}J)J (\pi dw)^{(1,0)} \\
s\, d\left((\cdot) \rfloor d\lambda) \circ j\right) & -\Delta(\lambda(\cdot)) \,dA
\end{matrix}
\right) =: L_s
\ee
which is a continuous family of Fredholm operators.
\item And the principal symbol
$$
\sigma(z,\eta): w^*TM|_z \to w^*\xi|_z \oplus \Lambda^2(T_z\Sigma), \quad 0 \neq \eta \in T^*_z\Sigma
$$
of \eqref{eq:diagonal} is given by the matrix
\beastar
\left(\begin{matrix} \frac{\eta + i\eta \circ j}{2} Id  & 0 \\
0 & |\eta|^2
\end{matrix}\right).
\eeastar
\end{enumerate}
\end{prop}
Then we have the Fredholm index of $D\Upsilon(w)$ is given by
$$
\operatorname{Index}D\Upsilon(w) =
\operatorname{Index}(\delbar^{\nabla^\pi} + T_{dw}^{\pi,(0,1)}+ B^{(0,1)}) +
\operatorname{Index}(-\Delta_0).
$$

For this purpose of computing the Fredholm index of the linearized operator in terms of a topological index, especially in terms of those who will equip the relevant Floer-type
complex with the absolute grading, we need  the notion of anchored Legendrian submanifolds  \cite{oh-yso:index}
which is an adaptation of that of Lagrangian submanifolds studied in \cite{fooo:anchored} in symplectic geometry.
\begin{defn}\label{defn:anchor-intro}
Fix a base point $y$ of ambient contact manifold $(Q,\xi)$. Let $R$ be a Legendrian submanifold of $(Q,\xi)$. We define an
\emph{anchor} of $R$ to $y$ is a path $\ell:[0,1] \rightarrow M$ such that $\ell(0)=y,\;\ell(1)\in R$. We call a pair $(R,\ell)$
an \emph{anchored} Legendrian submanifold.
A chain $\mathcal{E} = ((R_0,\ell_0),\ldots,(R_k,\ell_k))$ is called an \emph{anchored Legendrian chain}.
\end{defn}

We refer readers to \cite{oh-yso:index} for the details of
derivation of the relevant index formula
which expresses the analytical index of the linearized problem of \eqref{eq:contacton-Legendrian-bdy-intro}
in terms of a topological index of the Maslov-type.
This index is made more explicit in \cite{oh-yso:spectral}
for the case of Hamiltonian isotopes of the zero section of the one-jet bundle.

\begin{rem} While we are preparing this survey, Dylan Cant informed
the first author of his thesis work \cite{cant:thesis} in which he studied the
Fredholm theory and index formula \emph{for the open string case} of
pseudoholomorphic curves for the  symplectization of \emph{general pair}
$(Q\times \R,R \times \R)$ with Legendrian submanifold $R$
in general dimension. As mentioned in \cite[Section 1.5]{oh-yso:index},
this case is included as a part of the study of pseudoholomorphic curves on the
$\mathfrak{lcs}$-fication of contact manifold as the exact case or as
the `zero-temperature limit' thereof.
\end{rem}

\subsection{Finer asymptotics}
\medskip

Finally we describe our work on the precise fine asymptotics of contact instantons (and so of
pseudoholomorphic curves on symplectization) through
the coordinate-free analysis of asymptotic operators and the study of their eigenvalues under the perturbation of adapted pairs $(\lambda,J)$
in our work  \cite{kim-oh:asymp-analysis}. We refer readers to Section \ref{subsec:tangentplane} for more detailed summary of the materials below.
\begin{rem} A study of precise asymptotic formula of pseudholomorphic curves on
symplectization in the more general context of stable Hamiltonian structures
is given by Siefring in \cite{siefring:relative,siefring:intersection}
which is applied to develop local intersection theory and embedding controls
(in 4 dimension) of two embedded pseudoholomorphic curves with the
same asymptotic limits.
\end{rem}

Our study provides a precise
generic description of the spectral behavior of asymptotic operators under the change $d\lambda$-compatible CR almost complex structures $J$ when $\lambda$ is fixed. For example, we prove the following in \cite{kim-oh:asymp-analysis}.

An upshot of our study of asymptotic behavior of finite contact instantons
(and hence that of pseudoholomorphic curves on symplectization)
is our usage of the following variable
$$
\zeta(\tau,t): = \left(\frac{\del w}{\del \tau}\right)^\pi
$$
on the cylindrical coordinates of $\dot \Sigma$ near a puncture
in our various study of contact instantons which enables us to avoid
any usage of special target coordinates that has been essential in
Hofer-Wyosocki-Zehnder's approach, when we take the covariant tensorial
approach using the contact triad connection.

We will derive in \eqref{eq:fundamental-isothermal} that $\zeta$ satisfies
\be\label{eq:fundamental-isothermal2}
\nabla^\pi_\tau \zeta + J \nabla^\pi_t + S \zeta = 0
\ee
where $S$ is the zero operator
$$
S \zeta = \frac12 w^*\lambda(\del_\tau) (\CL_{R_\lambda}J) \zeta
 + \frac12w^*\lambda(\del_t)(\CL_{R_\lambda}J)J \zeta.
$$
We write the loop $w_\tau: = w(\tau, \cdot)$ and
$$
A^\tau =  ( - J \nabla^\pi_t - S)|_{w_\tau^*\xi}
$$
which we know converges to the asymptotic operator $A^\pi_{(\lambda,J,\nabla)}$.
We define
$$
\nu_\tau: = \frac{\zeta_\tau}{\|\zeta_\tau\|_{L^2}}
$$
so that it has unit $L^2$-norm on $S^1$ for all $\tau \geq \tau_0$.
With this definition, one can following essentially verbatim the proof of \cite[Theorem 2.8]{HWZ:behavior} using the tensorial
way considering parallel transport of sections of $w^*\xi$
with respect to the canonical connection.

Then the following theorem describes the asymptotic behavior of $w$ as $|\tau| \to \infty$,
which is the analog to \cite[Theorem 1.4]{HWZ:asymptotics}.
Our proof is the covariant tensorial version of the Hofer-Wysocki-Zehnder's proof of
\cite[Theorem 2.8]{HWZ:behavior} given in Section 3 therein via coordinate calculations
in 3 dimension. (See Remark \ref{rem:alert} for some caution required when
readers compare our proof with theirs.)

\begin{thm}[Asymptotic behavior]
\label{thm:behavior-intro} Assume that $(\gamma,T)$ is nondegenerate.
Consider the unit vectors.
$$
\nu(\tau,t): = \frac{\zeta(\tau,t)}{\|\zeta(\tau)\|_{L^2(w_\tau^*\xi)}}, \quad
\alpha(\tau) : = \frac{d}{d\tau} \log \|\zeta(\tau)\|_{L^2(w_\tau^*\xi)}.
$$
Then we have the following:
\begin{enumerate}
\item Either $\zeta(\tau,t) = 0$ for all $(\tau,t) \in [\tau_0,\infty)$ or
\item the following hold:
\begin{enumerate}
\item There exists an eigenvector $e$ of $A^\pi_{(\lambda,J,\nabla)}(\gamma_T)$ of eigenvalue
$\mu$ such that $\nu_\tau \to e$ as $\tau \to \infty$ and
$$
\zeta(\tau,t) = e^{\int_{\tau_0}^\tau \alpha(s)\, ds} (e(t) + \widetilde r(\tau,t))
$$
where $\widetilde r(\tau,t) \to 0$ in $C^\infty$ topology.
\item For any $0 < r < \mu$,
there exist constants $\delta >0$, $\tau_0> 0$ and $C_{\beta}$ such that
for all multi-indices $\beta = (\beta_1,\beta_2)) \in \mathbb{N}^2$ such that
$$
\sup_{(\tau,t)}|(\nabla^\beta \zeta)(\tau,t)| \leq C_\beta e^{-r \tau}
$$
for all $\tau \geq \tau_0$ and $t \in S^1$.
\end{enumerate}
\end{enumerate}
\end{thm}
The exponential convergence here appearing can be derived by a boot-strap argument
using the exponential convergence result on $\zeta$ as $\tau \to \infty$ given in
Section \ref{sec:exponential-convergence}. (See \cite{oh-wang:CR-map2},
\cite{oh-yso:index} for the details of this exponential convergence and the boot-strap
argument.)

Note that the above representation formula of $\zeta$ implies the following convergence of
the tangent plane.

\begin{cor}[Convergence of tangent plane]\label{cor:tangentplane-intro}
Assume the second alternative in
Theorem \ref{thm:behavior-intro} and denote
$$
P(\tau,t) := \Image dw(\tau,t) \in \text{\rm Gr}_2(\xi_{w(\tau,t)})
$$
where $\text{\rm Gr}_2(\xi_x)$ is the set of 2 dimensional subspaces of the
contact hyperplane $\xi_x \subset T_xQ$.
Then $P(\tau,t) \to \span_\R\{ e(t), J e(t)\}$ exponentially fast in $C^\infty$ topology
uniformly in $t \in S^1$.
\end{cor}
This convergence statement is a trivial vacuous
statement in dimension 3 since $\dim \xi = 2$ but is a nontrivial statement
for higher dimensions. We refer readers to \cite{siefring:relative,siefring:intersection} for a precise local intersection
theory of two pseudoholomorphic curves at the punctures and topological
controls on the intersection number of two curves with the same
`fine' asymptotic limit.

Finally we would like to just mention that the same asymptotic study
can be made in a straightforward way by incorporating the boundary condition
by now as done in \cite{oh-yso:index}, \cite{oh:contacton-transversality},
\cite{oh:contacton-gluing}.

\medskip

This article is largely a survey of the first author and his collaborators'  series of works on the analysis of contact instantons,
 focusing mainly on the \emph{unperturbed ones}
\cite{oh-wang:connection}, \cite{oh-wang:CR-map1,oh-wang:CR-map2}, \cite{oh:contacton} and \cite{oh:contacton-Legendrian-bdy}.

This survey  does not touch upon the case of \emph{Hamiltonian-perturbed
contact instantons} studied in \cite{oh:perturbed-contacton-bdy}
for the elliptic regularity theory generalizing that of
\cite{oh:contacton-Legendrian-bdy,oh-yso:index} in which does lie
the real power of our approach through the interplay between
geometric analysis of perturbed contact instantons and
the calculus of Hamiltonian geometry and dynamics. We refer
readers to   \cite{oh:entanglement1}, \cite{oh:shelukhin-conjecture}
in which the interplay has been exhibited by the proof
of Sandon-Shelukhin type conjectures.

The case of unperturbed contact instantons
corresponds to the case of Gromov's original pseudoholomorphic curves
\cite{gromov:invent} while the perturbed ones correspond to solutions of Floer's Hamiltonian-perturbed  trajectory equations
\cite{floer:fixedpoints,salamon-zehnder} in symplectic geometry.  We refer interested readers
to \cite{oh:perturbed-contacton-bdy} for the case of perturbed equation.
We also refer readers to \cite{oh-wang:CR-map2} for the tensorial proof of exponential convergence result
in the Morse-Bott nondegenerate case of contact instantons again utilizing Weitezenb\"ock-type
formulae with respect to contact triad connections.

Throughout the paper, we freely use the (covariant) exterior calculus of $E$-valued differential
forms and Weitzenb\"ock formula with respect to the contact triad connection $\nabla$.
For the convenience of the prospective readers  of this survey article,
we duplicate \cite[Appendix A \& B]{oh-wang:CR-map1} here which summarizes the exterior
calculus of vector-valued forms and the derivation of Weitzenb\"ock formula
 in Appendix of the present paper.

There is one exception of the usage of contact triad connection: This is
for the study of asymptotic operators
along the Reeb orbits for the closed case (or along the Reeb chords in the open string case) where
the derivation of asymptotic operators is done using the triad connection but converted the formula in the final conclusion
to one involving the Levi-Civita connection of the triad metric \emph{along the Reeb orbits (or along the Reeb chords}) because of our desire to more
widely advertize the wonderful property of the Levi-Civita connection 
given in Proposition \ref{prop:blair}.

\medskip
\noindent{\bf Acknowledgement:}  We thank MATRIX for providing
an excellent research environment and Brett Parker for his great effort
for smoothly running the IBSCGP-MATRIX Symplectic Topology Workshop.
We also thank all participants of the workshop
for making the workshop a big success.
The first author also thanks Givental for useful communication on
SFT compactification, and Hutchings for brining our attention to
Siefring's paper \cite{siefring:relative}.

\bigskip

\noindent{\bf Conventions:} All the conventions regarding the definition of Hamiltonian vector fields, canonical symplectic forms on the cotangent bundle and definition of contact Hamiltonians and others
are the same as those adopted and listed in \cite[Conventions]{oh:contacton-Legendrian-bdy}.
These also coincide with the conventions used in \cite{dMV}.

\part{Contact triads and their $\mathfrak{lcs}$-fications}

\section{Contact triad connection and canonical connection}
\label{sec:connection}

Let $(Q,\xi)$ be a given contact manifold. When a contact form $\lambda$
is given, we have the projection
$\pi=\pi_\lambda$ from $TM$ to $\xi$ associated to the decomposition
$$
TQ = \xi \oplus \R \langle R_\lambda \rangle.
$$
We denote by $\Pi=\Pi_\lambda: TM \to TM$
 the corresponding idempotent, i.e., the endomorphism of $TM$ satisfying
${\Pi}^2 = \Pi$, $\Im \Pi = \xi$, $\ker \Pi = \R\{R_\lambda\}$.

\subsection{Contact triads and triad connections}

Let  $(M, \lambda, J)$ be a contact triad of dimension $2n+1$ for the contact manifold $(M, \xi)$, and equip with it the contact triad metric
$g=g_\xi+\lambda\otimes\lambda$.
In \cite{oh-wang:connection}, Wang and the first author
 introduced the \emph{contact triad connection} associated to every contact triad $(M, \lambda, J)$ with the contact triad metric and proved its existence and uniqueness and naturality.

\begin{thm}[Contact Triad Connection \cite{oh-wang:connection}]\label{thm:connection}
For every contact triad $(M,\lambda,J)$, there exists a unique affine connection $\nabla$, called the contact triad connection,
 satisfying the following properties:
\begin{enumerate}
\item The connection $\nabla$ is  metric with respect to the contact triad metric, i.e., $\nabla g=0$;
\item The torsion tensor $T$ of $\nabla$ satisfies $T(R_\lambda, \cdot)=0$;
\item The covariant derivatives satisfy $\nabla_{R_\lambda} R_\lambda = 0$, and $\nabla_Y R_\lambda\in \xi$ for any $Y\in \xi$;
\item The projection $\nabla^\pi := \pi \nabla|_\xi$ defines a Hermitian connection of the vector bundle
$\xi \to M$ with Hermitian structure $(d\lambda|_\xi, J)$;
\item The $\xi$-projection of the torsion $T$, denoted by $T^\pi: = \pi T$ satisfies the following property:
\be\label{eq:TJYYxi}
T^\pi(JY,Y) = 0
\ee
for all $Y$ tangent to $\xi$;
\item For $Y\in \xi$, we have the following
$$
\del^\nabla_Y R_\lambda:= \frac12(\nabla_Y R_\lambda- J\nabla_{JY} R_\lambda)=0.
$$
\end{enumerate}
\end{thm}
From this theorem, we see that the contact triad connection $\nabla$ canonically induces
a Hermitian connection $\nabla^\pi$ for the Hermitian vector bundle $(\xi, J, g_\xi)$,
and we call it the \emph{contact Hermitian connection}. This connection will be
used to study estimates for the $\pi$-energy in later sections.

Moreover, the following fundamental properties of the contact triad connection was
proved in \cite{oh-wang:connection}
\begin{cor}[Naturality]
\begin{enumerate}
\item Let  $\nabla$ be the contact triad connection of
the triad $(Q,\lambda, J)$. Then for any diffeomorphism $\phi: Q \to Q$,
the pull-back connection $\phi^*\nabla$ is the triad connection
associated to the triad $(Q,\phi^*\lambda,\phi^*J)$ associated to
the pull-back contact structure $\phi^*\xi$.
\item
In particular if $\phi$ is contact, i.e., $d\phi(\xi) \subset \xi$, then
$(Q, \phi^*, \phi^*J)$ is a contact triad of $\xi$ and $\phi^*\nabla$
the contact triad connection $(Q,\xi)$.
\end{enumerate}
\end{cor}

The following identities are also very useful to perform tensorial
calculations in the study of a priori elliptic estimates and
in the derivation of the linearization formula.

\begin{cor}\label{cor:connection}
Let $\nabla$ be the contact triad connection. Then
\begin{enumerate}
\item For any vector field $Y$ on $M$,
\be\label{eq:nablalambdaY}
\nabla_Y R_\lambda = \frac{1}{2}(\CL_{R_\lambda}J)JY;
\ee
\item $\lambda(T|_\xi)=d\lambda$.
\end{enumerate}
\end{cor}

We refer readers to \cite{oh-wang:connection} for more discussion on the contact triad connection
and its relation with other related canonical type connections.

\subsection{Canonical connection on almost Hermitian manifold}
\label{sec:canonical}

In this subsection, we give the definition of canonical connection on general almost
Hermitian manifolds and apply it to the case of $\mathfrak{lcs}$-fication of
contact triad $(Q,\lambda,J)$. See \cite[Chapter 7]{oh:book1} for the exposition of
canonical connection for almost Hermitian manifolds and its relationship with the Levi-Civita connection,
and in relation to the study of pseudoholomorphic curves on general symplectic manifolds
emphasizing the Weitzenb\"ock formulae in the study of elliptic regularity in the same spirit of the present survey.

Let $(M,J)$ be any almost complex manifold.
\begin{defn} A metric $g$ on $(M,J)$ is called Hermitian, if $g$ satisfies
$$
g(Ju,Jv) = g(u,v), \quad u, \, v \in T_x M, \, x \in M.
$$
We call the triple $(M,J,g)$ an almost Hermitian manifold.
\end{defn}

For any given almost Hermitian manifold $(M,J,g)$, the bilinear form
$$
\Phi : = g(J \cdot, \cdot)
$$
is called the fundamental two-form in \cite{kobayashi-nomizu}, which
is nondegenerate.
\begin{defn} An almost Hermitian manifold $(M,J,g)$ is an \emph{almost
K\"ahler manifold} if the two-form $\Phi$ above is closed.
\end{defn}

\begin{defn}
A (almost) Hermitian connection $\nabla$ is an affine connection
satisfying
$$
\nabla g = 0 = \nabla J.
$$
\end{defn}
Existence of such a connection is easy to check.
In general the torsion $T = T_\nabla$ of the almost Hermitian connection $\nabla$ is not zero, even when
$J$ is integrable.  The following is the almost complex version of the Chern connection in complex
geometry.

\begin{thm}[\cite{gauduchon}, \cite{kobayashi}]
On any almost Hermitian manifold $(M,J,g)$,
there exists a unique Hermitian connection $\nabla$ on $TM$ satisfying
\be\label{eq:canonical-nabla}
T(X,JX) = 0
\ee
for all $X \in TM$.
\end{thm}
In complex geometry \cite{chern:connection} where $J$ is integrable, a Hermitian connection
satisfying \eqref{eq:canonical-nabla} is called the Chern connection.

\begin{defn}\label{defn:canonical-nabla} A \emph{canonical connection}
 of an almost Hermitian connection is
defined to be one that has the torsion property \eqref{eq:canonical-nabla}.
\end{defn}
The triple \eqref{eq:almost-Hermitian-triple}
$$
(Q \times \R, \widetilde J, \widetilde g_\lambda)
$$
is a natural example of an almost Hermitian manifold associated to the
contact triad $(Q,\lambda,J)$.

Let $\widetilde \nabla$ be the canonical
connection thereof. Then we have the following which also provides a natural
relationship between the contact triad connection and the canonical connection.

\begin{prop}[Canonical connection versus contact triad connection]\label{prop:canonical}
Let $\widetilde g = \widetilde g_\lambda $ be the almost Hermitan metric given above.
Let $\widetilde \nabla$ be the canonical connection of the almost Hermitian manifold
\eqref{eq:almost-Hermitian-triple}, and $\nabla$ be the contact triad connection for the triad $(Q,\lambda,J)$.
Then $\widetilde \nabla$ preserves the splitting \eqref{eq:TM-splitting}
and satisfies $\widetilde \nabla|_\xi = \nabla|_\xi$.
\end{prop}

\begin{rem}\label{rem:compatibility}
 \begin{enumerate}
\item In fact it was shown in \cite{oh-wang:connection} that for each
real constant $c$, there is the unique connection that satisfies all
properties (1)--(5) and (6) replaced by (6;c)
\be\label{eq:6c}
\nabla_{JY}R_\lambda + J\nabla_YR_\lambda = c\, Y, \text{or equivalently }\,
\del_Y^\nabla R_\lambda = \frac{c}{2}Y
\ee
for all $Y \in \xi$. Our canonical connection corresponds to $c = 0$.
In particular all of these connections, temporarily denoted by
$\nabla^c$ and called $c$-triad connection, induce the same Hermitian connection on $\xi$, i.e.,
$$
\nabla^{c;\pi} = \nabla^\pi
$$
for all $c$. With this choice of connection $\nabla^c$
\eqref{eq:nablalambdaY} is replaced by
$$
\nabla_Y^c R_\lambda = - \frac{c}{2} JY + \frac12(\CL_{R_\lambda}J) JY.
$$
It seems to us that this constant $c$ is related to the way how one lifts
triad connection to a canonical connection on the symplectization.
Recall if one starts from a Liouville manifold $M$ with cylindrical end,
only the contact structure $\xi$ is naturally induced from the
Liouville structure, but not the contact form,
on its ideal boundary $\del_\infty M$.
\item We suspect that Proposition \ref{prop:canonical} will be important
to make the Kuranishi structure constructed in
\cite{kim-oh:contacton-kuranishi} compatible with Kuranishi structure
on the moduli spaces of pseudoholomorphic curves
on noncompact symplectic manifolds $M$ of contact-type boundary
such as on Liouville manifolds. In this way, we conjecture
existence of an $A_\infty$-type functor
$$
D^\pi \CW(M)/D^\pi \CF(M) \to D^\pi \mathfrak{Leg}(\del_\infty M):
$$
Here the codomain $ D^\pi\mathfrak{Leg}(Q,\xi)$ is the
derived category of the Fukaya-type category $\mathfrak{Leg}(Q) = \mathfrak{Leg}(Q,\xi)$
generated by Legendrian submanifolds that will be
constructed in \cite{kim-oh:contacton-kuranishi}
and \cite{oh:entanglement2}. For the domain,
$D^\pi \CW(M)$ and $D^\pi \CF(M)$ are the derived wrapped
Fukaya category and the derived compact Fukaya category of
the Liouville manifold $M$ respectively. See \cite{bae-jeong-kim}
for the description of the quotient category associated to a pluming space
in terms of a cluster category associated. According thereto,
Ganatra-Gao-Venkatesh relate
this quotient to a derived Rabinowitz Fukaya category.
\end{enumerate}
\end{rem}

\section{Contact instantons and pseudoholomorphic curves on symplectization}
\label{sec:CRmap}

Denote by $(\dot\Sigma, j)$ a punctured Riemann surface (including the case of closed Riemann surfaces without punctures).

\subsection{Analysis of pseudoholomorphic curves on symplectization in the literature}

We would like to make it clear that the analytical results themselves we describe on the symplectization in the present article
are mostly known and established by Hofer-Wysocki-Zehnder's in a series of papers in \cite{HWZ:asymptotics-correction}
- \cite{HWZ:fredholm} in 3 dimension and some in higher dimensions by Bourgeois \cite{bourgeois} and Siefring \cite{siefring:relative,siefring:intersection}.
Wendl's book manuscript \cite{wendl:lecture} also describe the analysis in general dimension and
Cant's thesis \cite{cant:thesis} explains the relative case in general dimension.

To highlight the main differences between the above and Wang and the first author's analysis of
contact instantons \cite{oh-wang:CR-map1,oh-wang:CR-map2}, \cite{oh:contacton-Legendrian-bdy},
we start here with quoting a few sample statements of the main results
from \cite{HWZ:asymptotics-correction}-\cite{HWZ:fredholm}, which are propagated to other later literature.

The following is the prototype of the statements made in \cite{HWZ:asymptotics-correction}
- \cite{HWZ:fredholm}  on the description of finer asymptotics
of pseudoholomorphic curves near punctures,
and repeated in other later literature on the pseudoholomorphic curves on symplectization.

\subsubsection{Choice of special coordinates}

\begin{thm}[Theorem 2.8 \cite{HWZ:asymptotics}, Asymptotic behavior of nondegenerate finite energy planes]
Assume the functions $(a,u); [s_0,\infty) \times \R \to \R^4$ meet the above conditions. Then either
\begin{enumerate}
\item There exists $c \in \R$ such that
$$
(a(s,t),\vartheta(s,t), z(s,t)) = (Ts + c,kt,0)
$$
or
\item There are constants $c \in \R, \, d> 0$ and $M_\alpha > 0$ for all
$\alpha = (\alpha_1,\alpha_2) \in \N \times \N$ such that
\beastar
|\del^\alpha[a(s,t) - Ts - c]| &\leq& M_\alpha e^{-d\cdot s} \\
|\del^\alpha[\vartheta(s,t) - kt]| & \leq & M_\alpha e^{-d\cdot s}
\eeastar
for all $s \geq s_0,\, t \in \R$. Moreover
$$
z(s,t) = e^{\int_{s_0}^s \gamma(\tau) \, d\tau} [e(t) + r(s,t)].
$$
Here $e \neq 0$ is an eigenvector of the self-adjoint operator $A_\infty$ corresponding to
a negative eigenvalue $\lambda < 0$ and $\gamma:[s_0,\infty) \to \R$ is a smooth function
satisfying $\gamma(s) \to \lambda$ as $\lambda \to \infty$. In particular $e(t) \neq 0$ pointwise and
the remainder $r(s,t)$ satisfies $\del^\alpha r(s,t) \to 0$ for
all derivatives $\alpha = (\alpha_1,\alpha_2)$, uniformly in $t \in \R$.
\end{enumerate}
\end{thm}

To arrive at these statements, the authors therefrom start with a finite energy cylinder $\widetilde v: \R \times S^1
\to Q \times \R$ with $\widetilde v = (v,a)$ for a given contact manifold $(Q,\lambda)$.

\medskip

They take special coordinates $(a,\vartheta,z)$
in an \emph{$S^1$-invariant} neighborhood $W \subset Q$ of the given Reeb orbit $x(T\cdot) \in C^\infty(S^1,Q)$ in such a way that
\begin{enumerate}
\item $\Image v \subset W$ for all $s$ large enough,
\item They choose the coordinates $(\theta,x,y) \in S^1 \times \R^2$ by considering
contact diffeomorphism $(S^1 \times \R^2, f\cdot \lambda_0)$ to $(Q,\lambda)$ where
\begin{itemize}
\item  the periodic solution corresponds to $S^1 \times \{0\}$,
\item $f$ a positive function,
\item $f \cdot \lambda_0$ is a contact form with $\lambda_0 = d\theta + x\,dy$ the standard
contact form on $S^1 \times \R^2$.
\end{itemize}
\end{enumerate}

\subsubsection{Local adjustment of $J$ along the Reeb orbits}

Next, they put some constraints on the almost complex structure in
$W \times \R \subset Q \times \R$
which has troubled the first author much:
Along the way, they \emph{change the given almost complex structure} on $W$ \emph{
in the way that depends on the given Reeb orbit} as follows.
Consider the symplectic inner product $d\lambda$ on the $S^1 \times \R^2$ family of
2-dimensional contact plane $\xi_m = \ker \lambda_m$ is given by
$$
\xi_m \subset \R^3 \cong T_m(S^1 \times \R^2)
$$
where $\xi_m = \span\langle e_1, e_2 \rangle$ with
$$
e_1 = (0,1,0), \quad e_2 = (-x_1, 0, 1).
$$
Then the $2 \times 2$ matrix $\Omega  = \Omega(\theta,x,y)$ associated to $d\lambda$ on $\xi_m$
is given by
$$
\Omega = f J_0, \quad J_0 = \left(\begin{matrix} 0 & -1 \\ 1 & 0 \end{matrix}\right)
$$
where $J_0$ is the standard complex structure on $\R^2$. Then the almost complex structure
$j_m: \xi_m \to \xi_m$ is the pull-back of a standard $\lambda$-adapted CR almost complex structure
on $Q$ restricted to its contact distributions. Then they put a `very special' almost complex structure
\cite[In the middle of p.351]{HWZ:asymptotics} denoted by $\widetilde J$.

\subsubsection{Weak points of the aforementioned coordinate approach}

This usage of special coordinates and adjustment on $J$
along the Reeb orbits \emph{depending on the Reeb orbits} prevents one
from enabling to study the $J$-dependence on the generic properties of asymptotic behaviors of pseudoholomorphic curves, e.g.,
about the generic properties of eigenvalues of the
asymptotic operators in terms of the choice of $J$.
In this regard, Siefring \cite{siefring:relative,siefring:intersection}
studied an operator that he calls
the asymptotic operator in the coordinate-free way but
did not develop its spectral perturbation theory under the change of
CR almost complex structures $J$.

The first author's attempt to find a different way of doing the
aforementioned asymptotic
study of Reeb orbits, especially that of doing those practiced in the higher dimensional Morse-Bott case
given in \cite{bourgeois}, motivated Wang and him to pursue the current tensorial approach at the time of writing
\cite{oh-wang:connection,oh-wang:CR-map1,oh-wang:CR-map2} in the beginning of 2010's. Furthermore our tensorial approach also helps our
understanding of the background geometry of contact triads $(M,\lambda, J)$ and the symplectization
irrespective of the analytic study of pseudoholomorphic curves as described
in Section \ref{sec:connection}.

\subsection{Contact Cauchy-Riemann maps}

The following definition is introduced in \cite{oh-wang:CR-map1}.

\begin{defn}[Contact Cauchy-Riemann map]
A smooth map $w:\dot\Sigma\to M$ is called a
\emph{contact Cauchy-Riemann map}
(with respect to the contact triad $(M, \lambda, J)$), if $w$ satisfies the following Cauchy--Riemann equation
$$
\delbar^\pi w:=\delbar^{\pi}_{j,J}w:=\frac{1}{2}(\pi dw+J\pi dw\circ j)=0.
$$
\end{defn}

Recall that for a fixed smooth map $w:\dot\Sigma\to M$,
the triple
$$
(w^*\xi, w^*J, w^*g_\xi)
$$
becomes  a Hermitian vector bundle
over the punctured Riemann surface $\dot\Sigma$.  This  introduces a Hermitian bundle structure on
$Hom(T\dot\Sigma, w^*\xi)\cong T^*\dot\Sigma\otimes w^*\xi$ over $\dot\Sigma$,
with inner product given by
$$
\langle \alpha\otimes \zeta, \beta\otimes\eta \rangle =h(\alpha,\beta)g_\xi(\zeta, \eta),
$$
where $\alpha, \beta\in\Omega^1(\dot\Sigma)$, $\zeta, \eta\in \Gamma(w^*\xi)$,
and $h$ is the K\"ahler metric on the punctured Riemann surface $(\dot\Sigma, j)$.

Let $\nabla^\pi$ be the contact Hermitian connection.
Combining the pulling-back of this connection and the Levi--Civita connection of the Riemann surface,
we get a Hermitian connection for the bundle $T^*\dot\Sigma\otimes w^*\xi \to \dot\Sigma$, which we will still denote by $\nabla^\pi$ by a slight abuse
of notation. \emph{This is the setting where we apply the harmonic theory and
Weitzenb\"ock formulae to study the global a priori $W^{1,2}$-estimate of
$d^\pi w$:} The smooth map $w$ has an associated $\pi$-harmonic energy density,  the function $e^\pi(w):\dot\Sigma\to \R$, defined by
$$
e^\pi(w)(z):=|d^\pi w|^2(z).
$$
(Here we use $|\cdot|$ to denote the norm from $\langle\cdot, \cdot \rangle$ which should be clear from the context.)

Similar to standard Cauchy--Riemann maps for almost Hermitian manifolds (i.e., pseudo-holomorphic curves),
we have the following whose proofs are straightforward and so omitted.
(See \cite[Lemma 3.2]{oh-wang:CR-map1} for the proofs.)

\begin{lem}\label{lem:energy-omegaarea}
Fix a K\"ahler metric $h$ on $(\dot\Sigma,j)$,
and consider a smooth  map $w:\dot\Sigma \to M$.  Then we have the following equations
\begin{enumerate}
\item $e^\pi(w):=|d^\pi w|^2 = |\del^\pi w| ^2 + |\delbar^\pi w|^2$;
\item $2\, w^*d\lambda = (-|\delbar^\pi w|^2 + |\del^\pi w|^2) \,dA $
where $dA$ is the area form of the metric $h$ on $\dot\Sigma$;
\item $w^*\lambda \wedge w^*\lambda \circ j = - |w^*\lambda|^2\, dA$.
\end{enumerate}
As a consequence, if $w$ is a contact Cauchy-Riemann map, i.e., $\delbar^\pi w=0$, then
\be\label{eq:onshell}
|d^\pi w|^2 = |\del^\pi w| ^2 \quad \text{and}\quad w^*d\lambda = \frac{1}2|d^\pi w|^2 \,dA.
\ee
\end{lem}

However the contact Cauchy-Riemann equation itself
$\delbar^\pi w = 0$ does not form
an elliptic system because it is degenerate along the Reeb direction: Note that the rank of
$w^*TQ$ has $2n+1$ while that of $w^*\xi\otimes \Lambda^{0,1}(\Sigma)$ is $2n$.
Therefore to develop suitable deformation theory and a priori estimates, one needs to
lift the equation to an elliptic system by incorporating the data of the Reeb
direction.  In hindsight, the pseudoholomorphic curve
system of the pair $(a,w)$ is one of many possible
such liftings by introducing an auxiliary variable $a$,
\emph{when the one-form $w^*\lambda \circ j$ is exact.}
Hofer \cite{hofer:invent} did this by lifting the equation to the
symplectization $\R \times Q$
and considering the pull-back function $f: = s\circ w$ of the $\R$-coordinate function
$s$ of $\R \times Q$. By doing so, he \emph{added one more variable} to the equation $\delbar^\pi w =0$
while \emph{adding 2 more equations $w^*\lambda \circ j = df$} which becomes Gromov's pseudoholomorphic curve system on the product
$\R \times Q$.

We now introduce two other possible elliptic liftings of Cauchy-Riemann
maps.

\subsection{Gauged sigma model lifting of contact Cauchy-Riemann map}

We have the lifting of Cauchy-Riemann map $w$
 involving a section of complex line bundle over $Q$
\be\label{eq:CLlambda}
\CL_\lambda \to Q
\ee
whose fiber at $q \in Q$ is given by
$$
\CL_{\lambda,q} = \R_{\lambda,q} \otimes \C
$$
where $\R_\lambda \to Q$ is the trivial real line bundle whose fiber at $q$ is given by
$$
\R_{\lambda,q} = \R \{R_\lambda(q)\}.
$$
Now let $w: \Sigma \to Q$ be a smooth map
where $\dot \Sigma$ is either closed or
a punctured Riemann surface, and $\chi$ be a section of the pull-back bundle
$w^*\CL_\lambda$.

\begin{defn} We call a triple $(w,j,\chi)$ consisting of a complex structure $j$ on $\Sigma$,
$w: \Sigma \to Q$ and a $\C$-valued one-form
$\chi$ a \emph{gauged contact instanton} if they satisfy
\be\label{eq:gauged-instanton}
\begin{cases}
\delbar^\pi w = 0 \\
\delbar \chi = 0, \quad
\Im \chi = w^*\lambda.
\end{cases}
\ee
\end{defn}

This system is a coupled system of the contact Cauchy-Riemann map equation and the well-known Riemann-Hilbert problem of the type
which solves the real part in terms of
the imaginary part of holomorphic functions in complex variable theory.

\subsection{Contact instanton lifting of contact Cauchy-Riemann map}

By augmenting the closedness condition $d(w^*\lambda\circ j)=0$ to
Contact Cauchy-Riemann map equation $\delbar^\pi w = 0$, we arrive at
an elliptic system \eqref{eq:contact-instanton}.The current contact instanton map system
\be\label{eq:contact-instanton}
\delbar^\pi w = 0, \quad d(w^*\lambda \circ j) = 0
\ee
is another such an elliptic lifting which is more natural
in some respect in that it does not introduce
any additional variable and keeps the original `bulk',
the contact manifold $Q$.

To illustrate the effect of the closedness condition on the behavior of
contact instantons, we look at them on \emph{closed}
Riemann surface and prove the following classification
result. The following proposition is stated by Abbas as a part of \cite[Proposition 1.4]{abbas}.  We refer readers
\cite[Proposition 3.4]{oh-wang:CR-map1} for another
proof which is somewhat different from Abbas' in \cite{abbas}.

\begin{prop}[Proposition 1.4, \cite{abbas}]\label{prop:abbas} Assume $w:\Sigma\to M$ is a smooth contact instanton from a closed Riemann surface.
Then
\begin{enumerate}
\item If $g(\Sigma)=0$, $w$ is a constant map;
\item If $g(\Sigma)\geq 1$, $w$ is either a constant or the locus of its image
is a \emph{closed} Reeb orbit.
\end{enumerate}
\end{prop}

\part{A priori estimates}

\section{Weitzenb\"ock formulae}
\label{sec:weitzenbock}

In this section, we use the contact triad connection, first to derive
Weitzenb\"ock-type formula associated to the $\pi$-harmonic energy for contact Cauchy--Riemann maps, and then to derive another formula
associated to the full harmonic energy density for the contact instantons.
The contact triad connection fits well for this purpose which will be seen clearly in this section.

\subsection{Weitzenb\"ock formulae for contact Cauchy-Riemann maps}

We start with by looking at the (Hodge) Laplacian of the $\pi$-harmonic energy density
of an arbitrary smooth map $w: \dot \Sigma \to M$, i.e., in the
\emph{off-shell level} in physics terminology.
As the first step, we apply the standard Weitzenb\"ock formula to the connection $\nabla^\pi$
on $T^*\dot\Sigma \otimes w^*\xi$ that is induced by the pull-back connection on
bundle $w^*\xi$ and the Levi--Civita connection on $T\dot \Sigma$, and obtain
the following formula
\bea
-\frac{1}{2}\Delta e^\pi(w)&=&|\nabla^\pi(d^\pi w) |^2-\langle \Delta^{\nabla^\pi} d^\pi w, d^\pi w\rangle
+K\cdot |d^\pi w|^2+\langle \ric^{\nabla^\pi}(d^\pi w), d^\pi w\rangle.\nonumber\\
&&\label{eq:bochner-weitzenbock-e}
\eea
Here $e^\pi(w)$ is the $\pi$-harmonic energy density,
$K$ the Gaussian curvature of $\dot\Sigma$,
and $\ric^{\nabla^\pi}$ is the Ricci tensor of the connection $\nabla^\pi$
on the vector bundle $w^*\xi$.
(We refer to \cite[Appendix A]{oh-wang:CR-map1} for the
he proof of  \eqref{eq:bochner-weitzenbock-e}.)

The following fundamental
identity
is derived in \cite[Lemma 4.1]{oh-wang:CR-map1}
for whose derivation we refer readers to.
This  is an analog to a similar formula
\cite[Lemma 7.3.2]{oh:book1} in the symplectic context.

\begin{lem}\label{lem:FE-autono} Let $w: \dot\Sigma \to M$ be any smooth map. Denote by $T^\pi$ the torsion tensor of $\nabla^\pi$.
Then as a two form with values in $w^*\xi$,
$d^{\nabla^\pi} (d^\pi w)$ has the expression
\be\label{eq:dnabladpiw}
d^{\nabla^\pi} (d^\pi w)= T^\pi(\Pi dw, \Pi dw)+ w^*\lambda \wedge \left(\frac{1}{2} (\CL_{R_\lambda}J)\, Jd^\pi w\right).
\ee
\end{lem}

We now restrict the above lemma to the case of
contact Cauchy--Riemann map, i.e., maps satisfying $\delbar^\pi w = 0$.
In this case, by the property $T(X,JX) = 0$
of the torsion $T$ of the contact triad connection,
we derive the following formula as an immediate corollary of the previous lemma.

\begin{thm}[Fundamental Equation; Theorem 4.2 \cite{oh-wang:CR-map1}]\label{thm:Laplacian-w}
Let $w$ be a contact Cauchy--Riemann map, i.e., a solution of $\delbar^\pi w=0$.
Then
\be\label{eq:Laplacian-w}
d^{\nabla^\pi} (d^\pi w) = -w^*\lambda\circ j \wedge\left(\frac{1}{2} (\CL_{R_\lambda}J)\,
d^\pi w\right).
\ee
\end{thm}

The following elegant expression of Fundamental Equation in any
\emph{isothermal coordinates} $(x,y)$, i.e., one such that
$z = x+ iy$ provides a complex coordinate of $(\dot \Sigma,j)$
such that $h = dx^2 + dy^2$, will be extremely useful for the study of
higher a priori $C^{k,\alpha}$ H\"older estimates.

\begin{cor}[Fundamental Equation in Isothermal Coordinates]
Let $(x,y)$ be an isothermal coordinates.
Write $\zeta := \pi \frac{\del w}{\del \tau}$
as a section of $w^*\xi \to M$. Then
\be\label{eq:fundamental-isothermal}
\nabla_x^\pi \zeta + J \nabla_y^\pi \zeta
 - \frac{1}{2} \lambda\left(\frac{\del w}{\del x}\right)(\CL_{R_\lambda}J)\zeta
 + \frac{1}{2} \lambda\left(\frac{\del w}{\del y}\right)(\CL_{R_\lambda}J)J\zeta =0.
\ee
\end{cor}
\begin{proof}
We denote $\pi \frac{\del w}{\del y}$ by $\eta$. By the isothermality of the coordinate $(x,y)$, we
have $J \frac{\del}{\del x} = \frac{\del}{\del y}$. Using the $(j,J)$-linearity of $d^\pi w$, we derive
$$
\eta = dw^\pi \left(\frac{\del}{\del y}\right) = dw^\pi \left( j\frac{\del}{\del x}\right) = J dw^\pi \left(\frac{\del}{\del x}\right) = J\zeta.
$$
Now we evaluate each side of \eqref{eq:Laplacian-w} against
$(\frac{\del}{\del x}, \frac{\del}{\del y})$.
For the  left hand side, we get
$$
\nabla^\pi_x \eta -\nabla^\pi_y \zeta = \nabla^\pi_x J\zeta -\nabla^\pi_ty
\zeta = J\nabla^\pi_x \zeta -\nabla^\pi_y \zeta.
$$
For the right hand side, we get
\beastar
&{}& \frac{1}{2} \lambda\left(\frac{\del w}{\del x}\right)(\CL_{R_\lambda}J) J\eta
-\frac{1}{2} \lambda\left(\frac{\del w}{\del y}\right)(\CL_{R_\lambda}J) J\zeta\\
& = & - \frac{1}{2} \lambda\left(\frac{\del w}{\del x}\right)(\CL_{R_\lambda}J) \zeta
-\frac{1}{2} \lambda\left(\frac{\del w}{\del y}\right)(\CL_{R_\lambda}J) J\zeta
\eeastar
where we use the equation $\eta = J \zeta$ for the equality. By setting them equal and applying $J$
to the resulting equation using the fact that $\CL_{R_\lambda}J$ anti-commutes with $J$, we
obtain the equation.
\end{proof}

\begin{rem}
The fundamental equation in cylindrical (or strip-like) coordinates is nothing but the linearization equation of the contact Cauchy-Riemann
equation in the direction
$\frac{\del}{\del\tau}$. This  plays an important role in the derivation
of the exponential decay of the derivatives at cylindrical ends.
(See \cite[Part II]{oh-wang:CR-map2}.)
\end{rem}

Now we can convert the general Weitzenb\"ock formula
\eqref{eq:bochner-weitzenbock-e} into

\begin{prop}[Equation (4.11) \cite{oh-wang:CR-map1}]
\label{prop:e-pi-weitzenbock}
Let $w$ be a contact Cauchy-Riemann map. Then
\bea\label{eq:e-pi-weitzenbock}
-\frac{1}{2}\Delta e^\pi(w)&=&|\nabla^\pi (\del^\pi w)|^2+K|\del^\pi w|^2+\langle \ric^{\nabla^\pi} (\del^\pi w), \del^\pi w\rangle\nonumber\\
&{}&+\langle \delta^{\nabla^\pi}[(w^*\lambda\circ j)\wedge (\CL_{R_\lambda}J)\del^\pi w], \del^\pi w\rangle.
\eea
\end{prop}
The upshot of the equation is that the \emph{a priori the third derivatives involving general LHS for general smooth maps
can be written in terms of those involving at most the second derivatives for the contact
Cauchy-Riemann maps}.

\begin{proof}[Outline of the proof]
The following formula expresses $\Delta^{\nabla^\pi}d^\pi w$,
which involves the third derivatives of $w$,
in terms of the terms involving derivatives of order at most two.

\begin{lem}\label{lem:Hodge-Laplacian}
For any contact Cauchy--Riemann map $w$,
\beastar
-\Delta^{\nabla^\pi}d^\pi w
&=&
\delta^{\nabla^\pi}[(w^*\lambda\circ j)\wedge (\CL_{R_\lambda}J)\del^\pi w]\\
&=&-*\langle (\nabla^\pi(\CL_{R_\lambda}J))\del^\pi w, w^*\lambda\rangle\\
&&
-*\langle (\CL_{R_\lambda}J)\nabla^\pi\del^\pi w, w^*\lambda\rangle
-*\langle (\CL_{R_\lambda}J)\del^\pi w, \nabla w^*\lambda\rangle.
\eeastar
\end{lem}

\begin{proof}The first equality immediately follows from the fundamental equation,
Theorem \ref{thm:Laplacian-w},  for contact Cauchy--Riemann maps.
For the second equality, we calculate by writing
\beastar
\delta^{\nabla^\pi}[(w^*\lambda\circ j)\wedge (\CL_{R_\lambda}J)\del^\pi w]=
 -*d^{\nabla^\pi}*[(\CL_{R_\lambda}J)\del^\pi w\wedge (*w^*\lambda)],
\eeastar
and then by applying the definition of the Hodge $*$  to the expression $*[(\CL_{R_\lambda}J)\del^\pi w\wedge (*w^*\lambda)]$,
we further get
\beastar
&&\delta^{\nabla^\pi}[(w^*\lambda\circ j)\wedge (\CL_{R_\lambda}J)\del^\pi w]\\
&=&-*d^{\nabla^\pi}\langle (\CL_{R_\lambda}J)\del^\pi w, w^*\lambda\rangle\\
&=&-*\langle (\nabla^\pi(\CL_{R_\lambda}J))\del^\pi w, w^*\lambda\rangle
-*\langle (\CL_{R_\lambda}J)\nabla^\pi\del^\pi w, w^*\lambda\rangle
-*\langle (\CL_{R_\lambda}J)\del^\pi w, \nabla w^*\lambda\rangle.
\eeastar
\end{proof}

This leads us to the following formula

\begin{cor} For any contact Cauchy-Riemann map $w$, we have
\beastar
- \langle \Delta^{\nabla^\pi}d^\pi w, d^\pi w \rangle
&=&
\langle \delta^{\nabla^\pi}[(w^*\lambda\circ j)
\wedge (\CL_{R_\lambda}J)d^\pi w], d^\pi w\rangle \\
&=&-\langle * \langle (\nabla^\pi(\CL_{R_\lambda}J))\del^\pi w, w^*\lambda\rangle, d^\pi w\rangle \\
&&
-\langle * \langle (\CL_{R_\lambda}J)\nabla^\pi\del^\pi w, w^*\lambda\rangle,
d^\pi w\rangle \\
&&
-\langle *\langle (\CL_{R_\lambda}J)\del^\pi w, \nabla w^*\lambda\rangle,
d^\pi w\rangle.
\eeastar
\end{cor}

Here in the above lemma $\langle\cdot, \cdot\rangle$ denotes the inner product induced from $h$, i.e.,
$\langle\alpha_1\otimes\zeta, \alpha_2\rangle:=h(\alpha_1, \alpha_2)\zeta$,
for any $\alpha_1, \alpha_2\in \Omega^k(P)$ and $\zeta$ a section of $E$. This inner product should not be confused with the inner product of the vector bundles.

By applying $\delta^{\nabla^\pi}$ to \eqref{eq:dnabladpiw} and the resulting expression of
$\Delta^\pi(d^\pi w) = \Delta^\pi(\del^\pi w)$ thereinto, we can convert
the Weitzenb\"ock formula \eqref{eq:bochner-weitzenbock-e} to
\eqref{eq:e-pi-weitzenbock}.
\end{proof}

\subsection{The case of contact instantons}

Now we consider contact instantons which are Cauchy--Riemann maps satisfying $d( w^*\lambda \circ j)=0$ in addition.
\begin{prop}\label{prop:Delta|w*lambda|2}  Let $w$ be a contact instanton. Then
\be\label{eq:Delta|w*lamba|2}
-\frac{1}{2}\Delta|w^*\lambda|^2=|\nabla w^*\lambda|^2+K|w^*\lambda|^2
+ \langle *\langle \nabla^\pi \del^\pi w, \del^\pi w\rangle,  w^*\lambda\rangle.
\ee
\end{prop}
\begin{proof}
In this case again by using the Bochner--Weitzenb\"ock formula
(for forms on Riemann surface), we get the following inequality
\be\label{eq:e-lambda-weitzenbock}
-\frac{1}{2}\Delta|w^*\lambda|^2=|\nabla w^*\lambda|^2+K|w^*\lambda|^2-\langle \Delta (w^*\lambda), w^*\lambda\rangle.
\ee
Write
$$
\Delta (w^*\lambda)=d\delta (w^*\lambda)+\delta d(w^*\lambda),
$$
in which the first term vanishes since $w$ satisfies the contact instanton equations. Then
\beastar
\langle \Delta (w^*\lambda), w^*\lambda\rangle
&=&\langle \delta d(w^*\lambda), w^*\lambda\rangle\\
&=&-\frac{1}{2}\langle *d|\del^\pi w|^2, w^*\lambda\rangle\\
&=&-\langle *\langle \nabla^\pi \del^\pi w, \del^\pi w\rangle,  w^*\lambda\rangle.
\eeastar
Substituting this into \eqref{eq:e-lambda-weitzenbock}
we have finished the proof.
\end{proof}

\section{A priori $W^{2,2}$-estimates for contact instantons}

In this section, we derive basic estimates for the  harmonic
energy density $|dw|^w$ of contact instantons $w$.
These estimates are important for the derivation of local regularity and
$\epsilon$-regularity needed for the compactification of moduli
spaces in general.

We first remark that
the total harmonic energy density  is decomposed into
$$
e(w):=|dw|^2=e^\pi(w)+|w^*\lambda|^2:
$$
This follows from the decomposition
$dw = d^\pi w + w^*\lambda \, R_\lambda$ and the
orthogonality of the two summands with respect to the triad metric
and $|R_\lambda| \equiv 1$.

Therefore we will derive the Laplacian of each summand $|d^\pi w|^2$
and $|w^*\lambda|^2$.

\subsection{Computation of $\Delta |dw|^2$ and Weitzenb\"ock formulae}

Recall the formula \eqref{eq:Laplacian-w} of the Laplacian $\Delta e^\pi(w)$ from the last section.
The following estimate is proved in \cite[Equation (5.3)]{oh-wang:CR-map1}.

\begin{lem}\label{lem:3rd-term}
For any constant $c > 0$, we have
\bea\label{eq:laplacian-pi-upper}
&{}&|\langle \delta^{\nabla^\pi}[(w^*\lambda\circ j)\wedge (\CL_{R_\lambda}J)\del^\pi w], \del^\pi w\rangle|\nonumber\\
&\leq&
\frac{1}{2c}\left(|\nabla^\pi (\del^\pi w)|^2+ |\nabla w^*\lambda|^2\right)
+\left(c\|\CL_{R_\lambda}J\|^2_{C^0(M)}+\|\nabla^\pi(\CL_{R_\lambda}J)\|_{C^0(M)}\right)|dw|^4\nonumber\\
\eea
\end{lem}

This gives rise to the following differential inequality
\bea\label{eq:e-pi-weitzenbock<}
-\frac{1}{2}\Delta e^\pi(w)& \leq &
\left(1 +\frac1{2c}  |\nabla^\pi (d^\pi w)|^2 +\frac1{2c}\right)
+  \frac{1}{2c}\left( |\nabla w^*\lambda|^2\right) \nonumber\\
&{}& + (\|K\|_{C^0} +  \|\ric^{\nabla^\pi}\|_{C^0}) |d^\pi w)|^2 \nonumber\\
&{}&  +\left(c\|\CL_{R_\lambda}J\|^2_{C^0(M)}+\|\nabla^\pi(\CL_{R_\lambda}J)\|_{C^0(M)}\right)|dw|^4
\eea

Similarly as in the previous estimates for the Laplacian term of $\del^\pi w$, we can bound
\bea\label{eq:laplacian-lambda-upper}
|-\langle \Delta (w^*\lambda), w^*\lambda\rangle|&=&|\langle *\langle \nabla^\pi \del^\pi w, \del^\pi w\rangle,  w^*\lambda\rangle|\nonumber\\
&\leq& |\nabla^\pi \del^\pi w||dw|^2\nonumber\\
&\leq& \frac{1}{2c}|\nabla^\pi \del^\pi w|^2+\frac{c}{2}|dw|^4.
\eea

We add \eqref{eq:e-pi-weitzenbock<} and \eqref{eq:e-lambda-weitzenbock}, and then apply the estimates \eqref{eq:laplacian-pi-upper} and
\eqref{eq:laplacian-lambda-upper} respectively. This yields the following differential inequality for the total energy density
\bea\label{eq:laplace-e-derivative}
&{}& -\frac{1}{2}\Delta e(w)\nonumber\\
&\geq& \left(1-\frac{1}{c}\right)|\nabla^\pi(\del^\pi w)|^2+\left(1-\frac{1}{2c}\right)|\nabla w^*\lambda|^2\nonumber\\
&{}&- \left(c\|\CL_{R_\lambda}J\|^2_{C^0(M)}+\|\nabla^\pi(\CL_{R_\lambda}J)\|_{C^0(M)}+\frac{c}{2}+\|\ric\|_{C^0(M)} \right)e(w)^2
+Ke(w)\nonumber\\
\\
&\geq& - \left(c\|\CL_{R_\lambda}J\|^2_{C^0(M)}+\|\nabla^\pi(\CL_{R_\lambda}J)\|_{C^0(M)}
+\frac{c}{2}+\|\ric\|_{C^0(M)} \right)e(w)^2+Ke(w),\nonumber
\eea
for any $c>1$. We fix $c=2$ and get the following
\begin{thm}[Theorem 5.1 \cite{oh-wang:CR-map1}]
For a contact instanton $w$, we have the following total energy density estimate
\be\label{eq:Deltaew<}
\Delta e(w)\leq Ce(w)^2+\|K\|_{L^\infty(\dot\Sigma)}e(w),
\ee
where
$$
C=2\|\CL_{R_\lambda}J\|^2_{C^0(M)}+\|\nabla^\pi(\CL_{R_\lambda}J)\|_{C^0(M)}+\|\ric\|_{C^0(M)}+1
$$
which is a positive constant independent of $w$.
\end{thm}
An immediate corollary of \eqref{eq:Deltaew<}
is the following density estimate which is derived
by the standard  argument from \cite{schoen}. (Also see the proof of \cite[Theorem 8.1.3]{oh:book1}
given in the context of pseudoholomorphic curves.)
It is in turn a consequence of an application of the
mean value inequality of Morrey
(See \cite[Problems 4.5 in p.67]{gilbarg-trudinger}
 for the relevant extension of
 the mean-value inequality for the Poisson equation.)

\begin{cor}[$\epsilon$-regularity and interior density estimate;
Corollary 5.2 \cite{oh-wang:CR-map1}]\label{cor:density}
There exist constants $C, \, \varepsilon_0$ and $r_0 > 0$, depending only on $J$ and
the Hermitian metric $h$ on $\dot \Sigma$, such that for any
 $C^1$ contact instanton $w: \dot \Sigma \to M$ with
$$
E(r_0): = \frac{1}{2}\int_{D(r_0)} |dw|^2 \leq \varepsilon_0,
$$
and discs $D(2r) \subset \operatorname{Int}\Sigma$ with $0 < 2r \leq r_0$,
$w$ satisfies
\be\label{eq:schoen's}
\max_{\sigma \in (0,r]} \left(\sigma^2 \sup_{D(r-\sigma)}
e(w)\right) \leq CE(r)
\ee
for all $0< r \leq r_0$. In particular, letting $\sigma = r/2$, we obtain
\be\label{eq:supeu}
\sup_{D(r/2)} |dw|^2 \leq \frac{4C E(r)}{r^2}
\ee
for all $r \leq r_0$.
\end{cor}

Now we rewrite $\eqref{eq:laplace-e-derivative}$ into
\bea\label{eq:laplace-higherderivative}
&{}&\left(1-\frac{1}{c}\right)|\nabla^\pi(\del^\pi w)|^2+\left(1-\frac{1}{2c}\right)|\nabla w^*\lambda|^2\nonumber\\
&\leq&-\frac{1}{2}\Delta e(w) - Ke(w) \nonumber\\
& {}& \quad +\left(c\|\CL_{R_\lambda}J\|^2_{C^0(M)}+\|\nabla^\pi(\CL_{R_\lambda}J)\|_{C^0(M)}+\frac{c}{2}+\|\ric\|_{C^0(M)} \right)e^2
\eea

Taking $c=2$ and get the following coercive estimate for contact instantons,
we obtain the following differential inequality.

\begin{prop}[Equation (5.13) \cite{oh-wang:CR-map1}]
\be\label{eq:higher-derivative}
|\nabla(dw)|^2\leq C_1|dw|^4-4K|dw|^2-2\Delta e(w)
\ee
where
$$
C_1:=9\|\CL_{R_\lambda}J\|^2_{C^0(M)}+4\|\nabla^\pi(\CL_{R_\lambda}J)\|_{C^0(M)}+4\|\ric\|_{C^0(M)}+4
$$
denotes a constant.
\end{prop}

Then by multiplying the cut-off functions and doing integration by parts,
we obtain the following local $W^{2,2}$ estimate.

\begin{prop}[Proposition 5.3 \cite{oh-wang:CR-map1}]
\label{prop:coercive-L2}
For any pair of open domains $D_1$ and $D_2$ in $\dot\Sigma$ such that $\overline{D}_1\subset \Int(D_2)$,
$$
\|\nabla(dw)\|^2_{L^2(D_1)}\leq C_1(D_1, D_2)\|dw\|^2_{L^2(D_2)}+C_2(D_1, D_2)\|dw\|^4_{L^4(D_2)}
$$
for any contact instanton $w$,
where $C_1(D_1, D_2)$, $C_2(D_1, D_2)$ are some constants which depend on $D_1$, $D_2$ and $(M, \lambda, J)$, but are independent of $w$.
\end{prop}

\subsection{Local boundary $W^{2,2}$-estimate}

Now let us consider the contact instantons with Legendrian boundary condition. Let $\vec R = (R_0,\cdots, R_k)$ be a $(k+1)$-component
Legendrian link.

Consider the equation
\be\label{eq:contacton-Legendrian-bdy}
\begin{cases}
\delbar^\pi w = 0, \, \, d(w^*\lambda \circ j) = 0\\
w(\overline{z_iz_{i+1}}) \subset R_i
\end{cases}
\ee
for a smooth map $w: (\dot \Sigma, \del \dot \Sigma) \to (Q, \vec R)$.

The following local boundary a priori estimate is established in
\cite{oh:contacton-Legendrian-bdy}, \cite{oh-yso:index}.

\begin{thm}\label{thm:local-W12} Let $w: \R \times [0,1] \to M$ satisfy \eqref{eq:contacton-Legendrian-bdy-intro}.
Then for any relatively compact domains $D_1$ and $D_2$ in
$\dot\Sigma$ such that $\overline{D_1}\subset D_2$, we have
\be\label{eq:without-bdyterm}
\|dw\|^2_{W^{1,2}(D_1)}\leq C_1 \|dw\|^2_{L^2(D_2)} + C_2 \|dw\|^4_{L^4(D_2)}
\ee
where $C_1, \ C_2$ are some constants which
depend only on $D_1$, $D_2$ and $(M,\lambda, J)$ and $C_3$ is a
constant which also depends on $R_i$ with $w(\del D_2) \subset R_i$ as well.
\end{thm}

Leaving the details of the proof to \cite{oh-yso:index},
we just outline the strategy of the proof here
(See  \cite[Section 8.2 \& 8.3]{oh:book1} for the same strategy
used for  pseudoholomorphic curves with Lagrangian boundary condition.):
\begin{enumerate}
\item As the first step, we utilize the contact triad connection $\nabla$ for the study of
boundary value problem to derive the following differential inequality
\be\label{eq:differential-inequality}
\|dw\|^2_{W^{1,2}(D_1)} \leq  C_1 \|dw\|^2_{L^2(D_2)} + C_2 \|dw\|^4_{L^4(D_2)}
+ \int_{\del D}|C(\del D)|
\ee
where we have
\be\label{eq;CdelD}
C(\del D): = - 8\left\langle B\left(\frac{\del w}{\del x},\frac{\del w}{\del x}\right),\frac{\del w}{\del y}\
\right \rangle
\ee
for the second fundamental form $B$ of $\nabla$ in
the isothermal coordinate $z = x+ iy$ adapted to $\del \dot \Sigma \cap D_2$.
\item Once we derive \eqref{eq:differential-inequality}, noting that Legendrian boundary
condition $\vec R = (R_0, \cdots, R_k)$ for the contact instanton is automatically the free boundary value problem, i.e.,
$$
\frac{\del w}{\del \nu} \perp TR_i,
$$
one can use the Levi-Civita connection of a metric for which $\vec R$ becomes
\emph{totally geodesic} (i.e., $B=0$) which will eliminate the boundary contribution
appearing above in \eqref{eq:differential-inequality}. Then recalling the standard fact
that $\nabla = \nabla^{\text{\rm LC}} + P$ for a $(2,1)$ tensor, we can convert the inequality
into \eqref{eq:without-bdyterm}. (See \cite[Section 8.3]{oh:book1} for such detail.)
\end{enumerate}

The following lemmata is an important ingredients entering in the proof.

\begin{lem}\label{prop:*de|delD} Let $e = |dw|^2$ be the total
harmonic energy density function. Then we have
\be\label{eq:*de-on-bdy}
* de = - 4 \left \langle B\left(\frac{\del w}{\del x},\frac{\del w}{\del x}\right),
\frac{\del w}{\del y}\right \rangle
\ee
on $\del D$.
\end{lem}

\section{$C^{k,\delta}$ coercive estimates for $k \geq 1$: alternating boot-strap}
\label{sec:Ckdelta-estimates}

Once we have established $W^{2,2}$ estimate, we could proceed with the $W^{k+2,2}$ estimate $k \geq 1$ inductively
as in \cite[Section 5.2]{oh-wang:CR-map1}. Because of
the effect of the Legendrian boundary condition on
the higher derivative estimate, it is not quite straightforward
to boot-strap the Sobolev norms but is better to work with
$C^{k,\delta}$ H\"older norms as in \cite{oh:contacton-Legendrian-bdy},
\cite{oh-yso:index}.

Here we provide an outline of the main steps of the \emph{alternating
boot-strap arguments}  from \cite[Section 4]{oh-yso:index}
to establish higher $C^{k,\delta}$ regularity results.

We start with the fundamental equation in isothermal coordinates $z = x+ iy$
$$
\nabla_x^\pi \zeta + J \nabla_y^\pi \zeta
+ \frac{1}{2} \lambda\left(\frac{\del w}{\del y}\right)(\CL_{R_\lambda}J)\zeta - \frac{1}{2}
\lambda\left(\frac{\del w}{\del x}\right)(\CL_{R_\lambda}J)J\zeta =0.
$$
 By writing
$$
\overline \nabla^\pi := \nabla^{\pi(0,1)} = \frac{\nabla^\pi + J \nabla^\pi_{j(\cdot)}}{2}
$$
which is the anti-complex linear part of $\nabla^\pi$, and the linear operator
$$
P_{w^*\lambda}(\zeta): = \frac{1}{4} \lambda\left(\frac{\del w}{\del y}\right)(\CL_{R_\lambda}J)\zeta - \frac{1}{4}
\lambda\left(\frac{\del w}{\del x}\right)(\CL_{R_\lambda}J)J\zeta,
$$
the equation becomes
\be\label{eq:nablabar-P}
\overline \nabla^\pi \zeta + P_{w^*\lambda}(\zeta) = 0
\ee
which is a linear first-order PDE of Cauchy-Riemann type
once $w^*\lambda$ is given. We note that
by the Sobolev embedding, $ W^{2,2} \subset C^{0,\delta}$
for $0 \leq \delta < 1/2$.
Therefore we start from $C^{0,\delta}$ bound with $0 < \delta <1/2$ and will inductively bootstrap it to get
$C^{k, \delta}$ bounds for $k \geq 1$.

WLOG, we assume that $D_2 \subset \dot \Sigma$ is a semi-disc with $\del D \subset \del \dot \Sigma$
and equipped with an isothermal coordinates $(x,y)$ such that
$$
D_2 = \{ (x,y) \mid |x|^2 + |y|^2 < \epsilon, \, y \geq 0\}
$$
for some $\epsilon > 0$
and so $\del D_2 \subset \{(x,y) \in D \mid y = 0\}$. Assume $D_1 \subset D_2$
is the semi-disc with radius $\epsilon /2$.
We denote $\zeta = \pi \frac{\del w}{\del x}$, $\eta = \pi \frac{\del w}{\del y}$
as in \cite{oh-wang:CR-map1}, and consider the complex-valued function
\be\label{eq:alpha}
\alpha(x,y) = \lambda\left(\frac{\del w}{\del y}\right)
+ \sqrt{-1}\left(\lambda\left(\frac{\del w}{\del x}\right)\right)
\ee
as in \cite[Subsection 11.5]{oh-wang:CR-map2}.
\begin{rem}\label{rem:alpha} In \cite[Subsection 11.5]{oh-wang:CR-map2}, the global isothermal coordinate $(\tau,t)$
of $[0,\infty) \times S^1$ with circle-valued flat coordinate $t$ is used and the function $\alpha$ defined by
$$
\alpha(x,y) = \lambda\left(\frac{\del w}{\del y}\right) - T
+ \sqrt{-1}\left(\lambda\left(\frac{\del w}{\del x}\right)\right)
$$
is used for the exponential convergence result.
See the displayed formula right above in Lemma 11.19 of \cite[Subsection 11.5]{oh-wang:CR-map2}.
\end{rem}

We note that since $w$ satisfies the Legendrian boundary condition
 $w(\del \dot \Sigma) \subset \vec R$, we have
\be\label{eq:lambda(delw)=0}
\lambda\left(\frac{\del w}{\del x}\right) = 0
\ee
on $\del D_2$. The following formula is crucially used in \cite[Subsection 11.5]{oh-wang:CR-map2}
for the exponential decay result, and
in \cite{oh:contacton-Legendrian-bdy,oh-yso:index} for the alternating boot-strap argument for the higher regularity
but without detailed proof. Because the proof well demonstrates how important closedness of $w^*\lambda \circ j$
and the equality \eqref{eq:onshell} are
in the study of a priori elliptic estimates of contact instanton $w$ and also because how the Legendrian boundary condition
interacts with the equation,
we give the full details of its proof here.

\begin{prop}\label{prop:alpha} The complex-valued function
 $\alpha$ satisfies the equations
\be\label{eq:atatau-equation}
\begin{cases}
\delbar \alpha =\nu, \quad \nu = \frac{1}{2}|\zeta|^2 \\
\alpha(z) \in \R \quad \text{\rm for } \, z \in \del D_2
\end{cases}
\ee
\end{prop}
\begin{proof} For the equation, we first recall
$$
d(w^*\lambda\circ j)=0, \quad dw^*\lambda = \frac12 |d^\pi w|^2 \, dA
$$
where the second equality is from \eqref{eq:onshell}.
In isothermal coordinate $(x,y)$, we have $dA = dx \wedge dy$ and hence we have
$$
* (dw^*\lambda) = \frac12 |d^\pi w|^2.
$$
By the isothermality of $(x,y)$, $\{\frac{\del}{\del x}, \frac{\del}{\del y}\}$ is an orthonormal
frame of $T\dot \Sigma$ and hence
$$
|d^\pi w|^2 = \left|\pi \frac{\del w}{\del x}\right|^2 + \left|\pi \frac{\del w}{\del y}\right|^2
= 2 \left|\pi \frac{\del w}{\del x}\right|^2
$$
where the last equality follows since $0 = \delbar^\pi w  = (\delbar w)^\pi$ for
the Cauchy-Riemann operator $\delbar$ for the standard complex structure $J_0=\sqrt{-1}$.

Therefore if we write $\zeta = \pi \frac{\del w}{\del x}$,
$$
*d(w^*\lambda) =|\zeta|^2.
$$
On the other hand, using this identity, the isothermality of the coordinate again $(x,y)$
and the equation $d(w^*\lambda \circ j) = 0$, we derive
\beastar
\delbar \alpha & = & \frac12\left(\frac{\del}{\del x} + i \frac{\del}{\del y}\right)
\left(\lambda\left(\frac{\del w}{\del y}\right)
+ i \left(\lambda\left(\frac{\del w}{\del x}\right)\right)\right)\\
& = & \frac12 \left(\frac{\del}{\del x}\left(\lambda\left(\frac{\del w}{\del y}\right)\right)
- \frac12 \frac{\del}{\del y}\left(\lambda\left(\frac{\del w}{\del x}\right)\right)\right)\\
&{}& \quad + \frac{i}{2} \left(\frac{\del}{\del x}\left(\lambda\left(\frac{\del w}{\del x}\right)\right)
+ \frac12 \frac{\del}{\del y}\left(\lambda\left(\frac{\del w}{\del y}\right)\right)\right)\\
& = & \frac12 \left(d(w^*\lambda) - i d(w^*\lambda \circ j)\right)\left(\frac{\del}{\del x},\frac{\del}{\del y}\right)\\
& = & \frac12 d(w^*\lambda)\left(\frac{\del}{\del x},\frac{\del}{\del y}\right) = \frac 12* d(w^*\lambda).
\eeastar
Combining the two, we have derived $\delbar \alpha = \frac12 |\zeta|^2$.
This finishes the proof of the equation.

For the boundary condition $\alpha(z) \in \R$ for  $z \in \del D_2$, it
follows from the Legendrian boundary condition: We have
$$
\text{\rm Im}\, \alpha(z) = \lambda\left(\frac{\del w}{\del x}\right) = 0
$$
since the vector $\frac{\del w}{\del x}$ is tangent to the given Legendrian submanifold by the
adaptedness $\frac{\del}{\del x} \in \del D_2$ of the isothermal coordinate $(x,y)$ to the boundary $\del D_2$.
This finishes the proof.
\end{proof}

Then we arrive at the following system of equations for the pair $(\zeta,\alpha)$
\be\label{eq:equation-for-zeta0}
\begin{cases}\nabla_x^\pi \zeta + J \nabla_y^\pi \zeta
+ \frac{1}{2} \lambda(\frac{\del w}{\del y})(\CL_{R_\lambda}J)\zeta - \frac{1}{2} \lambda(\frac{\del w}{\del x})(\CL_{R_\lambda}J)J\zeta =0\\
\zeta(z) \in TR_i \quad \text{for } \, z \in \del D_2
\end{cases}
\ee
for some $i = 0, \ldots, k$, and
\be\label{eq:equation-for-alpha}
\begin{cases}
\delbar \alpha = \frac{1}{2}|\zeta|^2 \\
\alpha(z) \in \R \quad \text{for } \, z \in \del D_2.
\end{cases}
\ee
These two equations form a nonlinear elliptic system for $(\zeta,\alpha)$ which are coupled:
$\alpha$ is fed into
\eqref{eq:equation-for-zeta0} through its coefficients and then $\zeta$ provides the input
for the right hand side of the equation \eqref{eq:equation-for-alpha} and then back and forth. Using this structure of
coupling, we can derive the higher derivative estimates
by alternating boot strap arguments between $\zeta$ and $\alpha$
which is now in order.

\begin{thm}[Theorem 1.4, \cite{oh-yso:index}]\label{thm:local-regularity}
Let $w$ be a contact instanton satisfying \eqref{eq:contacton-Legendrian-bdy-intro}.
Then for any pair of domains $D_1 \subset D_2 \subset \dot \Sigma$ such that $\overline{D_1}\subset D_2$, we have
$$
\|dw\|_{C^{k,\delta}(D_1)} \leq C_\delta( \|dw\|_{W^{1,2}(D_2)})
$$
for some  positive function $C_\delta = C_\delta(r)$ that
is continuous at $r = 0$ which depends on $J$, $\lambda$ and $D_1, \, D_2$
but independent of $w$.
\end{thm}

The rest of the proof of this theorem given in \cite{oh-yso:index} consist of the following
steps:

\begin{enumerate}
\item  Start of alternating boot-strap: $W^{1,2}$-estimate for $dw$.
\item $C^{1,\delta}$-estimate for $w^*\lambda = f\, dx + g\, dy$.
\item $C^{1,\delta}$-estimate for $d^\pi w$.
\item $C^{2,\delta}$-estimate for $w^*\lambda$.
\item $C^{2,\delta}$-estimate for $d^\pi w$.
\item Wrap-up of the alternating boot-strap argument:
We repeat the above alternating boot strap arguments between
$\zeta$ and $\alpha$
back and forth by taking the differential with respect to $\nabla_x^{\text{\rm LC}}$
to inductively derive the $C^{k,\delta}$-estimates both for $\zeta$ and $\alpha$
in terms of $\|\zeta\|_{L^4(D_2)}$ and $\|\alpha\|_{L^4(D_2)}$ which is equivalent to
considering the full $\|dw\|_{L^4}$.
This completes the proof of Theorem \ref{thm:local-regularity}.
\end{enumerate}
We refer readers to \cite{oh-yso:index} for complete details of this
for the alternating boot strap arguments which go
back and forth between $\zeta$ and $\alpha$.

\part{Asymptotic convergence and charge vanishing}

In this part, we study the most basic asymptotic behavior of finite
energy contact instantons near the punctures.
We divide the study into two cases separately,
one the case of closed strings, i.e., on the
cylinderical region $[0,\infty) \times S^1$ near interiors
punctures, and the other the case of open strings, i.e., on the strip-like
region $[0,\infty) \times [0,1]$ with Legendrian boundary condition of
the pair $(R, R')$.

\section{Generic nondegeneracy of Reeb orbits and of Reeb chords}

 Nondegeneracy of closed Reeb orbits or of Reeb chords
is fundamental in the Fredholm property of the linearized operator of
contact instanton equations as well as of pseudoholomorphic curves on symplectization.
The main conclusion of the present subsection will be the statement on the generic nondegeneracy
 under the perturbation of contact forms or of Legendrian boundary conditions.
The study of the case of closed Reeb orbits is standard
(see \cite{albers-bramham-wendl}), and our exposition on
the results for the case of open strings is based on
\cite[Appendix B]{oh:contacton-transversality}.

 \subsection{The case of closed Reeb orbits}

Let $\gamma$ be a closed Reeb orbit of period $T > 0$. In other words,
$\gamma: \R \to M$ is a solution of $\dot x = R_\lambda(x)$ satisfying
$\gamma(T) = \gamma(0)$.
By definition, we can write $\gamma(T) = \phi^T_{R_\lambda}(\gamma(0))$
for the Reeb flow $\phi^T= \phi^T_{R_\lambda}$ of the Reeb vector field $R_\lambda$.
Therefore if $\gamma$ is a closed orbit, then we have
$$
\phi^T_{R_\lambda}(\gamma(0)) = \gamma(0)
$$
i.e.,
$p = \gamma(0)$ is a fixed point of the diffeomorphism $\phi^T$.
Since $\CL_{R_\lambda}\lambda = 0$, $\phi^T_{R_\lambda}$ is a (strict) contact diffeomorphism and so
induces an isomorphism
$$
d\phi^T(p)|_{\xi_p}: \xi_p \to \xi_p
$$
which is the linearization restricted to $\xi_p$ of the Poincar\'e return map.

\begin{defn} We say a $T$-closed Reeb orbit $(T,\lambda)$ is \emph{nondegenerate}
if $d\phi^T(p)|_{\xi_p}:\xi_p \to \xi_p$ with $p = \gamma(0)$ has not eigenvalue 1.
\end{defn}

Denote $\CL(Q)=C^\infty(S^1,Q)$ the space of loops $z: S^1 = \R /\Z \to Q$.
We consider the assignment
$$
\Phi: (T,\gamma,\lambda) \mapsto \dot \gamma - T \,R_\lambda(\gamma)
$$
which we would like to consider a section of some Banach vector bundle over
$ (0,\infty) \times \CL^{1,2}(Q) \times \CC(Q,\xi)$ where $\CL^{1,2}(Q)$
is the $W^{1,2}$-completion of $\CL(Q)$. We note the value
$$
\dot \gamma - T\, R_\lambda(\gamma) \in \Gamma(\gamma^*TQ).
$$
We denote by $L^2(\gamma^*TQ)$ the space of $L^2$-sections of
the vector bundle $\gamma^*TQ$. Then we define the vector bundle
$$
\CL^2 \to (0,\infty) \times \CL^{1,2}(Q) \times \Cont(Q,\xi)
$$
whose fiber at $(T,\gamma,\lambda)$ is $L^2(\gamma^*TQ)$. We denote by
$\pi_i$, $i=1,\, 2, \, 3$ the corresponding projections.

We denote $\mathfrak{Reeb}(Q,\xi) = \Phi^{-1}(0)$. Then the set
$Reeb(\lambda)$ of $\lambda$-Reeb orbits
$(\gamma,T)$ is nothing but $\mathfrak{Reeb}(Q,\xi) \cap \pi_3^{-1}(\lambda)$.

\begin{prop} A $T$-closed Reeb orbit $(T,\gamma)$ is nondegenerate if and only if
the linearization
$$
d_{(T,\gamma)}\Phi: \R \times W^{1,2}(\gamma^*TQ) \to L^2(\gamma^*TQ)
$$
is surjective.
\end{prop}

The following generic nondegeneracy result is proved by
Albers-Bramham-Wendl in \cite{albers-bramham-wendl}. We denote by
$$
\CC(Q,\xi)
$$
the set of contact forms of $(Q,\xi)$ equipped with $C^\infty$-topology.

\begin{thm}[Albers-Bramham-Wendl] \label{thm:ABW}
Let $(Q,\xi)$ be a contact manifold. Then there exists
a residual subset $\CC^{\text{\rm reg}}(Q,\xi) \subset \CC(Q,\xi)$ such that
for any contact form $\lambda \in \CC^{\text{\rm reg}}(Q,\xi)$ all
Reeb orbits are nondegenerate for $T> 0$.
\end{thm}
(The case $T = 0$ can be included as the Morse-Bott nondegenerate case
if we allow the action $T = 0$ by extending the definition of Reeb trajectory
to \emph{isospeed Reeb chords}  of the pairs $(\gamma,T)$ with
$\gamma:[0,1] \to Q$ with $T = \int \gamma^*\lambda$ as done in
\cite{oh:entanglement1,oh:contacton-transversality}.)

\subsection{The case of Reeb chords}

We first recall the notion of iso-speed Reeb trajectories used in
\cite{oh:entanglement1} and recall the definition of
nondegeneracy of thereof.

Consider contact triads $(Q,\lambda,J)$ and the boundary
value problem
for $(\gamma, T)$ with $\gamma:[0,1] \to Q$
\be\label{eq:chord-equation}
\begin{cases}
\dot \gamma(t) = T R_\lambda(\gamma(t)),\\
\gamma(0) \in R_0, \quad \gamma(1) \in R_1.
\end{cases}
\ee
\begin{defn}[Isospeed Reeb trajectory; Definition 2.1 \cite{oh:contacton-transversality}]
We call a pair $(\gamma,T)$
of a smooth curve $\gamma:[0,1] \to Q$ and $T \in \R$
an \emph{iso-speed Reeb trajectory} if they satisfy
\be\label{eq:isospeed-chords}
\dot \gamma(t) = T R_\lambda(\gamma(t)), \quad \int \gamma^*\lambda = T
\ee
for all $t \in [0,1]$. We call $(\gamma, T)$ an iso-speed closed Reeb orbit
if $\gamma(0) = \gamma(1)$, and an iso-speed Reeb chord of $(R_0,R_1)$
it $\gamma(0) \in R_0$ and $\gamma(1) \in R_1$ from $R_0$ to $R_1$.
\end{defn}

With this definition, we state the corresponding notion of nondegeneracy

\begin{defn}\label{defn:nondegeneracy-chords}
We say a Reeb chord $(\gamma, T)$ of $(R_0,R_1)$ is nondegenerate if
the linearization map $d\phi^T(p): \xi_p \to \xi_p$ satisfies
$$
d\phi^T(p)(T_{\gamma(0)} R_0) \pitchfork T_{\gamma(1)} R_1  \quad \text{\rm in }  \,  \xi_{\gamma(1)}
$$
or equivalently
$$
d\phi^T(p)(T_{\gamma(0)} R_0) \pitchfork T_{\gamma(1)} Z_{R_1} \quad \text{\rm in} \, T_{\gamma(1)}Q.
$$
Here $\phi^t_{R_\lambda}$ is the flow generated by the Reeb vector field $R_\lambda$.
\end{defn}
More generally, we consider the following situation.
We recall the definition of \emph{Reeb trace} $Z_R$ of a Legendrian submanifold $R$, which is defined to be
$$
Z_R: = \bigcup_{t \in \R} \phi_{R_\lambda}^t(R).
$$
(See \cite[Appendix B]{oh:contacton-transversality} for detailed discussion on its genericity.)

\begin{defn}[Nondegeneracy of Legendrian links]
\label{defn:nondegeneracy-links}
Let $\vec{R}=(R_1,\cdots,R_k)$ be a chain of Legendrian submanifolds,
which we call a (ordered) Legendrian link. We assume that we have
$$
Z_{R_i} \pitchfork R_j
$$
for all $i, \, j= 1,\ldots, k$.
\end{defn}

We denote by
$$
{\mathcal Leg}(Q,\xi)
$$
the set of Legendrian submanifold and by ${\mathcal Leg}(Q,\xi;R)$ its connected component
containing $R \in {\mathcal Leg}(Q,\xi)$, i.e, the set of Legendrian submanifolds Legendrian isotopic to
$R$. We denote by
$$
\CP({\mathcal Leg}(Q,\xi))
$$
the monoid of Legendrian isotopies $[0,1] \to {\mathcal Leg}(Q,\xi)$.
We have
natural evaluation maps
$$
\ev_0, \, \ev_1:\CP({\mathcal Leg}(Q,\xi)) \to {\mathcal Leg}(Q,\xi)
$$
and denote by
$$
\CP({\mathcal Leg}(Q,\xi), R)
= \ev_0^{-1}(R) \subset \CP({\mathcal Leg}(Q,\xi))
$$
and
$$
\CP({\mathcal Leg}(Q,\xi), (R_0,R_1)) = (\ev_0\times \ev_1)^{-1}(R_0,R_1) \subset \CP({\mathcal Leg}(Q,\xi)).
$$

We now provide the off-shell framework for the proof of nondegeneracy
in general. Denote by $\LL(Q;R_0,R_1)$ the space of paths
$$
\gamma: ([0,1], \{0,1\}) \to (Q;R_0,R_1).
$$
We consider the assignment
\be\label{eq:Phi-TR}
\Phi: (T,\gamma,\lambda) \mapsto \dot \gamma - T \,R_\lambda(\gamma)
\ee
as a section of the Banach vector bundle over
$$
(0,\infty) \times \CL^{1,2}(Q;R_0,R_1) \times \CC(Q,\xi)
$$
where $\CL^{1,2}(Q;R_0,R_1)$
is the $W^{1,2}$-completion of $\CL(Q;R_0,R_1)$. We have
$$
\dot \gamma - T\, R_\lambda(\gamma) \in \Gamma(\gamma^*TQ'T_{\gamma(0)}R_0, T_{\gamma(1)}R_1).
$$
We  define the vector bundle
$$
\CL^2(R_0,R_1) \to (0,\infty) \times \CL^{1,2}(Q;R_0,R_1) \times \CC(Q,\xi)$$
whose fiber at $(T,\gamma,\lambda)$ is $L^2(\gamma^*TQ)$. We denote by
$\pi_i$, $i=1,\, 2, \, 3$ the corresponding projections as before.

We denote $\mathfrak{Reeb}(M,\lambda;R_0,R_1) = \Phi_\lambda^{-1}(0)$,
where
$$
\Phi_\lambda: = \Phi|_{ (0,\infty) \times \CL^{1,2}(Q;R_0,R_1) \times \{\lambda\}}.
$$
Then we have
$$
\mathfrak{Reeb}(\lambda;R_0,R_1) =  \Phi_\lambda^{-1}(0) = \mathfrak{Reeb}(Q,\xi) \cap \pi_3^{-1}(\lambda).
$$
The following relative version of Theorem \ref{thm:ABW} is proved in
\cite[Appendix B]{oh:contacton-transversality}.

\begin{thm} [Perturbation of contact forms;
Theorem B.3 \cite{oh:contacton-transversality}]
\label{thm:Reeb-chord-lambda}
Let $(Q,\xi)$ be a contact manifold. Let  $(R_0,R_1)$ be a pair of Legendrian submanifolds
allowing the case $R_0 = R_1$.  There
exists a residual subset $\CC^{\text{\rm reg}}_1(Q,\xi) \subset \CC(Q,\xi)$
such that for any $\lambda \in \CC^{\text{\rm reg}}_1(Q,\xi)$ all
Reeb chords from $R_0$ to $R_1$ are nondegenerate for $T > 0$ and
Bott-Morse nondegenerate when $T = 0$.
\end{thm}
The following theorem is also proved in \cite{oh:contacton-transversality}.

\begin{thm}[Perturbation of boundaries;
Theorem B.10 \cite{oh:contacton-transversality}]
\label{thm:Reeb-chords-MB}
Let $(Q,\xi)$ be a contact manifold. Let  $(R_0,R_1)$ be a pair of Legendrian submanifolds allowing the case $R_0 = R_1$.
 For a given contact form $\lambda$ and $R_1$,
there exists a residual subset
$$
R_0 \in{\mathcal Leg}^{\text{\rm reg}}(Q,\xi) \subset {\mathcal Leg}(Q,\xi)
$$
of Legendrian submanifolds such that for all $R_0 \in {\mathcal Leg}(Q,\xi)$ all
Reeb chords from $R_0$ to $R_1$ are nondegenerate for $T > 0$ and Morse-Bott nondegenerate when $T = 0$.
\end{thm}

We refer readers to \cite[Appendix B]{oh:contacton-transversality} for the proofs of these results.

\section{Subsequence convergence}
\label{sec:subsequence-convergence}

In this section, we study the asymptotic behavior of contact instantons
on the Riemann surface $(\dot\Sigma, j)$ associated with a metric $h$ with \emph{cylinder-like ends} for the closed string context and with
\emph{strip-like ends} for the open string context.

\subsection{Closed string case}

We assume there exists a compact set $K_\Sigma\subset \dot\Sigma$,
such that $\dot\Sigma-\Int(K_\Sigma)$ is a disjoint union of punctured disks
 each of which is isometric to the half cylinder $[0, \infty)\times S^1$ or
 the half strip $(-\infty, 0]\times [0,1]$, where
the choice of positive or negative strips depends
on the choice of analytic coordinates at the punctures.
We denote by $\{p^+_i\}_{i=1, \cdots, l^+}$ the positive punctures, and by $\{p^-_j\}_{j=1, \cdots, l^-}$ the negative punctures.
Here $l=l^++l^-$. Denote by $\phi^{\pm}_i$ such cylinder-like coordinates.
We first state our assumptions for the study of the behavior of boundary punctures.
(The case of interior punctures is treated in \cite[Section 6]{oh-wang:CR-map1}.)

\begin{defn} Let $\dot\Sigma$ be a boundary-punctured Riemann surface of genus zero with punctures
$\{p^+_i\}_{i=1, \cdots, l^+}\cup \{p^-_j\}_{j=1, \cdots, l^-}$ equipped
with a metric $h$ with \emph{strip-like ends} outside a compact subset $K_\Sigma$.
Let
$w: \dot \Sigma \to M$ be any smooth map with Legendrian boundary condition.
We define the total $\pi$-harmonic energy $E^\pi(w)$
by
\be\label{eq:pienergy}
E^\pi(w) = E^\pi_{(\lambda,J;\dot\Sigma,h)}(w) = \frac{1}{2} \int_{\dot \Sigma} |d^\pi w|^2
\ee
where the norm is taken in terms of the given metric $h$ on $\dot \Sigma$ and the triad metric on $M$.
\end{defn}

Throughout this section, we work locally near one interior puncture
$p$, i.e., on a punctured semi-disc
$D^\delta(p) \setminus \{p\}$. By taking the associated conformal coordinates $\phi^+ = (\tau,t)
:D^\delta(p) \setminus \{p\} \to [0, \infty)\times [0,1]$ such that $h = d\tau^2 + dt^2$,
we need only look at a map $w$ defined on the half cylinder
 $[0, \infty)\times S^1 \to Q$ without loss of generality.

We put the following hypotheses in our asymptotic study of the finite
energy contact instanton maps $w$ as in \cite{oh-wang:CR-map1}:
\begin{hypo}\label{hypo:basic-intro}
Let $h$ be the metric on $\dot \Sigma$ given above.
Assume $w:\dot\Sigma\to M$ satisfies the contact instanton equations \eqref{eq:contacton-Legendrian-bdy-intro},
and
\begin{enumerate}
\item $E^\pi_{(\lambda,J;\dot\Sigma,h)}(w)<\infty$ (finite $\pi$-energy);
\item $\|d w\|_{C^0(\dot\Sigma)} <\infty$.
\item $\Image w \subset \mathsf K\subset M$ for some compact subset $\mathsf K$.
\end{enumerate}
\end{hypo}

The above finite $\pi$-energy and $C^0$ bound hypotheses imply
\be\label{eq:hypo-basic-pt}
\int_{[0, \infty)\times S^1}|d^\pi w|^2 \, d\tau \, dt <\infty, \quad \|d w\|_{C^0([0, \infty)\times S^1)}<\infty
\ee
in these coordinates.

\begin{defn}[Asymptotic action and charge]
Assume that the limit of $w(\tau, \dot)$ as $\tau \to \infty$ exists.
Then we can associate two
natural asymptotic invariants at each puncture defined as
\beastar
T & := & \lim_{r \to \infty} \int_{\{r\}\times S^1} (w|_{\{0\}\times S^1 })^*\lambda \label{eq:TQ-T}\\
Q & : = & \lim_{r \to \infty} \int_{\{r\}\times S^1}((w|_{\{0\}\times [0,1] })^*\lambda\circ j).\label{eq:TQ-Q}
\eeastar
(Here we only look at positive punctures. The case of negative punctures is similar.)
We call $T$ the \emph{asymptotic contact action}
and $Q$ the \emph{asymptotic contact charge} of the contact instanton $w$ at the given puncture.
\end{defn}

The proof of the following subsequence convergence result largely
is proved in
\cite[Theorem 6.4]{oh-wang:CR-map1}.

\begin{thm}[Subsequence Convergence,
Theorem 6.4 \cite{oh-wang:CR-map1}]
\label{thm:subsequence}
Let $w:[0, \infty)\times S^1 \to M$ satisfy the contact instanton equations \eqref{eq:contacton-Legendrian-bdy-intro}
and Hypothesis \eqref{eq:hypo-basic-pt}.
Then for any sequence $s_k\to \infty$, there exists a subsequence, still denoted by $s_k$, and a
massless instanton $w_\infty(\tau,t)$ (i.e., $E^\pi(w_\infty) = 0$)
on the cylinder $\R \times [0,1]$  that satisfies the following:
\begin{enumerate}
\item $\delbar^\pi w_\infty = 0$ and
$$
\lim_{k\to \infty}w(s_k + \tau, t) = w_\infty(\tau,t)
$$
in the $C^l(K \times [0,1], M)$ sense for any $l$, where $K\subset [0,\infty)$ is an arbitrary compact set.
\item $w_\infty^*\lambda = -Q\, d\tau + T\, dt$
\end{enumerate}
\end{thm}

In general $Q  = 0$ does not necessarily hold for the closed string case.
When this $Q \neq 0$ combined with $T = 0$, we say $w$ has
the bad limit of \emph{appearance of spiraling instantons along the Reeb core}. It is also proven in \cite{oh:contacton} that If $Q = 0 = T$,
then the puncture is removable.

When $Q = 0$, which is always the case when contact instanton
is exact such as those arising from the symplectization case,
we have the following asymptotic convergence result.

\begin{cor} Assume that $\lambda$ is nondegenerate.
Suppose that $w_\tau$ converges as $|\tau| \to \infty$  and its massless
limit instanton has  $Q = 0$ but $T \neq 0$, then the $w_\tau$
converges to a Reeb orbit of period $|T|$ exponentially fast.
\end{cor}

\subsection{Open string case}

In this section, we study the asymptotic behavior of contact instantons
on bordered Riemann surface $(\dot\Sigma, j)$ associated with a metric $h$ with \emph{strip-like ends}.
To be precise, we assume there exists a compact set $K_\Sigma\subset \dot\Sigma$,
such that $\dot\Sigma-\Int(K_\Sigma)$ is a disjoint union of punctured semi-disks
 each of which is isometric to the half strip $[0, \infty)\times [0,1]$ or $(-\infty, 0]\times [0,1]$, where
the choice of positive or negative strips depends on the choice of analytic coordinates
at the punctures.

Again under the assumption that the limit $\lim_{\tau \to \infty} w(\tau, \dot)$ exists,
we can define the asymptotic action $T$ and charge $Q$
at each puncture in the way
as in the closed string case by replacing $S^1$ by $[0,1]$.

The following subsequence convergence and charge vanishing result
is proved in \cite{oh:contacton-Legendrian-bdy}, \cite{oh-yso:index}.
 One may say that \emph{the presence of Legendrian barrier
prevents the instanton from spiraling.}

\begin{thm}[Subsequence convergence and Charge vanishing]\label{thm:subsequence-open}
Let $w:[0, \infty)\times [0,1]\to M$ satisfy the contact instanton equations \eqref{eq:contacton-Legendrian-bdy-intro}
and converges $|\tau| \to \infty$.
Then for any sequence $s_k\to \infty$, there exists a subsequence, still denoted by $s_k$, and a
massless instanton $w_\infty(\tau,t)$ (i.e., $E^\pi(w_\infty) = 0$)
on the cylinder $\R \times [0,1]$  that satisfies the following:
\begin{enumerate}
\item $\delbar^\pi w_\infty = 0$ and
$$
\lim_{k\to \infty}w(s_k + \tau, t) = w_\infty(\tau,t)
$$
in the $C^l(K \times [0,1], M)$ sense for any $l$, where $K\subset [0,\infty)$ is an arbitrary compact set.
\item $w_\infty$ has vanishing asymptotic charge $Q = 0$ and satisfies $w_\infty(\tau,t)= \gamma(T\, t)$
for some Reeb chord $\gamma$ is some Reeb chord joining $R_0$ and $R_1$ with period $T$ at each puncture.
\item $T \neq 0$ at each  puncture with the associated pair $(R,R')$ of boundary condition with $R \cap R'
= \emptyset$.
\end{enumerate}
\end{thm}

\begin{rem}\label{rem:big-difference} This charge vanishing $Q = 0$
of the massless instanton is a huge advantage over
the closed string case studied in \cite{oh-wang:CR-map1,oh-wang:CR-map2,oh:contacton}.
We refer to \cite[Section 8.1]{oh:contacton} for the classification of massless instantons on the cylinder $\R \times S^1$ for which the kind of
massless instantons with $Q \neq 0$ but $T = 0$ appear
 or the closed string case. The \emph{appearance of spiraling
 instantons along the Reeb core} is the only obstacle towards the Fredholm theory and the compactification of
the moduli space of finite energy contact instantons for the general closed string context.

The above theorem automatically removes this obstacle for the open string case of Legendrian boundary condition.
One could say that the presence of the Legendrian obstacle
blocks this spiraling phenomenon of the contact instantons.
\end{rem}

From the previous theorem, we  immediately get the following corollary as in \cite[Section 8]{oh-wang:CR-map1}.

\begin{cor}[Corollary 5.11 \cite{oh-yso:index}]\label{cor:tangent-convergence}
Assume that the pair $(\lambda, \vec R)$ is nondegenerate
in the sense of Definition \ref{defn:nondegeneracy-links}.
Let $w: \dot \Sigma \to M$ satisfy the contact instanton equation \eqref{eq:contacton-Legendrian-bdy-intro} and Hypothesis \eqref{eq:hypo-basic-pt}.
Then on each strip-like end with strip-like coordinates $(\tau,t) \in [0,\infty) \times [0,1]$ near a puncture
\beastar
&&\lim_{s\to \infty}\left|\pi \frac{\del w}{\del\tau}(s+\tau, t)\right|=0, \quad
\lim_{s\to \infty}\left|\pi \frac{\del w}{\del t}(s+\tau, t)\right|=0\\
&&\lim_{s\to \infty}\lambda(\frac{\del w}{\del\tau})(s+\tau, t)=0, \quad
\lim_{s\to \infty}\lambda(\frac{\del w}{\del t})(s+\tau, t)=T
\eeastar
and
$$
\lim_{s\to \infty}|\nabla^l dw(s+\tau, t)|=0 \quad \text{for any}\quad l\geq 1.
$$
All the limits are uniform for $(\tau, t)$ in $K\times [0,1]$ with compact $K\subset \R$.
\end{cor}

\section{Off-shell energy of contact instantons}
\label{sec:lambda-energy}

Now following
\cite{oh:contacton}, \cite{oh:entanglement1}, we
explain the definition of off-shell energy that governs
the global convergence behavior of finite energy contact
instantons.

Assume that $\lambda$ for the closed string case (or $(\lambda, \vec R)$
for the open string case)  is nondegenerate.
We also assume that the asymptotic
charges vanish so that in cylindrical (or in strip-like) coordinates
$w^\lambda \to T\, dt$ and $w^*\lambda \circ J \to T\, d\tau$
exponentially fast.

We start with the $\pi$-energy
\begin{defn}[The $\pi$-energy of  contact instanton]
Let $u:\R \times [0,1] \rightarrow J^1B$ be any smooth map. We define
$$
E^\pi_{J,H} (u) := \frac{1}{2} \int  |d^\pi u |^2_J.
$$
\end{defn}

Next we borrow the presentation of $\lambda$-energy
from \cite[Section 5]{oh:contacton}, \cite[Section 11]{oh:entanglement1}.
We follow the procedure exercised in \cite{oh:contacton}
for the closed string case. We introduce the following class of test functions.
Especially the automatic charge vanishing in our current circumstance also enables us
to define the vertical part of energy, called the $\lambda$-energy whose definition is in order. We will focus on the more nontrivial open-string case
below.

\begin{defn}\label{defn:CC} We define
\be
\CC = \left\{\varphi: \R \to \R_{\geq 0} \, \Big| \, \supp \varphi \, \text{is compact}, \, \int_\R \varphi = 1\right\}
\ee
\end{defn}

\begin{defn}[Contact instanton potential] We call the above function $f$
the \emph{contact instanton potential} of the contact instanton charge form $w^*\lambda \circ j = df$ on $\dot \Sigma$.
\end{defn}

Such a function exists modulo addition by a constant.
 \emph{Using the assumption of  vanishing of
asymptotic charge}, we can explicitly write the potential as
\be\label{eq:f-defn}
f(z) = \int_{+\infty}^z w^*\lambda \circ j
\ee
where the integral is over any path from any puncture
denote by $\infty$ to $z$ along a path $\dot \Sigma$.
By the closedness of $w^*\lambda \circ j$ on $\dot \Sigma$,
the function $f$ is well-defined on $\dot \Sigma$ and satisfies
$w^*\lambda \circ j = df$ \emph{as long as
$H_1(\dot \Sigma, \Z) = 0$.}
(Compare this with  \cite[Formula above (5.5)]{oh:contacton}
where the general case with nontrivial charge is considered.)
By the vanishing theorem, Theorem \ref{thm:subsequence-open}
of the asymptotic charge
of bordered contact instantons,
the local version of potential function
 is always well-defined locally near punctures  for the bordered contact instantons

We denote by $\psi$ the function determined by
\be\label{eq:psi}
\psi' = \varphi, \quad \psi(-\infty) = 0, \, \psi(\infty) = 1.
\ee
\begin{defn}\label{defn:CC-energy} Let $w$ satisfy $d(w^*\lambda \circ j) = 0$. Then we define
\beastar
E_{\CC}(j,w;p) & = & \sup_{\varphi \in \CC} \int_{D_\delta(p) \setminus \{p\}} df\circ j \wedge d(\psi(f)) \\
& = &\sup_{\varphi \in \CC} \int_{D_\delta(p) \setminus \{p\}}  (- w^*\lambda ) \wedge d(\psi(f)).
\eeastar
\end{defn}

We note that
$$
df \circ j \wedge d(\psi(f)) = \psi'(f) df\circ j \wedge df = \varphi(f) df\circ j  \wedge df \geq 0
$$
since
$$
df\circ j  \wedge df = |df|^2\, d\tau \wedge dt.
$$
Therefore we can rewrite $E_{\CC}(j,w;p)$ into
$$
E_{\CC}(j,w;p) = \sup_{\varphi \in \CC} \int_{D_\delta(p) \setminus \{p\}} \varphi(f) df \circ j \wedge df.
$$
The following proposition shows that  the definition of $E_{\CC}(j,w;p)$ does not
depend on the constant shift in the choice of $f$.

\begin{prop}[Proposition 11.6 \cite{oh:entanglement1}]\label{prop:a-independent}
For a given smooth map $w$ satisfying $d(w^*\lambda \circ j) = 0$,
we have $E_{\CC;f}(w) = E_{\CC,g}(w)$ for any pair $(f,g)$ with
$$
df = w^*\lambda\circ j = dg
$$
on $D^2_\delta(p) \setminus \{p\}$.
\end{prop}

This proposition enables us to introduce the following vertical energy where we write $E_\pm^\lambda: = E_{\pm \infty}^\lambda$
on $\R \times [0,1] \cong D^2 \setminus \{\pm 1\}$.

\begin{defn}[Vertical energy]
We define the \emph{vertical energy}, denoted by $E^\perp(w)$, to be the sum
$$
E^\perp(w) = E^\lambda_+(w) + E^\lambda_-(w)
$$
\end{defn}

Now we define the final form of the off-shell energy.
\begin{defn}[Total energy]\label{defn:total-enerty}
Let $w:\dot \Sigma \to Q$ be any smooth map. We define the \emph{total energy} to be
the sum
\be\label{eq:final-total-energy}
E(w) = E^\pi(w) + E^\perp(w).
\ee
\end{defn}

\begin{rem}[Uniform $C^1$ bound]
The upshot  is that the Sachs-Uhlenbeck \cite{sacks-uhlen}, Gromov \cite{gromov:invent} and Hofer
\cite{hofer:invent} style bubbling-off analysis
can be carried out with this choice of energy. (See \cite{oh:contacton,oh:entanglement1} for
the details of this bubbling-off analysis.) In particular
\emph{all moduli spaces of finite energy perturbed contact instantons we consider in the present paper will
have uniform $C^1$-bounds inside each given moduli spaces.}
\end{rem}

\section{Exponential $C^\infty$ convergence}
\label{sec:exponential-convergence}

In this section, we outline the main steps
for proving  the exponential convergence result from
\cite[Section 11]{oh-wang:CR-map2} for the closed string case
and from \cite[Section 6]{oh-yso:index}
for the open string case, respectively.

Under the nondegeneracy hypothesis from Definition
\ref{defn:nondegeneracy-links}
we can improve the subsequence convergence to the exponential $C^\infty$ convergence
under the transversality hypothesis. Suppose that the tuple $\vec R= (R_0, \ldots, R_k)$ are
transversal in the sense all pairwise Reeb chords are nondegenerate. In particular we assume
that the tuples are pairwise disjoint. The proof is divided into several steps.

\subsection{$L^2$-exponential decay of the Reeb component of $dw$}
\label{subsec:L2-expdecay}

We will prove the exponential decay of the Reeb component $w^*\lambda$.
We focus on a punctured neighborhood around a puncture $z_i \in \del \Sigma$ equipped with
strip-like coordinates $(\tau,t) \in [0,\infty) \times [0,1]$.

We again consider a complex-valued function $\alpha$ given in \eqref{eq:alpha}.
Then by the Legendrian boundary condition, we know $\alpha(\tau, i) \in \R$, i.e.
$$
\text{\rm Im}\, \alpha = 0
$$
for $i =0, \, 1$.

The following lemma was proved in the closed string case in
\cite{oh-wang:CR-map2} (in the more general context of Morse-Bott case).
For readers' convenience, we
provide some details by indicating how we adapt the argument
with the presence of boundary condition.

\begin{lem}[Lemma 6.1 \cite{oh-yso:index}; Compare with Lemma 11.20 \cite{oh-wang:CR-map2}]\label{lem:exp-decay-lemma}
Suppose the complex-valued functions $\alpha$ and $\nu$ defined on $[0, \infty)\times [0,1]$
satisfy
\beastar
\begin{cases}
\delbar \alpha = \nu, \\
\alpha(\tau, i) \in \R \,   \text{\rm for } i = 0,\, 1, \\
\|\nu\|_{L^2([0,1])}+\left\|\nabla\nu\right\|_{L^2([0,1])}\leq Ce^{-\delta \tau}
\, & \text{\rm for some constants } C, \delta>0\\
\lim_{\tau\rightarrow +\infty}\alpha(\tau,t) = T
\end{cases}
\eeastar
then $\|\alpha -T\|_{L^2(S^1)}\leq \overline{C}e^{-\delta \tau}$
for some constant $\overline{C}$.
\end{lem}
\begin{proof}[Outline of proof]
Notice that from previous section we have already
established the $W^{1, 2}$-exponential decay of $\nu = \frac{1}{2}|\zeta|^2$.
Once this is established, the proof of this $L^2$-exponential decay result
is proved by the standard three-interval method.
See  Theorem \ref{thm:three-interval} for a general abstract framework which is  in Appendix \ref{sec:three-interval} of the present paper.  We also refer to
the Appendix of the arXiv version of \cite{oh-wang:CR-map1} for friendly details for the  current nondegenerate case.
\end{proof}
(See Remark \ref{rem:alpha} for the difference of the definitions of $\alpha$ here and in Lemma 11.20 \cite{oh-wang:CR-map2}.)
\subsection{$C^0$ exponential convergence}

Now the $C^0$-exponential convergence of $w(\tau,\cdot)$ to some Reeb chord as $\tau \to \infty$
can be proved from the $L^2$-exponential estimates presented in previous sections by
the verbatim same argument as the proof of \cite[Proposition 11.21]{oh-wang:CR-map2}
with $S^1$ replaced by $[0,1]$ here. Therefore we omit its proof.

\begin{prop}[ Proposition 11.21 \cite{oh-wang:CR-map2}, Proposition 6.3 \cite{oh-yso:index}]\label{prop:czero-convergence}
Under Hypothesis \ref{hypo:basic-intro}, for any contact instanton $w$ satisfying the Legendrian boundary condition,
there exists a unique Reeb orbit $\gamma$ such that the curve  $z(\cdot)=\gamma(T\cdot):[0,1] \to M$  satisfies
$$
\|d(w(\tau, \cdot), z(\cdot))
\|_{C^0([0,1])}\rightarrow 0,
$$
as $\tau\rightarrow +\infty$,
where $d$ denotes the distance on $M$ defined by the triad metric.
Here $T = T_\gamma $ is action of $\gamma$ given by  $T_\gamma =\int \gamma^*\lambda$.
\end{prop}

Then the following $C^0$-exponential convergence is also proved.
\begin{prop}[Proposition 11.23 \cite{oh-wang:CR-map2}, Proposition 6.5 \cite{oh-yso:index}]
There exist some constants $C>0$, $\delta>0$ and $\tau_0$ large such that for any $\tau>\tau_0$,
\beastar
\|d\left( w(\tau, \cdot), z(\cdot) \right) \|_{C^0([0,1])} &\leq& C\, e^{-\delta \tau}
\eeastar
\end{prop}

\subsection{$C^\infty$-exponential decay of $dw - R_\lambda(w) \, d\tau$}
\label{subsec:Cinftydecaydu}

 So far, we have established the following:
\begin{itemize}
\item $W^{1,2}$-exponential decay of $w$,
\item $C^0$-exponential convergence of $w(\tau,\cdot) \to z(\cdot)$ as $\tau \to \infty$
for some Reeb chord $z$ between two Legendrians $R, \, R'$.
\end{itemize}

Now we are ready to complete the proof of
$C^\infty$-exponential convergence $w(\tau,\cdot) \to z$
by establishing the $C^\infty$-exponential decay of $dw - R_\lambda(w)\, dt$.
The proof of the latter decay is now in order which will be
carried out by the bootstrapping arguments
applied to the system \eqref{eq:contacton-Legendrian-bdy-intro}.

Combining the above three, we have obtained $L^2$-exponential estimates of the full derivative $dw$.
By the bootstrapping argument using the local uniform a priori estimates
on the strip-like region as in the proof of Lemma \ref{lem:exp-decay-lemma},
we obtain higher order $C^{k,\alpha}$-exponential decays of the term
$$
\frac{\del w}{\del t} - T R_\lambda(z), \quad \frac{\del w}{\del\tau}
$$
for all $k\geq 0$, where $w(\tau,\cdot)$ converges to $z$ as $\tau \to \infty$ in $C^0$ sense.
Combining this, Lemma \ref{lem:exp-decay-lemma} and elliptic $C^{k,\alpha}$-estimates
given in Theorem \ref{thm:local-regularity}, we complete the proof of $C^\infty$-convergence of
$w(\tau,\cdot) \to z$ as $\tau \to \infty$.

\part{Compactification, Fredholm theory and asymptotic analysis}

In this section, we develop the Fredholm theory of moduli space of pseudoholomorphic curves on the
symplectization as a special case of the theory of pseudoholomorphic curves on the $\mathfrak{lcs}$-fication
of contact manifolds as done in by the first author with Savelyev \cite{oh-savelyev}. The
symplectization is just the zero-temperature limit of the $\mathfrak{lcs}$-fications. Contact triad connection and its symplectification
interact very well with the decomposition
\beastar
TQ & = & \xi \oplus \R \langle R_\lambda \rangle\\
TM & \cong & \xi \oplus \R \langle R_\lambda \rangle
\oplus \R \left \langle \frac{\del}{\del s} \right\rangle
\eeastar
which enable us to derive a tensorial formula
for the linearized operator in a precise matrix form
all whose summands carry natural geometric meaning.
This enables us to do complete coordinate free
asymptotic analysis. On the other hand,
\cite{HWZ:asymptotics,HWZ:asymptotics-correction,HWZ:smallarea}
rely on the choice of special coordinates in 3 dimension.
Such a coordinate approach, especially in higher dimensions,
leads to many complicated tedious expressions in the asymptotic
convergence results. This was the starting point of Wang
and the first author of the present survey for them to develop
the notion of contact instantons and its tensorial study and to
have discovered the notion of \emph{contact triad connection}
which well suits the purpose. This approach especially leads to
an simple transparent formulae for the linearized operators both for the
contact instantons and for the pseudoholomorphic curves
in symplectization.

Leaving the case of contact instantons to \cite{oh:contacton} for the
closed string case and to \cite{oh:contacton-transversality} for
the open string case, we now focus on the case of pseudoholomorphic
curves  in symplectization.

We start with the following completion of exponential convergence
result for the finite energy punctured pseudoholomorphic curves.

\section{Exponential convergence in symplectization}

In this section, we consider the symplectization
$$
M = Q \times \R, \quad \omega = d(e^s \pi^*\lambda) = e^s (ds \wedge \pi^*\lambda + d\pi^*\lambda)
$$
of the contact manifold $(Q,\xi)$ equipped with contact form $\lambda$.

On $Q$, the Reeb vector field $R_\lambda$ associated to the contact
form $\lambda$ is the unique vector field satisfying
\be\label{eq:Liouville} X \rfloor \lambda = 1, \quad X \rfloor
d\lambda = 0.
\ee
We call $(y,s)$ the cylindrical coordinates.
On the cylinder $[0, \infty) \times Q \subset (-\infty,\infty) \times Q$,
we have the natural splitting
$$
TM \cong TQ  \oplus \R\cdot \frac{\del}{\del s} \cong \xi \oplus \span\left\{\widetilde R_\lambda,
\frac{\del}{\del s}\right\}  \cong \xi \oplus \R^2.
$$
We denote by $\widetilde R_\lambda$ the unique vector field on $[0,\infty) \times Q$ which is
invariant under the translation, tangent to the level sets of $s$
and projected to $R_\lambda$. When there is no danger of confusion, we will
sometimes just denote it by $R_\lambda$.

Now we describe a special family of almost complex structure adapted to
the given cylindrical structure of $M$.

\begin{defn} An almost complex structure $J$ on $(0,\infty) \times Q$ is called
\emph{$\lambda$-adapted} if it is split into
$$
J = J_\xi \oplus J_0: TM \cong \xi \oplus \R^2  \to TM \cong \xi \oplus \R^2
$$
where $J|_\xi$ is compatible to $d\lambda|_{\xi}$ and $j: \R^2 \to \R^2$
maps $\frac{\del}{\del s}$ to $R_\lambda$.
\end{defn}

For our purpose, we will need to consider a family of symplectic forms
to which the given $J$ is compatible and their associated metrics.
For any $\lambda$-adapted $J$, the $J$-compatible metric associated
to $\omega$ is expressed as
\be\label{eq:gJ}
g_{(\omega,J)}  = e^s(ds^2 + g_Q)
\ee
on $\R \times Q$.

Now we regard the triple $(\omega,J,g_{(\omega,J)})$ be an almost
K\"ahler manifold near the level surface $s = 1$. We then fix the canonical
connection $\nabla$ associated to $(\omega,J,g_{(\omega,J)})$. The following
is a general property of the canonical connection.

\begin{prop}\label{prop:TJYY} Let $(W,\omega, J)$ be an almost K\"ahler manifold
and $\nabla$ be the canonical connection. Denote by $T$ be its torsion tensor. Then
\be\label{eq:TJYY}
T(JY,Y) = 0
\ee
for all vector fields $Y$ on $W$.
\end{prop}
Consider the decomposition
$$
TM \cong \xi \oplus \R\{R_\lambda\} \oplus \R \left\{\frac{\del}{\del r}\right\}
$$
and the canonical connection $\widetilde \nabla$ on $Q \times \R$, which in particular
is $J$-linear.

Now denote $f = s\circ u$ and $w = \Theta \circ u$, i.e., $u = (w,a)$ in $Q \times \R$
be $J$-holomorphic. Then we have already shown the exponential convergence $w$ to
the Reeb orbit $w(\cdot, t)$. For the convergence, we use the equation
$$
w^*\lambda \circ j = df
$$
which implies
$$
w^*\lambda = - df\circ j
$$
By taking differential, we obtain
$$
-d(df\circ j) = w^*d\lambda = \frac12 |d^\pi w|^2 \, d\tau \wedge dt.
$$
This first shows that $f$ is a subharmonic function and
satisfies
\be\label{eq:a-asymptotics}
\frac{\del^2 f}{\del \tau^2} + \frac{\del^2 f}{\del t^2} - \frac12 |d^\pi w|^2 = 0
\ee
where we know  from Theorem \ref{thm:subsequence} that the convergence $\frac12 |d^\pi w|^2 \to T^2$ is exponentially fast.
This immediately
gives rise to the following exponential convergence of the radial component
which will complete the study of asymptotic convergence property of
finite energy pseudoholomorphic planes in symplectization.
(See Theorem \ref{thm:three-interval} for the general abstract framework of
establishing exponential convergence.)

\begin{prop}[Exponential convergence of radial component] Let
$u=(w,f)$ be a finite energy $\widetilde J$-holomorphic plane in
$Q\times \R$. Then we have convergence
$$
df \to T\, d\tau
$$
exponentially fast.
\end{prop}
\begin{proof} We have already established $w^*\lambda \to T\, dt$ before.
By composing by $j$, the statement follows.
\end{proof}

\section{The moduli spaces of contact instantons and of pseudoholomorphic curves}

In this section, we recall the definitions of the moduli spaces of contact instantons and of
pseudoholomorphic curves and compare their compactifications.

\subsection{Moudli space of pseudoholomorphic curves on symplectization}

Let $(\dot \Sigma, j)$ be a punctured Riemann surface and let
$$
p_1, \cdots, p_{s^+}, q_1, \cdots, q_{s^-}
$$
be the positive and negative punctures. For each $p_i$ (resp. $q_j$), we associate the isothermal
coordinates $(\tau,t) \in [0,\infty) \times S^1$ (resp. $(\tau,t) \in (-\infty,0] \times S^1$)
on the punctured disc $D_{e^{-2\pi R_0}}(p_i) \setminus \{p_i\}$
(resp. on $D_{e^{-2\pi R_0}}(q_i) \setminus \{q_i\}$) for some sufficiently large $R_0 > 0$.

Following \cite{hofer:invent}, \cite{behwz},
we define the associated energy
$$
E_\eta(u) = E^\pi(u) + E^\perp_\eta(u)
$$
for each smooth map $u = (w,f)$ in class $\eta$, i.e., $[u]_{S^1} = \eta$.
\begin{defn} We define
$$
\widetilde{\CM}_{k,\ell} (\dot \Sigma, M;J), \quad k, \, \ell \geq 0, \, k+\ell \geq 1
$$
to be the moduli space of  pseudoholomorphic curve
$u = (w,f)$ with $E_\eta(u) < \infty$.
\end{defn}
Then we have a decomposition
$$
\widetilde{\CM}_{k,\ell} (\dot \Sigma, M;J)
= \bigcup_{\vec \gamma^+,\vec \gamma^-} \widetilde{\CM}_{k,\ell}
\left(\dot \Sigma, M;J;(\vec \gamma^+,\vec \gamma^-)\right)
$$
by Theorem \ref{thm:subsequence} where
\beastar
\widetilde{\CM}_{k,\ell} (\dot \Sigma, M;J;(\vec \gamma^+,\vec \gamma^-))
& = &\{u = (w,f) \mid u \, \text{is an lcs instanton with } \\
&{}& \quad E_\eta(u) < \infty, \, w(-\infty_j) = \gamma^-_j, \, w(\infty_i) = \gamma_i \}.
\eeastar
Here we have
the collections of Reeb orbits $\gamma^+_i$ and $\gamma^-_j$
and of points $\theta^+_i$, $\theta^-_j$ for $i =1, \cdots, s^+$
and for $j = 1, \cdots, s^-$ respectively such that
\be\label{eq:limatinfty}
\lim_{\tau \to \infty}w((\tau,t)_i) = \gamma^+_i(T_i(t+t_i)), \quad
\lim_{\tau \to - \infty}w((\tau,t)_j) = \gamma^-_j(T_j(t-t_j))
\ee
for some $t_i, \, t_j \in S^1$, where
$$
T_i = \int_{S^1} (\gamma^+_i)^*\lambda, \, T_j = \int_{S^1} (\gamma^-_j)^*\lambda.
$$
Here $t_i,\, t_j$ depends on the given analytic coordinate and the parameterization of the Reeb orbits.

Due to the $\R$-equivariance of the equation \eqref{eq:tildeJ-holo}
under the $\R$ action of translations, this action induces a free action on
$\widetilde{\CM}_{k,\ell} (\dot \Sigma, M;J)$. Then we denote
\be\label{eq:CMkell}
\CM_{k,\ell} (\dot \Sigma, M;J) = \widetilde{\CM}_{k,\ell}(\dot \Sigma, M;J)/\R.
\ee
We also have the decomposition
$$
\widetilde{\CM}_{k,\ell} (\dot \Sigma, M;J) = \bigcup_{\vec \gamma^\pm}
\widetilde{\CM}_{k,\ell} (\dot \Sigma, M;J;(\vec \gamma^+,\vec \gamma^-)).
$$
Here we denote
$
\vec \gamma^+ = (\gamma^+_i), \quad \vec \gamma^- = (\gamma^-_j)
$
By quotienting the above out by the $\R$-action,
we obtain the moduli space of $J$-holomorphic curves
$$
\CM_{k,\ell} (\dot \Sigma, M;J) = \bigcup_{\vec \gamma^\pm}
\CM_{k,\ell} (\dot \Sigma, M;J;(\vec \gamma^+,\vec \gamma^-)).
$$
We now introduce a uniform energy bound for $u = (w,f) $ with given asymptotic condition at its punctures.  Recall that they satisfy
$w^*\lambda \circ j = df$.

The following proposition is the analog to \cite[Lemma 5.15]{behwz} and
\cite[Proposition 9.2]{oh:contacton}  whose proof is also similar.

\begin{prop}\label{prop:proper-energy} Let
$u = (w,f) \in \CM_{k,\ell} (\dot \Sigma, M;J;(\vec \gamma^+,\vec \gamma^-))$.
Suppose that $E^\pi(w) < \infty$ and the function $f:\dot \Sigma \to \R$ is proper.
Then $E(w) < \infty$.
\end{prop}

The following a priori energy bounds for $u = (w,f)$  is proved in
\cite{behwz}.

\begin{prop}[Lemma 5.15, \cite{behwz}]\label{prop:uniform-energy-bound}
Let $u = (w,f) \in \CM_{k,\ell} (\dot \Sigma, M;J;(\vec \gamma^+,
\vec \gamma^-))$.
Let $\widetilde f: \dot \Sigma \to \R$ be the function satisfying $w^*\lambda \circ j = d\widetilde f$.
Suppose the function $\widetilde f$ is proper. Then we have
\beastar
E^\pi(u) & = &  E^\pi(w) = \sum_{i=1}^k \CA_\lambda(\gamma^+_i) - \sum_{j=1}^\ell \CA_\lambda(\gamma^-_j)\\
E^\perp_\eta (u)& = & \sum_{j=1}^k \CA_\lambda(\gamma^+_i)  \\
E(u) & = & 2  \sum_{i=1}^k \CA_\lambda(\gamma^+_i) - \sum_{j=1}^\ell \CA_\lambda(\gamma^-_j).
\eeastar
\end{prop}

\subsection{Moduli space of contact instantons with prescribed charge}
\label{subsec:prescribed-charge}

We first consider the closed string case. In this case, the asymptotic charge
$Q(p)$ at some interior puncture $p$ may not vanish in general although they
always vanish at the boundary punctures. In \cite{oh-savelyev}, the authors
introduced the notion of \emph{charge class} $\eta \in H^1(\dot \Sigma,\Z)$ and
defined the moduli space of pseudoholomorphic curves on the $\mathfrak{lcs}$-fication
$$
(Q \times S^1, \omega_\lambda), \quad \omega_\lambda = d\lambda + d\theta \wedge \lambda
$$
for the canonical angular form $d\theta$ on $S^1$ which is the generator of $H^1(S^1,\Z)$
similarly as in the case of symplectization or more precisely in the zero-temperature limit
$\mathfrak{lcs}$-fication
$$
(Q \times \R, \omega_\lambda), \quad \omega_\lambda = d\lambda + ds \wedge \lambda.
$$
They lift a $\lambda$-adapted CR almost complex structure to an almost complex structure on $Q \times S^1$ by requiring
$$
\widetilde J|_\xi = J|_\xi, \quad \widetilde J\left(\frac{\del}{\del \theta}\right) = R_\lambda.
$$
Similarly as in the case of symplectization, they showed that a $\widetilde J$-holomorphic curve $u = (w,f)$ satisfies
\be\label{eq:lcs-instanton}
\delbar^\pi w = 0, \quad w^*\lambda \circ j = df
\ee
where $f: \dot \Sigma \to S^1$ given by $f: = \theta \circ u$. By definition, each lcs instanton
carries a cohomology class $\eta_w: = [w^*\lambda \circ j] \in H^1(\dot \Sigma, \Z)$.

Recalling the isomorphism
$$
[\dot \Sigma,S^1] \cong H^1(\dot \Sigma,\Z)
$$
we may also regard the cohomology class $[u]_{S^1}$ as an element in $[\dot \Sigma,S^1]$.
This enables us to define an element in the set of homotopy classes $[\dot \Sigma,S^1]$, which
we also denote by $\eta = \eta_u$. In fact, the isomorphism
$[\dot \Sigma, S^1] \cong H^1(\dot \Sigma;\Z)$
is directly induced by the period map
$$
[f] \mapsto [f^*d\theta].
$$

\begin{defn}[Period map and the charge class; Definition 5.5 \cite{oh-savelyev}]
Let $f: \dot \Sigma \to Q \times S^1$ be a smooth map.
\begin{enumerate}
\item We call the map
$$
C^\infty(\dot \Sigma,Q) \to H^1(\dot \Sigma,\Z); \quad f \mapsto [f^*d\theta]
$$
the \emph{period map} and call the cohomology class $[f^*d\theta]$
the \emph{charge class} of the map $f$.
\item For an $\text{\rm lcs}$ instanton $u = (w,f): \dot \Sigma \to Q \times S^1$, we call the cohomology class
$$
\left[f^*d\theta\right] \in H^1(\dot \Sigma,\Z)
$$
the \emph{charge class} of $u$ and write
$$
[u]_{S^1}: = [f^*d\theta].
$$
\end{enumerate}
\end{defn}

Now we consider the maps $u = (w,f)$ with a fixed charge class $\eta = \eta_u$ and define
the moduli space
\beastar
\widetilde{\CM}_{k,\ell}^\eta(\dot \Sigma, Q;J;(\vec \gamma^+,\vec \gamma^-))
& = &\{u = (w,f) \mid u \, \text{\rm is an $\lcs$ instanton with} \\
&{}& \quad  E_\eta(u) < \infty, \, w(-\infty_j) = \gamma^-_j, \, w(\infty_i) = \gamma_i \}.
\eeastar

For the open string case,  as proven in Theorem \ref{thm:subsequence-open}
the (local) charge near every boundary puncture vanishes. This shows that
the closed one-form $w^*\lambda\circ j$ defines a cohomology class
\be\label{eq:chargeclass-open}
[w^*\lambda \circ j] \in H^1(\overline \Sigma, \del_\infty \dot\Sigma); \Z).
\ee
\begin{defn}[The charge class; open string case] Let
$w: (\dot \Sigma, \del \dot \Sigma) \to (Q, \vec R)$ be a
bordered contact instanton. We call the cohomology class
$[w^*\lambda \circ j]$ given in \eqref{eq:chargeclass-open}
the charge class of $w$.
\end{defn}

The following is an immediate consequence of this definition and
Theorem \ref{thm:subsequence-open}.

\begin{cor} Suppose $\Sigma$ is an open Riemann surface of
genus 0. Then for any bordered contact instanton on
$\dot \Sigma = \Sigma \setminus \{z_0, \cdots z_k\}$, the
charge class vanishes.
\end{cor}

\subsection{Comparison of compactifications of the two moduli spaces}
\label{subsec:comparison}

Now we consider contact instantons $w$ arising from a pseudoholomorphic curves on symplectization $(w,f)$. In particular all such $w$ has
its the charge class
$[w^*\lambda \circ j] = 0$ in $H^1(\dot \Sigma, \Z)$ but can be lifted to
$H^1(\overline{\Sigma}, \del_\infty \dot \Sigma)$ where $\overline\Sigma$
is the real blow-up of $\dot \Sigma$ along the punctures. We denote the
moduli space of such contact instantons of finite energy by
$$
\widetilde \CM^{\text{\rm exact}}(\dot \Sigma, Q,J;\vec \gamma^-, \vec \gamma^+)
$$
and
$$
\CM^{\text{\rm exact}}(\dot \Sigma, Q,J;\vec \gamma^-, \vec \gamma^+)
: = \widetilde \CM^{\text{\rm exact}}(\dot \Sigma, Q,J;\vec \gamma^-, \vec \gamma^+)/\text{\rm Aut}(\dot \Sigma)
$$
the set of isomorphism classes thereof. We have
natural forgetful map $(w,f) \mapsto w$ which descends to
$$
\mathfrak{forget}: \CM(\dot \Sigma, M,\widetilde J;\vec \gamma^-, \vec \gamma^+)
\to \CM^{\text{\rm exact}}(\dot \Sigma, Q,J;\vec \gamma^-, \vec \gamma^+).
$$
By definition of the equivalence relation on $\widetilde \CM(\dot \Sigma, M,\widetilde J;\vec \gamma^-, \vec \gamma^+)$,
it follows that this forgetful map is a bijective correspondence, \emph{provided $\dot \Sigma$ is
connected}.

However when one considers the SFT compactification as in \cite{EGH}, \cite{behwz}, one needs to
consider the case of pseudoholomorphic curves with disconnected domain.
So let us consider such cases. Suppose we have the union
$$
\dot \Sigma = \bigsqcup_{i=1}^k \dot \Sigma_i
$$
of connected components with $k \geq 2$. We denote by
$$
\overline \CM(\dot \Sigma, M,\widetilde J;\vec \gamma^-, \vec \gamma^+)
$$
and
$$
\overline \CM^{\text{\rm exact}}(\dot \Sigma, Q,J;\vec \gamma^-, \vec \gamma^+)
$$
the stable map compactification respectively. The following proposition shows
the precise relationship between the two. We know that each story carries at least
one non-cylindrical component.

\begin{prop}\label{prop:Rk-1fibration}
Let $1 \leq \ell \leq k$ be the number of connected components which are
not cylinderical. The forgetful map $\mathfrak{forget}$ is a
principle $\R^{\ell-1}$ fibration.
\end{prop}
\begin{proof} Recall the equivalence relation on $\widetilde \CM(\dot \Sigma, M,\widetilde J;\vec \gamma^-, \vec \gamma^+)$:
We say two elements $(u_1, \cdots, u_k) \sim (u_1', \cdots, u_k')$ if there is a $s_0 \in \R$ and a reparameterization
$\varphi = (\varphi_1,\cdots, \varphi_k)$ of $\dot \Sigma$ such that
$$
(w_i',f_i') = (w\circ \varphi_i, f_i\circ \varphi_i - s_0)
$$
for all $i = 1, \cdots, k$. We define the map
$$
T_{\vec s}: \widetilde  \CM(\dot \Sigma, M,\widetilde J;\vec \gamma^-, \vec \gamma^+)\to
\widetilde \CM(\dot \Sigma, M,\widetilde J;\vec \gamma^-, \vec \gamma^+)
$$
by
$$
T_{\vec s}(u_1, \cdots,u_k) = (T_{s_1}u_1, \cdots, T_{s_k}u_k)
$$
where $T_{s_i}u_i = T_{s_i}(w_i,f_i) : = (w_i,f_i - s_i)$. Note that this action is trivial
for each trivial component.

Then all the following maps
$$
T_{\vec s}(u) = (T_{s_1}u_1, \cdots, T_{s_k}u_k)
$$
project down to the same element
$[(w_1,\cdots, w_k)] \in \CM^{\text{\rm exact}}(\dot \Sigma, Q,J;\vec \gamma^-, \vec \gamma^+)$.
The abelian group $\R^{\ell-1}$
admits a free transitive action of $\R^{\ell -1}$ on each fiber
$$
\mathfrak{forget}^{-1}([(w_1,\cdots, w_k)])
$$
given by
$$
([(u_1,\cdots, u_k)], (s_1,\cdots s_k))\mapsto ([(T_{s_1}u_1, \cdots, T_{s_k}u_k)]
$$
where we realize
$$
\R^{\ell-1} \cong \{(s_1, \cdots, s_\ell)\mid s_1 + \cdots + s_\ell = 0\}
$$
This finishes the proof.
\end{proof}

\section{Fredholm theory and index calculations}
\label{sec:Fredholm}

In this section, we work out the Fredholm theories of
pseudoholomorphic curves on symplectization. We will adapt the
exposition given in \cite{oh:contacton} \cite{oh-savelyev} for the case of contact instantons to that of pseudoholomorphic curves thereon as the zero-temperature limit
lcs instantons considered in \cite{oh-savelyev} just
by incorporating the presence of the $\R$-factor
in the product $M = Q^{2n-1} \times \R$.

We divide our discussion into the closed case and the punctured case.

\subsection{Calculation of the linearization map}
\label{subsec:linearization}

Let $\Sigma$ be a closed Riemann surface and $\dot \Sigma$ be its
associated punctured Riemann surface.  We allow the set of whose punctures
to be empty, i.e., $\dot \Sigma = \Sigma$.
We would like to regard the assignment $u \mapsto \delbar_J u$ which can be
decomposed into
$$
u= (w,f) \mapsto \left(\delbar^\pi w, w^*\lambda \circ j - f^*ds\right)
$$
for a map $w: \dot \Sigma \to Q$ as a section of the (infinite dimensional) vector bundle
over the space of maps of $w$. In this section, we lay out the precise relevant off-shell framework
of functional analysis.

Let $(\dot \Sigma, j)$ be a punctured Riemann surface, the set of whose punctures
may be empty, i.e., $\dot \Sigma = \Sigma$ is either a closed or a punctured
Riemann surface. We will fix $j$ and its associated K\"ahler metric $h$.

We consider the map
$$
\Upsilon(w,f) = \left(\delbar^\pi w, w^*\lambda \circ j - f^*ds \right)
$$
which defines a section of the vector bundle
$$
\CH \to \CF = C^\infty(\Sigma,Q)
$$
whose fiber at $u \in C^\infty(\Sigma,Q \times \R)$ is given by
$$
\CH_u: = \Omega^{(0,1)}(u^*\xi) \oplus \Omega^{(0,1)}(u^*\CV).
$$
Recalling $\CV_{(q,s)} = \span_\R\{R_\lambda, \frac{\del}{\del s}\}$,
we have a natural isomorphism
$$
\Omega^{(0,1)}(u^*\CV) \cong \Omega^1(\Sigma) = \Gamma(T^*\R)
$$
via the map
$$
\alpha \in \Gamma(T^*\R) \mapsto \frac12\left(\alpha \cdot \frac{\del}{\del s}
+ \alpha \circ j \cdot R_\lambda\right).
$$
Utilizing this isomorphism,
we decompose $\Upsilon = (\Upsilon_1,\Upsilon_2)$ where
\be\label{eq:upsilon1}
\Upsilon_1: \Omega^0(w^*TQ) \to \Omega^{(0,1)}(w^*\xi); \quad \Upsilon_1(w) = \delbar^\pi(w)
\ee
and
\be\label{eq:upsilon2}
\Upsilon_2: \Omega^0(w^*TQ) \to \Omega^1(\Sigma); \quad \Upsilon_2(w) = w^*\lambda \circ j - f^*ds.
\ee

We first compute the linearization map which defines a linear map
$$
D\Upsilon(u) : \Omega^0(w^*TQ) \oplus \Omega^0(f^*T\R) \to \Omega^{(0,1)}(w^*\xi) \oplus \Omega^1(\dot \Sigma)
$$
where we have
$$
T_u \CF = \Omega^0(w^*TQ) \oplus \Omega^0(f^*T\R).
$$

For the optimal expression of the linearization map and its relevant
calculations, we use the $\mathfrak{lcs}$-fication connection $\nabla$ of $(Q \times \R,\lambda,J)$
which is the lcs-lifting of contact triad connection introduced in \cite{oh-wang:CR-map1}.
We refer readers to \cite{oh-wang:CR-map1}, \cite{oh:contacton},
\cite{oh-savelyev} for the unexplained notations
appearing in our tensor calculations during the proof.

We define the covariant differential
$$
\delbar^{\nabla^\pi} := \frac12\left(\nabla^\pi + J \nabla^\pi \circ j\right).
$$

\begin{thm}[Theorem 10.1 \cite{oh-savelyev}] \label{thm:linearization} We decompose $d\pi = d^\pi w + w^*\lambda\otimes R_\lambda$
and $Y = Y^\pi + \lambda(Y) R_\lambda$, and $X = (Y, v) \in \Omega^0(w^*T(Q \times \R))$.
Denote $\kappa = \lambda(Y)$ and $\upsilon = ds(v)$. Then we have
\bea
D\Upsilon_1(w)(Y,v) & = & \delbar^{\nabla^\pi}Y^\pi + B^{(0,1)}(Y^\pi) +  T^{\pi,(0,1)}_{dw}(Y^\pi) \nonumber\\
&{}& \quad + \frac{1}{2}\kappa \cdot  \left((\CL_{R_\lambda}J)J(\del^\pi w)\right)
\label{eq:Dwdelbarpi}\\
D\Upsilon_2(u)(Y,v) & = &  w^*(\CL_Y \lambda) \circ j- \CL_v ds = d\kappa \circ j - d\upsilon
+ w^*(Y \rfloor d\lambda) \circ j
\nonumber\\
\label{eq:Dwddot}
\eea
where $B^{(0,1)}$ and $T_{dw}^{\pi,(0,1)}$ are the $(0,1)$-components of $B$ and
$T_{dw}^\pi$, where $B, \, T_{dw}^\pi: \Omega^0(w^*TQ) \to \Omega^1(w^*\xi)$ are
zero-order differential operators given by
$$
B(Y) =
- \frac{1}{2}  w^*\lambda \left((\CL_{R_\lambda}J)J Y\right)
$$
and
$$
T_{dw}^\pi(Y) = \pi T(Y,dw)
$$
respectively.
\end{thm}
\emph{We often omit $w^*$ from the Lie derivative $\CL_Y \lambda$ and from the interior product
$Y \rfloor d\lambda$ when $Y$ is already a vector field along the map $w$ below.}

We can also express the operator $D\Upsilon(u)$
in the following matrix form
\be\label{eq:matrix-form}
D\Upsilon(u)= \left(
\begin{matrix}
 \delbar^{\nabla^\pi} + B^{(0,1)} +  T^{\pi,(0,1)}_{dw} &, & \frac{1}{2}(\cdot) \cdot
  \left((\CL_{R_\lambda}J)J(\del^\pi w)\right)\\
\left((\cdot)^\pi \rfloor d\lambda\right)\circ j &, & \delbar
\end{matrix}
\right)
\ee
with respect to the decomposition
$$
\left(Y, \upsilon \frac{\del}{\del s}\right) = \left(Y^\pi+\kappa R_\lambda, \upsilon\frac{\del}{\del s}\right)
\cong (Y^\pi, \upsilon + i\kappa)
$$
in terms of the splitting
$$
T(Q\times \R) = \xi \oplus (\span\{R_\lambda\} \oplus T\R) \cong \xi \oplus \C.
$$

Now we evaluate $D\Upsilon_1(w)$ more explicitly.
We have  the expression of $B^{(0,1)}(Y)$
$$
B^{(0,1)}(Y) =
-\frac{1}{4}\left(w^*\lambda\,  \pi ((\CL_{R_\lambda}J)J Y)
+ w^*\lambda \circ j\, \pi (\CL_{R_\lambda}J)Y \right).
$$
\begin{rem}\label{rem:choice-of-connection}
Abstractly the linearization of the equation is well-defined
on shell which does not depend on the choice of connections.
The upshot of making a good choice of connection is to have a
good formula that makes it easier to extract its consequence from it.
\end{rem}

\subsection{The punctured case}
\label{subsec:punctured}

We consider some choice of weighted Sobolov spaces
$$
\CW^{k,p}_{\delta;\eta}\left(\dot \Sigma,Q \times \R;\vec \gamma^+, \vec \gamma^-\right)
$$
as he off-shell function space and linearize the map
$$
(w,\widetilde f) \mapsto \left(\delbar^\pi w,  d\widetilde f\right).
$$
This linearization operator then becomes cylindrical in cylindrical
coordinates near the punctures.

The local model of the tangent space  of $\CW^{k,p}_\delta(\dot \Sigma, Q;J;\gamma^+,\gamma^-)$ at
$w \in C^\infty_\delta(\dot \Sigma,Q) \subset W^{k,p}_\delta(\dot \Sigma, Q)$ is given by
\be\label{eq:tangentspace}
\Gamma_{s^+,s^-} \oplus W^{k,p}_\delta(w^*TQ)
\ee
where $W^{k,p}_\delta(w^*TQ)$ is the Banach space
\beastar
&{} & \{Y = (Y^\pi, \lambda(Y)\, R_\lambda)
\mid e^{\frac{\delta}{p}|\tau|}Y^\pi \in W^{k,p}(\dot\Sigma, w^*\xi), \,
\lambda(Y) \in W^{k,p}(\dot \Sigma, \R)\}\\
& \cong & W^{k,p}(\dot \Sigma, \R) \cdot R_\lambda(w) \oplus W^{k,p}(\dot\Sigma, w^*\xi).
\eeastar
Here we measure the various norms in terms of the triad metric of the triad $(Q,\lambda,J)$.

To describe the choice of $\delta > 0$, we need to recall the
\emph{covariant linearization} of the map
$
D\Phi_{\lambda, T}: W^{1,2}(z^*\xi) \to L^2(z^*\xi)
$
of the map
$$
\Phi_{\lambda,T}: z \mapsto \dot z - T\, R_\lambda(z)
$$
for a given $T$-periodic Reeb orbit $(T,z)$. The operator has the
expression
\be\label{eq:DUpsilon}
D\Phi_{\lambda, T} = \nabla_t^\pi - \frac{T}{2}(\CL_{R_\lambda}J) J
\ee
where $\nabla_t^\pi$ is the covariant derivative
with respect to the pull-back connection $z^*\nabla^\pi$ along the Reeb orbit
$z$ and $(\CL_{R_\lambda}J) J$ is a (pointwise) symmetric operator with respect to
the triad metric. (See Lemma 3.4 \cite{oh-wang:connection}.)
\begin{rem}\label{rem:covariant-linearlization}
 Again this \emph{covariant} linearization map
can be defined along any smooth curve and does not depend on
the choice of connection along the Reeb chords.
\end{rem}

We choose $\delta> 0$ so that $0 < \delta/p < 1$ is smaller than the
spectral gap
\be\label{eq:gap}
\text{gap}(\gamma^+,\gamma^-): = \min_{i,j}
\{d_{\text H}(\text{spec}A_{(T_i,z_i)},0),\, d_{\text H}(\text{spec}A_{(T_j,z_j)},0)\}.
\ee

We now
provide details of the Fredholm theory and the index calculation.
Fix an elongation function $\rho: \R \to [0,1]$
so that
\beastar
\rho(\tau) & = & \begin{cases} 1 \quad & \tau \geq 1 \\
0 \quad & \tau \leq 0
\end{cases} \\
0 & \leq & \rho'(\tau) \leq 2.
\eeastar

Then we consider sections of $w^*TQ$ by
\be\label{eq:barY}
\overline Y_i = \rho(\tau - R_0) R_\lambda(\gamma^+_k(t)),\quad
\underline Y_j = \rho(\tau + R_0) R_\lambda(\gamma^+_k(t))
\ee
and denote by $\Gamma_{s^+,s^-} \subset \Gamma(w^*TQ)$ the subspace defined by
$$
\Gamma_{s^+,s^-} = \bigoplus_{i=1}^{s^+} \R\{\overline Y_i\} \oplus \bigoplus_{j=1}^{s^-} \R\{\underline Y_j\}.
$$
Let $k \geq 2$ and $p > 2$. We denote by
$$
\CW^{k,p}_\delta(\dot \Sigma, Q;J;\gamma^+,\gamma^-), \quad k \geq 2
$$
the Banach manifold such that
$$
\lim_{\tau \to \infty}w((\tau,t)_i) = \gamma^+_i(T_i(t+t_i)), \quad
\lim_{\tau \to - \infty}w((\tau,t)_j) = \gamma^-_j(T_j(t-t_j))
$$
for some $t_i, \, t_j \in S^1$, where
$$
T_i = \int_{S^1} (\gamma^+_i)^*\lambda, \, T_j = \int_{S^1} (\gamma^-_j)^*\lambda.
$$
Here $t_i,\, t_j$ depends on the given analytic coordinate and the parameterization of
the Reeb orbits.

Now for each given $w \in \CW^{k,p}_\delta:= \CW^{k,p}_\delta(\dot \Sigma, Q;J;\gamma^+,\gamma^-)$,
we consider the Banach space
$$
\Omega^{(0,1)}_{k-1,p;\delta}(w^*\xi)
$$
the $W^{k-1,p}_\delta$-completion of $\Omega^{(0,1)}(w^*\xi)$ and form the bundle
$$
\CH^{(0,1)}_{k-1,p;\delta}(\xi) = \bigcup_{w \in \CW^{k,p}_\delta} \Omega^{(0,1)}_{k-1,p;\delta}(w^*\xi)
$$
over $\CW^{k,p}_\delta$. Then we can regard the assignment
$$
\Upsilon_1: (w,f) \mapsto \delbar^\pi w
$$
as a smooth section of the bundle $\CH^{(0,1)}_{k-1,p;\delta}(\xi) \to \CW^{k,p}_\delta$.

Furthermore the assignment
$$
\Upsilon_2: (w,f) \mapsto w^*\lambda \circ j - f^*ds
$$
defines a smooth section of the bundle
$$
\Omega^1_{k-1,p}(u^*\CV) \to \CW^{k,p}_\delta.
$$
We have already computed the linearization of each of these maps in the previous section.

With these preparations, the following is a corollary of exponential estimates established
in \cite{oh-wang:CR-map1}.

\begin{prop}[Corollary 6.5 \cite{oh-wang:CR-map1}]\label{prop:on-containedin-off}
Assume $\lambda$ is nondegenerate.
Let $w:\dot \Sigma \to Q$ be a contact instanton and let $w^*\lambda = a_1\, d\tau + a_2\, dt$.
Suppose
\bea
\lim_{\tau \to \infty} a_{1,i} = -Q(p_i), &{}& \, \lim_{\tau \to \infty} a_{2,i} = T(p_i)\nonumber\\
\lim_{\tau \to -\infty} a_{1,j} = -Q(p_j) , &{}& \, \lim_{\tau \to -\infty} a_{2,j} = T(p_j)
\eea
at each puncture $p_i$ and $q_j$.
Then $w \in \CW^{k,p}_\delta(\dot \Sigma, Q;J;\gamma^+,\gamma^-)$.
\end{prop}

Now we are ready to describe the moduli space of lcs instantons with prescribed
asymptotic condition as the zero set
\be\label{eq:defn-MM}
\CM(\dot \Sigma, Q;J;\gamma^+,\gamma^-) = \left(\CW^{k,p}_\delta(\dot \Sigma, Q;J;\gamma^+,\gamma^-)
\oplus \CW^{k,p}_\delta(\dot \Sigma, \R)\right) \cap \Upsilon^{-1}(0)
\ee
whose definition does not depend on the choice of $k, \, p$ or $\delta$ as long as $k\geq 2, \, p>2$ and
$\delta > 0$ is sufficiently small. One can also vary $\lambda$ and $J$ and define the universal
moduli space whose detailed discussion is postponed.

In the rest of this section, we establish the Fredholm property of
the linearization map
$$
D\Upsilon_{(\lambda,T)}(u): \Omega^0_{k,p;\delta}(u^*T(Q \times \R);J;\gamma^+,\gamma^-) \to
\Omega^{(0,1)}_{k-1,p;\delta}(w^*\xi) \oplus \Omega^{(0,1)}_{k-1,p}(u^*T\R)
$$
and compute its index. Here we also denote
$$
\Omega^0_{k-1,p;\delta}(u^*T(Q \times \R);J;\gamma^+,\gamma^-) =
W^{k-1,p}_\delta(u^*T(Q \times \R);J;\gamma^+,\gamma^-)
$$
for the semantic reason.

For this purpose, we remark that
as long as the set of punctures is non-empty, the symplectic vector bundle
$w^*\xi \to \dot \Sigma$ is trivial. We recall that $\overline \Sigma$ stands for the real blow-up
of the boundary punctured Riemann surface $\dot \Sigma$.  We denote by
$
\Phi: E \to \overline \Sigma \times \R^{2n}
$
a trivialization of $E \to \overline \Sigma$
and by
$$
\Phi_i^+: = \Phi|_{\del_i^+ \overline \Sigma}, \quad \Phi_j^- = \Phi|_{\del_j^- \overline \Sigma}
$$
its restrictions on the corresponding boundary components of $\del \overline \Sigma$.
Using the cylindrical structure near the punctures,
we can extend the bundle to the bundle $E \to \overline \Sigma$.
We then consider the following set
\beastar
\CS &: = & \{A: [0,1] \to Sp(2n,\R) \mid 1 \not \in \text{spec}(A(1)), \\
&{}& \hskip0.5in A(0) = id, \, \dot A(0) A(0)^{-1} = \dot A(1) A(1)^{-1} \}
\eeastar
of regular paths in $Sp(2n,\R)$ and denote by $\mu_{CZ}(A)$ the Conley-Zehnder index of
the paths following \cite{robbin-salamon:Maslovindex}. Recall that for each closed Reeb orbit $\gamma$ with a fixed
trivialization of $\xi$, the covariant linearization $A_{(T,z)}$ of the Reeb flow along $\gamma$
determines an element $A_\gamma \in \CS$. We denote by $\Psi_i^+$ and $\Psi_j^-$
the corresponding paths induced from the trivializations $\Phi_i^+$ and $\Phi_j^-$ respectively.

We have the decomposition
$$
\Omega^0_{k,p;\delta}(w^*T(Q \times \R);J;\gamma^+,\gamma^-) =
\Omega^0_{k,p;\delta}(w^*\xi) \oplus \Omega^0_{k,p;\delta}(u^*\CV)
$$
and again the operator
\be\label{eq:DUpsilonu}
D\Upsilon_{(\lambda,T)}(u): \Omega^0_{k,p;\delta}(w^*T(Q \times \R);J;\gamma^+,\gamma^-) \to
\Omega^{(0,1)}_{k-1,p;\delta}(w^*\xi) \oplus \Omega^{(0,1)}_{k-1,p;\delta}(u^*\CV)
\ee
which is decomposed into
$$
D\Upsilon_1(u)(Y,v)\oplus D\Upsilon_2(u)
$$
where the summands are given as in
\eqref{eq:Dwdelbarpi} and \eqref{eq:Dwddot} respectively. We see therefrom that
$D\Upsilon_{(\lambda,T)}$ is the first-order differential operator whose first-order part
is given by the direct sum operator
$$
(Y^\pi,(\kappa, \upsilon)) \mapsto \delbar^{\nabla^\pi} Y^\pi \oplus (d\kappa \circ j - d\upsilon)
$$
where we write $(Y,v) = \left(Y^\pi + \kappa R_\lambda, \upsilon \frac{\del}{\del s}\right)$
for $\kappa = \lambda(Y), \, \upsilon = ds(v)$.
Here we have
$$
\delbar^{\nabla^\pi} : \Omega^0_{k,p;\delta}(w^*\xi;J;\gamma^+,\gamma^-) \to
\Omega^{(0,1)}_{k-1,p;\delta}(w^*\xi)
$$
and the second summand can be written as the standard Cauchy-Riemann operator
$$
\delbar: W^{k,p}(\dot \Sigma;\C) \to \Omega^{(0,1)}_{k-1,p}(\dot \Sigma,\C); \quad
\upsilon + i \kappa =: \varphi \mapsto
\delbar \varphi.
$$

The following proposition can be derived from the arguments used by
Lockhart and McOwen \cite{lockhart-mcowen}. However before applying their
general theory, one needs to pay some preliminary
measure to handle the fact that the order of the operators $D\Upsilon(w)$ are
different depending on the direction of $\xi$ or on that of $R_\lambda$.

\begin{prop}\label{prop:fredholm} Suppose $\delta > 0$ satisfies the inequality
$$
0< \delta < \min\left\{\frac{\text{\rm gap}(\gamma^+,\gamma^-)}{p}, \frac{2}{p}\right\}
$$
where $\text{\rm gap}(\gamma^+,\gamma^-)$ is the spectral gap, given in \eqref{eq:gap},
of the asymptotic operators $A_{(T_j,z_j)}$ or $A_{(T_i,z_i)}$
associated to the corresponding punctures. Then the operator
\eqref{eq:DUpsilonu} is Fredholm.
\end{prop}
\begin{proof} We first note that the operators $\delbar^{\nabla^\pi} + T^{\pi,(0,1)}_{dw}  + B^{(0,1)}$ and
$\delbar$ are Fredholm: The relevant a priori coercive $W^{k,2}$-estimates for any integer $k \geq 1$
for the derivative $dw$ on the punctured Riemann surface $\dot \Sigma$ with cylindrical metric
near the punctures are established in \cite{oh-wang:CR-map1} for the operator
$\delbar^{\nabla^\pi}  + T^{\pi,(0,1)}_{dw} + B^{(0,1)}$ and the one for $\delbar$ is standard.
From this, the standard interpolation inequality establishes the $W^{k,p}$-estimates
for $D\Upsilon(w)$ for all $k \geq 2$ and $p \geq 2$.

\begin{prop}\label{prop:off-diagonal}
The off-diagonal terms decay exponentially fast as $|\tau|\to \infty$.
\end{prop}
\begin{proof} For the $(1,2)$-term, we derive from $\delbar^\pi w = 0$
$$
\left(\frac{\del w}{\del \tau}\right)^\pi + J \left(\frac{\del w}{\del t}\right)^\pi = 0.
$$
Therefore we have
$$
\del^\pi w\left(\frac{\del}{\del \tau}\right) = \frac12\left(\left(\frac{\del w}{\del \tau}\right)^\pi - J \left(\frac{\del w}{\del t}\right)^\pi \right)
= - J \left(\frac{\del w}{\del t}\right)^\pi.
$$
By the exponential convergence $\frac{\del w}{\del t} \to T R_\lambda(\gamma_\infty(t))$, we derive
$$
J \del^\pi w\left(\frac{\del}{\del \tau}\right)
= \left(\frac{\del w}{\del t}\right)^\pi \to 0
$$
since $\frac{\del w}{\del t} \to T R_\lambda$. Therefore
the off-diagonal term converges to the zero operator exponentially fast.

For the $(2,1)$-term, we evaluate
\beastar
(Y^\pi \rfloor d\lambda)\circ j\left(\frac{\del}{\del \tau}\right) & = & d\lambda\left(Y, \frac{\del w}{\del t}\right)\\
(Y^\pi \rfloor d\lambda)\circ j\left(\frac{\del}{\del t}\right) & = & - d\lambda\left(Y, \frac{\del w}{\del \tau}\right).
\eeastar
Therefore we have derived
$$
(Y^\pi \rfloor d\lambda)\circ j = d\lambda\left(Y, \frac{\del w}{\del t}\right) \, d\tau
- d\lambda\left(Y, \frac{\del w}{\del \tau}\right)\, dt.
$$
Therefore we have shown
$$
\left((\cdot)^\pi \rfloor d\lambda\right)\circ j \to  d\lambda\left(\cdot, \frac{\del w}{\del t}\right) \, d\tau - d\lambda\left(\cdot , \frac{\del w}{\del \tau}\right)\, dt.
$$
Since
$\frac{\del w}{\del t} \to T R_\lambda$, the first term converges to zero, and the second term converges to
$$
- d\lambda (\cdot, J TR_\lambda) = T d\lambda \left(\cdot, \frac{\del}{\del s}\right) = 0.
$$
This finishes the proof.
\end{proof}

Therefore it can be homotoped to the block-diagonal form, i.e., into the direct sum operator
$$
\left(\delbar^{\nabla^\pi} + T^{\pi,(0,1)}_{dw}  + B^{(0,1)}\right)\oplus \delbar
$$
via a continuous path of Fredholm operators given by
$$
s \in [0,1] \mapsto \left(
\begin{matrix}
 \delbar^{\nabla^\pi} + B^{(0,1)} +  T^{\pi,(0,1)}_{dw} &, &\frac{s}{2}(\cdot) \cdot
 \left((\CL_{R_\lambda}J)J(\del^\pi w)\right)\\
s \left((\cdot)^\pi \rfloor d\lambda\right)\circ j &, &\delbar
\end{matrix}
\right)
$$
from $s =1$ to $s = 0$. The Fredholm property of this path follows from the fact that
the off-diagonal terms are $0$-th order linear operators.
\end{proof}

Then by the continuous invariance of the Fredholm index, we obtain
\be\label{eq:indexDXiw}
\operatorname{Index} D\Upsilon_{(\lambda,T)}(w) =
\operatorname{Index} \left(\delbar^{\nabla^\pi} + T^{\pi,(0,1)}_{dw}  + B^{(0,1)}\right)
 + \operatorname{Index}(\delbar).
\ee
Therefore it remains to compute the latter two indices.

We denote by $m(\gamma)$ the multiplicity of the Reeb orbit in general.
Then we have the following index formula.

\begin{thm}\label{thm:indexforDUpsilon} We fix a trivialization
$\Phi: E \to \overline \Sigma$ and denote
by $\Psi_i^+$ (resp. $\Psi_j^-$) the induced symplectic paths associated to the trivializations
$\Phi_i^+$ (resp. $\Phi_j^-$) along the Reeb orbits $\gamma^+_i$ (resp. $\gamma^-_j$) at the punctures
$p_i$ (resp. $q_j$) respectively. Then we have
\bea
&{}& \operatorname{Index} (\delbar^{\nabla^\pi} + T^{\pi,(0,1)}_{dw}  + B^{(0,1)}) \nonumber\\
& = &
n(2-2g-s^+ - s^-) + 2c_1(w^*\xi)  + (s^+ + s^-) \nonumber \\
&{}& \quad  + \sum_{i=1}^{s^+} \mu_{CZ}(\Psi^+_i)- \sum_{j=1}^{s^-} \mu_{CZ} (\Psi^-_j)
\label{eq:Indexdelbarpi}
\eea
\be
\operatorname{Index} (\delbar) = 2\sum_{i=1}^{s^+} m(\gamma^+_i)+ 2\sum_{j=1}^{s^-} m(\gamma^-_j) -2 g.
\label{eq:indexdelbar}
\ee
In particular,
\bea\label{eq:indexforDUpsilon}
&{}& \operatorname{Index} D\Upsilon_{(\lambda,T)}(u) \nonumber\\
& = & n(2-2g-s^+ - s^-) + 2c_1(w^*\xi)\nonumber\\
&{}&  + \sum_{i=1}^{s^+} \mu_{CZ}(\Psi^+_i)
- \sum_{j=1}^{s^-} \mu_{CZ}(\Psi^-_j)\nonumber \\
&{}&  +
\sum_{i=1}^{s^+} (2m(\gamma^+_i)+1) + \sum_{j=1}^{s^-}( 2m(\gamma^-_j)+1)  - 2g.
\eea
\end{thm}
\begin{proof} The formula \eqref{eq:Indexdelbarpi} can be immediately derived from
the general formula given in the top of p. 52 of Bourgeois's thesis \cite{bourgeois}:
The summand $(s^+ + s^-)$ comes from the factor $\Gamma_{s^+,s^-}$ in the decomposition
\eqref{eq:tangentspace} which has dimension $s^+ + s^-$.

So it remains to compute the index \eqref{eq:indexdelbar}.
To compute the (real) index of $\delbar$, we consider the Dolbeault complex
$$
0 \to \Omega^0(\Sigma; D) \to \Omega^1(\Sigma;D) \to 0
$$
where $D = D^+ + D^-$ is the divisor associated to the set of punctures
$$
D^+ =  \sum_{i=1}^{s^+}m(\gamma^+_i) p_i, \quad
D^- = \sum_{j=1}^{s^-} m(\gamma^-_j) q_j
$$
where $m(\gamma^+_i)$ (resp. $m(\gamma^-_j)$) is the multiplicity of the Reeb orbit
$\gamma^+_i$ (resp. $\gamma^-_j$). The standard Riemann-Roch formula then gives rise to
the formula for the Euler characteristic
\beastar
\chi(D) & = & 2\dim_\C H^0(D) - 2\dim_\C H^1(D) = 2\operatorname{deg} (D) - 2g\\
& = &\sum_{i=1}^{s^+} 2m(\gamma^+_i)+ \sum_{j=1}^{s^-} 2m(\gamma^-_j) - 2g.
\eeastar

This finishes the proof.
\end{proof}

\begin{rem} We can also symplectify the Fredholm theory from \cite{oh:contacton-transversality} and
the index calculation given in \cite{oh-yso:index} in the similar way to give rise to
the relevant theory for the pseudoholomorphic curves on the symplectization which we leave
to the readers as an exercise.
\end{rem}

\section{Exponential asymptotic analysis}
\label{sec:asymptotic-analysis}

Recalling that for any $\widetilde J$-holomorphic curve $(w,f)$,
$w$ is a contact instanton for $J$ on $Q$. Furthermore we have
$$
(w,f)^*T(Q \times \R) = w^*TQ \oplus f^*T\R
= w^*\xi \oplus \span_\R \left\{\frac{\del}{\del s}, R_\lambda \right\}.
$$
The last splitting is respected by
by the \emph{canonical connection} of the almost Hermitian manifold
$$
(Q\times \R, d\lambda + ds \wedge \lambda, \widetilde J).
$$
Indeed, we have
$$
\nabla^{\text{\rm can}} = \nabla^\pi \oplus \nabla_0
$$
where $\nabla^\pi = \nabla|_{\xi}$ and $\nabla_0$ is the trivial
connection on $\span_\R\{\frac{\del}{\del s}, R_\lambda\}$.
(Recall Proposition \ref{prop:canonical-intro}.)

\begin{rem} Here again we would like to emphasize the usage of
the canonical connection of the almost Hermitian manifold, not
the Levi-Civita connection, admits this splitting.
\end{rem}

\subsection{Definition of asymptotic operators and their formulae}
\label{subsec:asymptotic-operators}

Now we study a finer analysis of the asymptotic behavior along the Reeb orbit. Our discussion thereof is close to the one given in
\cite[Section 11.2 \&11.5]{oh-wang:CR-map2} where the more general
Morse-Bott case is studied.

For this purpose, we  evaluate the linearization operator $D\Upsilon$ against $\frac{\del}{\del \tau}$. We have already checked the
off-diagonal terms of the matrix representation of $D\Upsilon(w)$
decays exponentially fast in the direction $\tau$
in the previous section and so we have only to examine the diagonal terms $D\Upsilon_1(w)$ and $D\Upsilon_2(w)$.

First we consider $D\Upsilon_2$ and rewrite
$$
D\Upsilon_2 =\delbar = \frac12(\del_\tau + i \del_t).
$$
Therefore we have the asymptotic operator
\be\label{eq:Aperp}
A^\perp_{(\lambda,J,\nabla)} : = i \del_t
\ee
which does not depend on the choice of $J \in \CJ_\lambda(Q,\xi)$.
The eigenfunction expansions for this operator is nothing but
the standard Fourier series for $f \in L^2(S^1,Q)$.

This being said, we now focus on the $Q$-component $D\Upsilon_1$ of the
asymptotic operator and compute
\be\label{eq:Dupsilon1-ddtau}
D\Upsilon_1(w)\left(\frac{\del}{\del \tau}\right) =
\frac12(\nabla_\tau^\pi +  J \nabla_t^\pi)
+ T_{dw}^{\pi,(0,1)}\left(\frac{\del}{\del \tau}\right)
+ B^{(0,1)}\left(\frac{\del}{\del \tau}\right).
\ee
In fact, this is nothing but the left hand side of \eqref{eq:fundamental-isothermal}.
We write
$$
2 D\Upsilon(w)\left(\frac{\del}{\del \tau}\right) = \nabla_\tau^\pi + A^\pi_{(\lambda,J,\nabla)}.
$$
and define the family of operators
$$
A^\tau_{(\lambda,J,\nabla)}: \Gamma(w_\tau^*\xi) \to \Gamma(w_\tau^*\xi)
$$
Thanks to the exponential convergence of $w_\tau \to \gamma_\pm$ as $\tau \to \pm \infty$,
we can take the limit of the conjugate operators
\be\label{eq:Atau}
\Pi_\tau^\infty A^\tau_{(\lambda,J,\nabla)}(\Pi_\tau^\infty)^{-1}: \Gamma(\gamma_\pm^*\xi) \to \Gamma(\gamma_\pm^*\xi)
\ee
as $\tau \to \pm \infty$ respectively, where $\Pi_\tau^\infty$ is the parallel transport
along the short geodesics from $w(\tau,t)$ to $w(\infty,t)$. This conjugate is defined
for all sufficiently large $|\tau|$.

Since the discussion at $\tau = -\infty$ will be the same, we will focus our discussion
on the case at $\tau = + \infty$ from now on.

\begin{defn}[Asymptotic operator]\label{defn:asymptotic-operator}
 Let $(\tau,t)$ be
the cylindrical (or strip-like) coordinate, and let $\nabla^\pi$ be
the almost Hermitian connection on $w^*\xi$ induced by
the contact triad connection $\nabla$ of $(Q,\lambda,J)$.
We define the \emph{asymptotic operator} of a contact
instanton $w$ to be the limit operator
\be\label{eq:asymptotic-operator}
A^\pi_{(\lambda,J,\nabla)}: = \lim_{\tau \to +\infty}\Pi_\tau^\infty A^\tau_{(\lambda,J,\nabla)}(\Pi_\tau^\infty)^{-1}.
\ee
\end{defn}
Obviously we can define the asymptotic operator at negative punctures in the similar way.

\begin{rem}\label{rem:asymptotic-operator}
The upshot of our definition lies in its naturality depending
only on the given adapted pair $(\lambda, J)$ and its associated
triad connection. Other literature definition of the asymptotic
operators is given differently in the way how the dependence
on the given $J$ of  the final formula is hard to analyze.
This definition can be given any connection $\nabla$ on $Q$ that has the property that $\nabla (\xi) \subset \xi$
and $\nabla^\pi (J|_\xi) = 0$. Using the exponential convergence, one can
check that the definition does not depend on the choice of such connections.
Therefore it is conceivable that a good choice of
connection will facilitate the study asymptotic operators,
 which is precisely what is happening by our choice of contact triad connection.
\end{rem}

Now we find the formula for this limit operator \emph{with respect to the contact triad connection}.
Since $T(R_\lambda, \cdot) = 0$, $\left(\frac{\del w}{\del \tau}\right)^\pi
= - J \left(\frac{\del w}{\del t}\right)^\pi$ and
$\frac{\del w}{\del t}(\tau, \cdot) \mapsto T R_\lambda$ exponentially fact, we obtain
$$
2 T_{dw}^{\pi,(0,1)}(\frac{\del}{\del \tau}) = T^\pi\left(\frac{\del w}{\del \tau},\cdot\right) \to  0.
$$
On the other hand, we have
$$
2 B^{(0,1)}\left(\frac{\del}{\del \tau}\right) =  -\frac12 \lambda\left(\frac{\del w}{\del \tau}\right) \CL_{R_\lambda}J
-\frac12 \lambda\left(\frac{\del w}{\del t}\right) J \CL_{R_\lambda}J.
$$
We note that the right hand side
converges to $\frac{T}{2} J \nabla_{R_\lambda}$ since
$w^*\lambda \to T \, dt$ as $\tau \to \infty$. This convergence proves

\begin{prop}\label{prop:A} Let $\nabla$ be the contact triad connection
associated to any adapted pair $(\lambda, J)$. Then the asymptotic operator $A^\pi_{(\lambda,J,\nabla)}$ is given by
$$
A^\pi_{(\lambda,J,\nabla)} = - J \nabla_t  + \frac{T}{2} \CL_{R_\lambda}J J.
$$
\end{prop}
\begin{proof} We recall  $T(R_\lambda, \cdot) = 0$ and the identity
\eqref{eq:nablalambdaY}. On the other hand we have
$\frac{\del w}{\del t}(\infty,t) = T R_\lambda$.
Combining the two and above convergence of $B^{(0,1)}(\del_\tau)$,
we have finished the proof.
\end{proof}

Now one may define the full asymptotic operator $A^\pi_{(\lambda,J,\nabla)}$ of the
pseudoholomorphic curves to be the operator
$$
A^\pi_{(\lambda,J,\nabla)}: \gamma^*\xi \oplus \C \to  \gamma^*\xi \oplus \C
$$
defined by
$$
A_{(\lambda,J,\nabla)} = A^\pi_{(\lambda,J,\nabla)} \oplus
A^\perp_{(\lambda,J,\nabla)}.
$$
Here $\C$ stands for the pull-back bundle
$$
(\pi \circ u_\tau)^*\left(\R \left\{\frac{\del}{\del s}, R_\lambda \right\}\right)
 = w_\tau^*\left(\R \left\{\frac{\del}{\del s}, R_\lambda \right\}\right)
\cong \R \left \{\frac{\del}{\del s}, R_\lambda \right\}
$$
which is canonically trivialized, and hence may be regarded as a vector bundle over
the curves $w_\tau$ on $Q$. (Compare this with \cite[Definition 2.28]{pardon:contacthomology}.)

\subsection{Asymptotic operator and the Levi-Civita connection}
\label{subsec:asymptotic-operator}

Up until now, we have emphasized the usage of triad connection which
give rise to an optimal form of tensorial expression. For a finer study of
asymptotic operators, it turns out that the Levi-Civita connection is much
better than the triad connection and hence
we temporarily switch back to the Levi-Civita connection. The main reason
behind this switch is the following surprising property of
Levi-Civita connection of the triad metric.

We recall that while $\nabla_Y J = 0$ for all $Y \in \xi$,
$\nabla_{R_\lambda}^{\text{\rm LC}}J \neq 0$ for the contact triad connection in general.
 (In fact, the latter holds if and only if $R_\lambda$ is a Killing vector field, i.e.,
$\nabla R_\lambda = 0$ with respect to $\nabla$. See \cite[Remark 2.4]{oh-wang:CR-map1}.)
However while $\nabla^{\text{\rm LC}}_Y J \neq 0$ for the Levi-Civita
connection in general, the Levi-Civita connection carries the following
extremely useful property for the study of asymptotic operators.

\begin{prop}[Lemma. 6.1 \cite{blair:book},
Proposition 4 \cite{oh-wang:connection}]\label{prop:nablaLCJ}
$\nabla^{\text{\rm LC}}_{R_\lambda} J = 0$.
\end{prop}

To utilize this property in our study of eigenfunction analysis for the
linearized operator $A_{(\lambda,K,\nabla)} = -\frac12 J {\nabla}_t$, we convert
it in terms of the Levi-Civita connection. For this purpose, the
following lemma is crucial.

\begin{lem}[Lemma 6.2 \cite{blair:book}, Lemma 9 \cite{oh-wang:connection}]
For any $Y \in \xi$, we have
$$
\nabla_Y^{\text{\rm LC}} R_\lambda = \frac12 JY
+ \frac12 \CL_{R_\lambda}JJ Y.
$$
\end{lem}

Therefore combining Proposition \ref{prop:A} and
this lemma, we can rewrite the linearized operator in terms of
the Levi-Civita connection as follows. It reveals a rather remarkable
property of the asymptotic operator \emph{computed in terms of
the triad connection} that it becomes
a Hermitian operator for any choice of adapted
pair $(\lambda, J)$ for the contact manifold $(Q,\xi)$.

\begin{cor}\label{cor:A-in-LC} Let $(\lambda, J)$ be any adapted pair
and let $\nabla^{\text{\rm LC}}$ be the Levi-Civita connection of the
triad metric of $(Q,\lambda,J)$. For given contact instanton $w$ with its action $\int \gamma^*\lambda =T$
at a puncture, let $A_{(\lambda, J, \nabla)}$ be the asymptotic operator of $w$
at the  with cylindrical coordinate $(\tau,t)$. Then
\begin{enumerate}
\item $[\nabla_t^{\text{\rm LC}},J] (= \nabla_t^{\text{\rm LC}} J)= 0$,
\item $A^\pi_{(\lambda,J,\nabla)} =
- J\nabla_t +  \frac{T}{2} \CL_{R_\lambda}JJ
= - J\nabla_t^{\text{\rm LC}} - \frac{T}{2} Id
+  \frac{T}{2} \CL_{R_\lambda}JJ.$
\end{enumerate}
In particular, it induces a $J$-Hermitian operator on $(\xi, g|_\xi)$
 with respect to the triad metric $g$.
\end{cor}
The same kind of property also holds for the open string case
for the Legendrian pair $(R_0,R_1)$. This explicit formula for
the asymptotic operator enables us to prove the following series of
perturbation results in \cite{kim-oh:asymp-analysis}
on the eigenfunctions and eigenvalues of
the asymptotic operators \emph{under the perturbation of $J$'s}
inside the set $\CJ_{\lambda}$ of $\lambda$-adapted CR almost
complex structures $J$.
The above explicit form of asymptotic operator on
$J$ and $\lambda$ enable us to prove the following genericity
in terms of the choice of contact triads $(Q,\lambda,J)$ of
contact manifold $(Q,\xi)$.

\begin{thm}[Simpleness of eigenvalues; \cite{kim-oh:asymp-analysis}]
\label{thm:simple-eigenvalue}  Let $(Q,\xi)$ be a contact manifold.
Assume that $\lambda$ is
nondegenerate. For a generic choice of adapted pair $(\lambda,J)$,
 all eigenvalues $\mu_i$ of the asymptotic
operator are simple for all
closed Reeb orbits.
\end{thm}

\subsection{Finer asymptotic behavior}
\label{subsec:tangentplane}

Once these theorems at our disposal, we can further proceed
with a finer asymptotic convergence result for the pseudoholomorphic
curves on symplectization in a canonical covariant tensorial way.
Similar study is given in \cite{HWZ:asymptotics,HWZ:asymptotics-correction,HWZ:smallarea}
in 3 dimension using special coordinates followed by
adjusting the almost complex structure along the given Reeb orbit.

Let $u = (w,f)$ be any finite energy $\widetilde J$-pseudoholomorphic curve from
a punctured Riemann surface $\dot \Sigma$ to the symplectization of $M = Q \times \R$.
Then $w$ is a contact instanton of the triad
$(Q,\lambda,J)$. Consider a puncture and let
$$
w : [\tau_0, \infty) \times S^1 \rightarrow Q
$$
be the representation of $w$ in the cylindrical coordinate
for some large $\tau_0 >0$ and denote $w_{\tau} := w(\tau, \cdot).$

We have shown that there exists a isospeed Reeb trajectory $(\gamma,T)$ such that
$$
w(\tau, t) = w_{\tau}(t) \rightarrow w_{\infty} = \gamma(T(\cdot) ) \text{ as } \tau \rightarrow \infty.
$$
Denote by $\gamma_T: S^1 \to Q$ the curve given by $\gamma_T(t): = \gamma(Tt)$.

At each $\tau$ we have the operator
$$
 A^{\tau,\pi}_{(\lambda,J,\nabla)} : \Gamma(w_{\tau}^* \xi) \rightarrow \Gamma(w_{\tau}^* \xi)
$$
 defined by
\be\label{eq:Ataupi}
 A^{\tau,\pi}_{(\lambda,J,\nabla)}: =  J \nabla_t^\pi
+ T_{dw}^{\pi,(0,1)}\left(\frac{\del}{\del \tau}\right)
+ B^{(0,1)}\left(\frac{\del}{\del \tau}\right).
\ee
We have shown that as $\tau \to \infty$ the operator
$$
\Pi^{\infty}_{\tau} \circ A^{\tau,\pi}_{(\lambda,J,\nabla)} \circ (\Pi_{\tau}^{\infty})^{-1} : \Gamma(\gamma_{T}^*\xi) \rightarrow  \Gamma(\gamma_{T}^*\xi),
$$
converges to the asymptotic operator
\be\label{eq:Ainftypi}
A^\pi_{(\lambda,J,\nabla)} : \Gamma(\gamma_T^* \xi) \rightarrow \Gamma(\gamma_T^*\xi),
\ee

We have shown that there exists an iso-speed Reeb trajectory
$(\gamma,T)$ such that
$$
w(\tau, t) = w_{\tau}(t) \rightarrow w_{\infty} = \gamma(T(\cdot) ) \text{ as } \tau \rightarrow \infty.
$$
Denote by $\gamma_T: S^1 \to Q$ the curve given by $\gamma_T(t): = \gamma(Tt)$.
At each $\tau$ we have the operator
$$
 A^{\tau,\pi}_{(\lambda,J,\nabla)} : \Gamma(w_{\tau}^* \xi) \rightarrow \Gamma(w_{\tau}^* \xi)
$$
 defined by
\be\label{eq:Ataupi}
 A^{\tau,\pi}_{(\lambda,J,\nabla)}: =  J \nabla_t^\pi
+ T_{dw}^{\pi,(0,1)}\left(\frac{\del}{\del \tau}\right)
+ B^{(0,1)}\left(\frac{\del}{\del \tau}\right).
\ee
We have shown that as $\tau \to \infty$ the operator
$$
\Pi^{\infty}_{\tau} \circ A^{\tau,\pi}_{(\lambda,J,\nabla)} \circ (\Pi_{\tau}^{\infty})^{-1} : \Gamma(\gamma_{T}^*\xi) \rightarrow  \Gamma(\gamma_{T}^*\xi),
$$
converges to the asymptotic operator
\be\label{eq:Ainftypi}
A^\pi_{(\lambda,J,\nabla)} :=
A^\pi_{(\lambda,J,\nabla)}(\gamma): \Gamma(\gamma_T^* \xi) \rightarrow \Gamma(\gamma_T^*\xi),
\ee
given by $A^\pi_{(\lambda,J,\nabla)}(\gamma)$.

Then the following theorem describes the asymptotic behavior of
finite energy contact instanon $w$ as $|\tau| \to \infty$, which is the analog to \cite[Theorem 1.4]{HWZ:asymptotics}.

\begin{thm}[Asymptotic behavior; Theorem 1.10 \cite{kim-oh:asymp-analysis}]
\label{thm:behavior} Assume that $\gamma_T$ is nondegenerate.
Consider t
$$
\nu(\tau,t): = \frac{\zeta(\tau,t)}{\|\zeta(\tau)\|_{L^2(w_\tau^*\xi}}, \quad
\alpha(\tau) = \frac{d}{d\tau}\log \|\zeta\|_{L^2(w_\tau^*\xi)}
$$
Then
\begin{enumerate}
\item  Either $\zeta(\tau,t) = 0$ for all $(\tau,t) \in [\tau_0,\infty)$
\item or otherwise we have the following:
\begin{enumerate}
\item There exists an eigenvector $e$ of $A^\pi_{(\lambda,J,\nabla)}$ of eigenvalue
$\mu$ such that $\nu_\tau \to e$ as $\tau \to \infty$ and
$$
\zeta(\tau,t) = e^{\int_{\tau_0}^\tau \alpha(s)\, ds} (e(t) + \widetilde r(\tau,t))
$$
\item
There exist constants $\delta >0,$ and $C_{\beta}$ for all multi-indices $\beta = (\beta_1,\beta_2) \in \mathbb{N} \times \mathbb{N}$ such that
$$
\sup_{(\tau,t)} |(\nabla^\beta \nu)(\tau,t)| \leq C_\beta
$$
for all $\tau \geq \tau_0$ and $t \in S^1$.
\end{enumerate}
\end{enumerate}
\end{thm}

Note that the above representation formula of $\zeta$ implies the following convergence of
the tangent plane.

\begin{cor}[Convergence of tangent plane; Corollary 1.11 \cite{kim-oh:asymp-analysis}]\label{cor:tangentplane}
Assume the second alternative in
Theorem \ref{thm:behavior-intro} and denote
$$
P(\tau,t) := \Image dw(\tau,t) \in \text{\rm Gr}_2(\xi_{w(\tau,t)})
$$
where $\text{\rm Gr}_2(\xi_x)$ is the set of 2 dimensional subspaces of the
contact hyperplane $\xi_x \subset T_xQ$.
Then $P(\tau,t) \to \span_\R\{ e(t), J e(t)\}$ exponentially fast in $C^\infty$ topology
uniformly in $t \i S^1$.
\end{cor}

Finally we would like to just mention that the same asymptotic study
can be made in a straightforward way by incorporating the boundary condition
by now as done in \cite{oh-yso:index}, \cite{oh:contacton-transversality},
\cite{oh:contacton-gluing}.
The exponential convergence statement can be derived by a boot-strap argument
using the exponential convergence result on $\zeta$ as $\tau \to \infty$ given in
Section \ref{sec:exponential-convergence}. (See \cite{oh-wang:CR-map2},
\cite{oh-yso:index} for the details of this exponential convergence and the boot-strap
argument.)

Combining all the above, we have finished the proof of Theorem \ref{thm:behavior}.

Now we prove the following two theorems which are analogs to
Theorem 5.1 and Theorem 5.2 respectively whose proofs follow those in
\cite[Section 5]{HWZ:behavior} after translating their coordinate
exposition into that of covariant tensorial exposition, and so omitted.

\begin{thm}[Compare with Theorem 5.1 \cite{HWZ:behavior}]
Consider a nondegenerate finite energy contact instanton $w$ on cylindrical
coordinates $(\tau,t) \in [0, \infty) \times S^1$.  Then there
exists some $\tau_0 > 0$ such that the sets
$$
\{(\tau,t) \mid w(\tau,z) \in \Image \gamma, \, \tau \geq \tau_0 \}
$$
$$
\{(\tau,z) \mid \zeta(z) = 0, \, \tau \geq \tau_0 \}
$$
consist of finitely many points.
\end{thm}

The following naturally occurs in the study of asymptotic behavior of bubbles
which provides a strong embdding control of such planes in 3 dimension.

\begin{thm}[Compare with Theorem 5.2 \cite{HWZ:behavior}]
Consider a nondegenerate finite energy contact instanton $w$ on
$\C \cong S^2 \setminus \{pt\}$. Then each of the sets
$$
\{z \in \C \mid w(z) \in \Image \gamma \}, \quad 
\{z \in \C \mid d^\pi w(z) = 0\}
$$
consists of finitely many points.
\end{thm}
Significance of this theorem is not as large as in 3 dimension in higher dimension
because they are expected to be empty for a generic
choice of $J$ by the transversality argument by the dimensional reason.

\appendix

\section{Wedge products of vector-valued forms}
\label{appen:forms}

In this section, we continue with the setting from Appendix \ref{appen:weitzenbock}.
To be specific, we assume $(P, h)$ is a Riemannian manifold of dimension $n$ with metric $h$, and denote by $D$
the Levi--Civita connection. $E\to P$ is a vector bundle with inner product $\langle \cdot, \cdot\rangle$
and $\nabla$ is a connection of $E$ which is compatible with $\langle \cdot, \cdot\rangle$.

We remark that we include this section for the sake of completeness of our treatment of vector valued forms, and
the content of this appendix is not used in any section of this article.
Actually one can derive exponential decay using the differential inequality method from the formulas we provide here.
We leave the proof to interested reader or to the earlier arXiv version of \cite{oh-wang:CR-map1}.

The wedge product of forms can be extended to $E$-valued forms by defining
\beastar
\wedge&:&\Omega^{k_1}(E)\times \Omega^{k_2}(E)\to \Omega^{k_1+k_2}(E)\\
\beta_1\wedge\beta_2&=&\langle \zeta_1, \zeta_2\rangle\,\alpha_1\wedge\alpha_2,
\eeastar
where $\beta_1=\alpha_1\otimes\zeta_1\in \Omega^{k_1}(E)$ and $\beta_2=\alpha_2\otimes\zeta_2\in \Omega^{k_2}(E)$
are $E$-valued forms.

\begin{lem}\label{lem:inner-star}
For $\beta_1, \beta_2\in \Omega^k(E)$,
$$
\langle \beta_1, \beta_2\rangle=*(\beta_1\wedge *\beta_2).
$$
\end{lem}
\begin{proof}
Write $\beta_1=\alpha_1\otimes\zeta_1$ and $\beta_2=\alpha_2\otimes\zeta_2$. Then
\beastar
*(\beta_1\wedge *\beta_2)&=&*\big((\alpha_1\otimes\zeta_1)\wedge ((*\alpha_2)\otimes\zeta_2)\big)\\
&=&*(\langle\zeta_1, \zeta_2\rangle\,\alpha_1\wedge *\alpha_2)\\
&=&\langle\zeta_1, \zeta_2\rangle\,*(\alpha_1\wedge *\alpha_2)\\
&=&\langle\zeta_1, \zeta_2\rangle\, h(\alpha_1, \alpha_2)\\
&=&\langle \beta_1, \beta_2\rangle.
\eeastar
\end{proof}

The following lemmas exploit the compatibility of $\nabla$ with the inner product $\langle \cdot, \cdot\rangle$.

\begin{lem}
$$
d(\beta_1\wedge\beta_2)=d^\nabla\beta_1\wedge \beta_2+(-1)^{k_1}\beta_1\wedge d^\nabla\beta_2,
$$
where $\beta_1\in \Omega^{k_1}(E)$ and $\beta_2\in \Omega^{k_2}(E)$
are $E$-valued forms.
\end{lem}

\begin{proof}
We write $\beta_1=\alpha_1\otimes\zeta_1$ and $\beta_2=\alpha_2\otimes\zeta_2$ and calculate
\beastar
d(\beta_1\wedge\beta_2)&=&d(\langle \zeta_1, \zeta_2\rangle\,\alpha_1\wedge\alpha_2)\\
&=&d\langle \zeta_1, \zeta_2\rangle\wedge\alpha_1\wedge\alpha_2+\langle \zeta_1, \zeta_2\rangle\,d(\alpha_1\wedge\alpha_2)\\
&=&\langle \nabla\zeta_1, \zeta_2\rangle\wedge\alpha_1\wedge\alpha_2+\langle \zeta_1, \nabla\zeta_2\rangle\wedge\alpha_1\wedge\alpha_2\\
&{}&+\langle \zeta_1, \zeta_2\rangle\,d\alpha_1\wedge\alpha_2+(-1)^{k_1}\langle \zeta_1, \zeta_2\rangle\,\alpha_1\wedge d\alpha_2,
\eeastar
while
\beastar
d^\nabla\beta_1\wedge \beta_2&=&d^\nabla(\alpha_1\otimes\zeta_1)\wedge(\alpha_2\otimes\zeta_2)\\
&=&(d\alpha_1\otimes\zeta_1+(-1)^{k_1}\alpha_1\wedge\nabla\zeta_1)\wedge(\alpha_2\otimes\zeta_2)\\
&=&\langle\zeta_1, \zeta_2\rangle\,d\alpha_1\wedge\alpha_2+\langle\nabla\zeta_1, \zeta_2\rangle\wedge\alpha_1\wedge\alpha_2.
\eeastar
A similar calculation shows that
$$
(-1)^{k_1}\beta_1\wedge d^\nabla\beta_2=
(-1)^{k_1}\langle \zeta_1, \zeta_2\rangle\,\alpha_1\wedge d\alpha_2+\langle \zeta_1, \nabla\zeta_2\rangle\wedge\alpha_1\wedge\alpha_2.
$$
Summing these up,  we get the equality we want.
\end{proof}

\begin{lem}\label{lem:metric-property}
Assume $\beta_0\in \Omega^k(E)$ and $\beta_1\in \Omega^{k+1}(E)$,
then we have
$$
\langle d^\nabla \beta_0, \beta_1\rangle-(-1)^{n(k+1)}\langle \beta_0, \delta^\nabla\beta_1\rangle
=*d(\beta_0\wedge *\beta_1).
$$
\end{lem}
\begin{proof}
We calculate
\beastar
*d(\beta_0\wedge *\beta_1)&=&*\big(d^\nabla\beta_0\wedge*\beta_1+(-1)^k\beta_0\wedge (d^\nabla*\beta_1)\big)\\
&=&\langle d^\nabla\beta_0, \beta_1\rangle
+(-1)^n*\big(\beta_0\wedge *(*d^\nabla *\beta_1\big)\\
&=&\langle d^\nabla\beta_0, \beta_1\rangle
-(-1)^{n(k+1)}\langle\beta_0, \delta^\nabla\beta_1\rangle.
\eeastar

\end{proof}

\section{The Weitzenb\"ock formula for vector-valued forms}\label{appen:weitzenbock}

In this appendix, we recall the standard Weitzenb\"ock formulas applied to our
current circumstance. A good exposition on the general Weitzenb\"ock formula is
provided in the appendix of \cite{freed-uhlen}.

Assume $(P, h)$ is a Riemannian manifold of dimension $n$ with metric $h$, and $D$ is the Levi--Civita connection.
Let $E\to P$ be any vector bundle with inner product $\langle\cdot, \cdot\rangle$,
and assume $\nabla$ is a connection on $E$ which is compatible with $\langle\cdot, \cdot\rangle$.

For any $E$-valued form $s$, calculating the (Hodge) Laplacian of the energy density
of $s$,  we get
\beastar
-\frac{1}{2}\Delta|s|^2=|\nabla s|^2+\langle Tr\nabla^2 s, s\rangle,
\eeastar
where for $|\nabla s|$ we mean the induced norm in the vector bundle $T^*P\otimes E$, i.e.,
$|\nabla s|^2=\sum_i|\nabla_{E_i}s|^2$ with $\{E_i\}$ an orthonormal frame of $TP$.
$Tr\nabla^2$ denotes the connection Laplacian, which is defined as
$Tr\nabla^2=\sum_i\nabla^2_{E_i, E_i}s$,
where $\nabla^2_{X, Y}:=\nabla_X\nabla_Y-\nabla_{\nabla_XY}$.

Denote by $\Omega^k(E)$ the space of $E$-valued $k$-forms on $P$. The connection $\nabla$
induces an exterior derivative by
\beastar
d^\nabla&:& \Omega^k(E)\to \Omega^{k+1}(E)\\
d^\nabla(\alpha\otimes \zeta)&=&d\alpha\otimes \zeta+(-1)^k\alpha\wedge \nabla\zeta.
\eeastar
It is not hard to check that for any $1$-forms, equivalently one can write
$$
d^\nabla\beta (v_1, v_2)=(\nabla_{v_1}\beta)(v_2)-(\nabla_{v_2}\beta)(v_1),
$$
where $v_1, v_2\in TP$.

We extend the Hodge star operator to $E$-valued forms by
\beastar
*&:&\Omega^k(E)\to \Omega^{n-k}(E)\\
*\beta&=&*(\alpha\otimes\zeta)=(*\alpha)\otimes\zeta
\eeastar
for $\beta=\alpha\otimes\zeta\in \Omega^k(E)$.

Define the Hodge Laplacian of the connection $\nabla$ by
$$
\Delta^{\nabla}:=d^{\nabla}\delta^{\nabla}+\delta^{\nabla}d^{\nabla},
$$
where $\delta^{\nabla}$ is defined by
$$
\delta^{\nabla}:=(-1)^{nk+n+1}*d^{\nabla}*.
$$
The following lemma is important for the derivation of the Weitzenb\"ock formula.
\begin{lem}\label{lem:d-delta}Assume $\{e_i\}$ is an orthonormal frame of $P$, and $\{\alpha^i\}$ is the dual frame.
Then we have
\beastar
d^{\nabla}&=&\sum_i\alpha^i\wedge \nabla_{e_i}\\
\delta^{\nabla}&=&-\sum_ie_i\rfloor \nabla_{e_i}.
\eeastar
\end{lem}
\begin{proof}Assume $\beta=\alpha\otimes \zeta\in \Omega^k(E)$. Then
\beastar
d^{\nabla}(\alpha\otimes \zeta)&=&(d\alpha)\otimes \zeta+(-1)^k\alpha\wedge\nabla\zeta\\
&=&\sum_i\alpha^i\wedge \nabla_{e_i}\alpha\otimes\zeta+(-1)^k\alpha\wedge\nabla\zeta.
\eeastar
On the other hand,
\beastar
\sum_i\alpha^i\wedge \nabla_{e_i}(\alpha\otimes\zeta)&=&
\sum_i\alpha^i\wedge\nabla_{e_i}\alpha\otimes\zeta+\alpha^i\wedge\alpha\otimes\nabla_{e_i}\zeta\\
&=&\sum_i\alpha^i\wedge \nabla_{e_i}\alpha\otimes\zeta+(-1)^k\alpha\wedge\nabla\zeta,
\eeastar
so we have proved the first statement.

For the second equality, we compute
\beastar
\delta^{\nabla}(\alpha\otimes\zeta)&=&(-1)^{nk+n+1}*d^{\nabla}*(\alpha\otimes\zeta)\\
&=&(\delta\alpha)\otimes\zeta+(-1)^{nk+n+1}*(-1)^{n-k}(*\alpha)\wedge\nabla\zeta\\
&=&-\sum_ie_i\rfloor \nabla_{e_i}\alpha\otimes\zeta+\sum_i(-1)^{nk-k+1}*((*\alpha)\wedge\alpha^i)\otimes\nabla_{e_i}\zeta\\
&=&-\sum_ie_i\rfloor \nabla_{e_i}\alpha\otimes\zeta-\sum_ie_i\rfloor \alpha\otimes\nabla_{e_i}\zeta\\
&=&-\sum_ie_i\rfloor \nabla_{e_i}(\alpha\otimes\zeta).
\eeastar
\end{proof}

\begin{thm}[Weitzenb\"ock Formula]\label{thm:weitzenbock}
Assume $\{e_i\}$ is an orthonormal frame of $P$, and $\{\alpha^i\}$ is the dual frame. Then
when applied to $E$-valued forms
\beastar
\Delta^{\nabla}=-Tr\nabla^2+\sum_{i,j}\alpha^j\wedge (e_i\rfloor R(e_i,e_j)(\cdot))
\eeastar
where $R$ is the curvature tensor of the bundle $E$ with respect to the connection $\nabla$.
\end{thm}
\begin{proof}Since the right hand side of the equality is independent of the choice of orthonormal basis,
and it is a pointwise formula,
we can take the normal coordinates $\{e_i\}$ at a point $p\in P$ (and $\{\alpha^i\}$ the dual basis), i.e., $h_{ij}:=h(e_i, e_j)(p)=\delta_{ij}$ and $dh_{i,j}(p)=0$,
 and prove that the given formula holds at $p$ for such coordinates. For the Levi--Civita connection, the condition $dh_{i,j}(p)=0$
 of the normal coordinate is equivalent to letting
$\Gamma^k_{i,j}(p):=\alpha^k(D_{e_i}e_j)(p)$ be $0$.

For $\beta\in \Omega^k(E)$, using Lemma \ref{lem:d-delta} we calculate
\beastar
\delta^{\nabla}d^{\nabla}\beta&=&-\sum_{i,j}e_i\rfloor \nabla_{e_i}(\alpha^j\wedge\nabla_{e_j}\beta)\\
&=&-\sum_{i,j}e_i\rfloor (D_{e_i}\alpha^j \wedge\nabla_{e_j}\beta+\alpha^j\wedge \nabla_{e_i}\nabla_{e_j}\beta).
\eeastar
At the point $p$, the first term vanishes, and we get
\beastar
\delta^{\nabla}d^{\nabla}\beta(p)&=&-\sum_{i,j}e_i\rfloor (\alpha^j\wedge \nabla_{e_i}\nabla_{e_j}\beta)(p)\\
&=&-\sum_i\nabla_{e_i}\nabla_{e_i}\beta(p)+\sum_{i,j}\alpha^j\wedge (e_i\rfloor\nabla_{e_i}\nabla_{e_j}\beta)(p)\\
&=&-\sum_i\nabla^2_{e_i, e_i}\beta(p)+\sum_{i,j}\alpha^j\wedge (e_i\rfloor\nabla_{e_i}\nabla_{e_j}\beta)(p).
\eeastar
Also,
\beastar
d^{\nabla}\delta^{\nabla}\beta&=&-\sum_{i,j}\alpha^i\wedge \nabla_{e_i}(e_j\rfloor \nabla_{e_j}\beta)\\
&=&-\sum_{i,j}\alpha^i\wedge (e_j\rfloor \nabla_{e_i}\nabla_{e_j}\beta)
-\sum_{i,j}\alpha^i\wedge ((D_{e_i}e_j) \rfloor \nabla_{e_j}\beta),
\eeastar
As before, at the point $p$, the second term vanishes.

Now we sum the two parts $d^\nabla\delta^\nabla$ and $\delta^\nabla d^\nabla$ and get
$$
\Delta^{\nabla}\beta(p)=-\sum_i\nabla^2_{e_i, e_i}\beta(p)
+\sum_{i,j}\alpha^j\wedge (e_i\rfloor R(e_i,e_j)\beta)(p).
$$
\end{proof}

In particular, when acting on zero forms, i.e., sections of $E$, the second term on the right hand side vanishes, and there is
$$
\Delta^{\nabla}=-Tr\nabla^2.
$$
When acting on full rank forms, the above also holds by easy checking.

When $\beta\in \Omega^1(E)$, which is the case we use in this article, there is the following
\begin{cor}
For $\beta=\alpha\otimes \zeta\in \Omega^1(E)$, the Weizenb\"ock formula can be written as
$$
\Delta^{\nabla}\beta=-\sum_i\nabla^2_{e_i, e_i}\beta
+\ric^{D*}(\alpha)\otimes\zeta+\ric^{\nabla}\beta,
$$
where $\ric^{D*}$ denotes the adjoint of $\ric^D$, which acts on $1$-forms.

In particular, when $P$ is a surface, we have
\bea\label{eq:bochner-weitzenbock}
\Delta^{\nabla}\beta&=&-\sum_i\nabla^2_{e_i, e_i}\beta
+K\cdot\beta+\ric^{\nabla}(\beta)\nonumber\\
-\frac{1}{2}\Delta|\beta|^2&=&|\nabla \beta|^2-\langle \Delta^\nabla \beta, \beta\rangle +K\cdot|\beta|^2+\langle\ric^{\nabla}(\beta), \beta\rangle.
\eea
where $K$ is the Gaussian curvature of the surface $P$.
\end{cor}


\section{Abstract framework of the three-interval method}
\label{sec:three-interval}

In this section, we provide the method in proving exponential decay using the abstract framework of the three-interval
method from \cite{oh-wang:CR-map2} referring interested readers \cite[Section 11]{oh-wang:CR-map2}.
We remark that the method can deal with the case with any exponentially decaying perturbation too
(see Theorem \ref{thm:three-interval}).

The three-interval method is based on the following analytic lemma.
\begin{lem}[\cite{mundet-tian} Lemma 9.4]\label{lem:three-interval}
For a sequence of nonnegative numbers $\{x_k\}_{k=0, 1, \cdots, N}$, if there exists some constant $0<\gamma<\frac{1}{2}$ such that
$$
x_k\leq \gamma(x_{k-1}+x_{k+1})
$$
 for every $1\leq k\leq N-1$, then it follows
$$
x_k\leq x_0\xi^{-k}+x_N\xi^{-(N-k)},\quad k=0, 1, \cdots, N,
$$
where $\xi:=\frac{1+\sqrt{1-4\gamma^2}}{2\gamma}$.
\end{lem}

\begin{rem}\label{rem:three-interval}
\begin{enumerate}
\item If we write $\gamma=\gamma(c):=\frac{1}{e^c+e^{-c}}$ where $c>0$ is uniquely determined by $\gamma$, then the conclusion
can be written into the exponential form
$$
x_k\leq x_0e^{-ck}+x_Ne^{-c(N-k)}.
$$
\item
For an infinite nonnegative sequence $\{x_k\}_{k=0, 1, \cdots}$, if we have a uniform bound of
in addition,
then the exponential decay follows as
$$
x_k\leq x_0e^{-ck}.
$$
\end{enumerate}
\end{rem}

The analysis of proving the exponential decay will be carried on a Banach bundle $\CE\to [0, \infty)$ modelled by the Banach space $\E$, for which we mean every fiber $\E_\tau$ is identified with the Banach space $\E$ smoothly depending on $\tau$. We omit this identification if there is no way of confusion.

First we emphasize the base $[0,\infty)$ is non-compact and
carries a natural translation map for any positive number $r$,
which is $\sigma_r: \tau \mapsto \tau + r$.
We introduce the following definition
which ensures us to study the sections in local trivialization after taking a subsequence.

\begin{defn}\label{def:unif-local-tame} Let $\CE$ be a Banach bundle modelled with a Banach space $\E$ over $[0, \infty)$.
Let $[a,b] \subset [0,\infty)$ be any given bounded interval and let
$s_k \to \infty$ be any given sequence. A \emph{tame family of trivialization}
over $[a,b]$ relative to the sequence $s_k$ is defined to be
a sequence of trivializations $\{\Phi_k\}:\CE|_{[a,b]} \to [a,b] \times \E$
$$
\Phi_k: \sigma^*_{s_\cdot}\CE|_{[a+s_k, b+s_k]} \to [a,b] \times \E
$$
for $k\geq 0$ satisfying the following: There exists a sufficiently large $k_0 > 0$
such that for any $k \geq k_0$ the bundle map
$$
\Phi_{k_0+k} \circ \Phi_{k_0}^{-1}: [a,b] \times \E \to [a,b] \times \E
$$
satisfies
\be\label{eq:locallytame}
\|\nabla_\tau^l(\Phi_{k_0+k} \circ \Phi_{k_0}^{-1})\|_{\CL(\E,\E)} \leq C_l<\infty
\ee
for constants $C_l = C_l(|b-a|)$ depending only on $|b-a|$, $l=0, 1, \cdots$.

We call $\CE$ \emph{uniformly locally tame}, if it carries a tame family of
trivializations over $[a,b]$ relative to the sequence $s_k$ for any given bounded
interval $[a,b] \subset [0,\infty)$ and a sequence $s_k \to \infty$.
\end{defn}

\begin{defn} Suppose $\CE$ is uniformly locally tame. We say a connection $\nabla$ on
$\CE$ is \emph{uniformly locally tame} if the push-forward $(\Phi_k)_*\nabla_\tau$ can be written as
$$
(\Phi_k)_*\nabla_\tau = \frac{d}{d\tau} + \Gamma_k(\tau)
$$
for any tame family $\{\Phi_k\}$ so that
$\sup_{\tau \in [a,b]}\|\Gamma_k(\tau)\|_{\CL(\E,\E)} < C$ for some $C> 0$ independent of $k$'s.
\end{defn}

\begin{defn}
Consider a pair $\CE_2 \subset \CE_1$ of uniformly locally tame bundles, and a bundle map
$B: \CE_2 \to \CE_1$.
We say $B$ is \emph{uniformly locally bounded}, if for any compact set $[a,b] \subset [0,\infty)$ and
any sequence $s_k \to \infty$, there exists a subsequence, still denoted by $s_k$, a sufficiently large $k_0 > 0$ and tame families
$\Phi_{1,k}$, $\Phi_{2,k}$ such that for any $k\geq 0$
\be\label{eq:loc-uni-bdd}
\sup_{\tau \in [a,b]} \|\Phi_{i,k_0+k} \circ B \circ \Phi_{i,k_0}^{-1}\|_{\CL(\E_2, \E_1)} \leq C
\ee
where $C$ is independent of $k$.
\end{defn}

For a given locally tame pair $\CE_2 \subset \CE_1$, we denote by $\CL(\CE_2, \CE_1)$ the set
of bundle homomorphisms which are uniformly locally bounded.

\begin{lem}
If $\CE_1, \, \CE_2$ are uniformly locally tame, then so is $\CL(\CE_2, \CE_1)$.
\end{lem}

\begin{defn}\label{defn:precompact} Let $\CE_2 \subset \CE_1$ be as above and let $B \in \CL(\CE_2, \CE_1)$.
We say $B$ is \emph{pre-compact} on $[0,\infty)$ if for any locally tame families $\Phi_1, \Phi_2$,
there exists a further subsequence such that
$
\Phi_{1, k_0+k} \circ B \circ \Phi_{1, k_0}^{-1}
$
converges to some $B_{\Phi_1\Phi_2;\infty}\in \CL(\Gamma([a, b]\times \E_2), \Gamma([a, b]\times \E_1))$.
\end{defn}

Assume $B$ is a bundle map from $\CE_2$ to $\CE_1$ which is uniformly locally bounded,
where $\CE_1 \supset \CE_2$ are uniformly locally tame with tame families
$\Phi_{1,k}$, $\Phi_{2,k}$. We can write
$$
\Phi_{2,k_0+k} \circ (\nabla_\tau + B) \circ \Phi_{1,k_0}^{-1} = \frac{\del}{\del \tau} + B_{\Phi_1\Phi_2, k}
$$
as a linear map from $\Gamma([a,b]\times \E_2)$ to $\Gamma([a,b]\times \E_1)$, since $\nabla$ is uniformly
locally tame.

Next we introduce the following notion of coerciveness.

\begin{defn}\label{defn:localcoercive}
Let $\CE_1, \, \CE_2$ be as above and $B: \CE_2 \to \CE_1$ be a uniformly locally bounded bundle map.
We say the operator
$$
\nabla_\tau + B: \Gamma(\CE_2) \to \Gamma(\CE_1)
$$
is \emph{uniformly locally coercive}, if the following holds:
\begin{enumerate}
\item For any pair of bounded closed intervals $I, \, I'$ with $I \subset \operatorname{Int}I'$,
\be\label{eq:coercive}
\|\zeta\|_{L^2(I,\CE_2)} \leq C(I,I') \|\nabla_\tau \zeta + B \zeta\|_{L^2(I,\CE_1)}.
\ee
\item if for given bounded sequence $\zeta_k \in \Gamma(\CE_2)$ satisfying
$$
\nabla_\tau \zeta_k + B \zeta_k = L_k
$$
with $\|L_k\|_{\CE_1}$ bounded on a given compact subset $K \subset [0,\infty)$,
there exists a subsequence, still denoted by $\zeta_k$, that uniformly converges in $\CE_2$.
\end{enumerate}
\end{defn}

\begin{rem}
Let $E \to [0,\infty)\times S$ be a (finite dimensional) vector bundle and denote by
$W^{k,2}(E)$ the set of $W^{k,2}$-section of $E$ and $L^2(E)$ the set of $L^2$-sections. Let
$D: L^2(E)\to L^2(E)$ be a first order elliptic operator with cylindrical end.
Denote by $i_\tau: S \to [0,\infty)\times S$ the natural inclusion map. Then
there is a natural pair of Banach bundles $\CE_2 \subset \CE_1$ over $[0, \infty)$ associated to $E$, whose fiber
is given by $\CE_{1,\tau}=L^2(i_\tau^*E)$, $\CE_{2,\tau} = W^{1,2}(i_\tau^*E)$.
Furthermore assume $\CE_i$ for $i=1, \, 2$ is uniformly local tame if $S$ is a compact manifold (without boundary).
Then $D$ is uniformly locally coercive, which follows from the elliptic bootstrapping and the Sobolev's embedding.
\end{rem}

Finally we introduce the notion of asymptotically cylindrical operator $B$.
\begin{defn}\label{defn:asympcylinderical} We call $B$ \emph{locally asymptotically cylindrical} if the following holds:
Any subsequence limit $B_{\Phi_1\Phi_2;\infty}$ appearing in Definition \ref{defn:precompact}
is a \emph{constant} section,
and $\|B_{\Phi_1\Phi_2, k}-\Phi_{2,k_0+k} \circ  B \circ \Phi_{1,k_0}^{-1}\|_{\CL(\E_i, \E_i)}$ converges to zero
as $k\to \infty$ for both $i =1, 2$.
\end{defn}

Now we specialize to the case of Hilbert bundles $\CE_2 \subset \CE_1$ over $[0,\infty)$ and assume that
$\CE_1$ carries a connection {which is compatible with the Hilbert inner product of $\CE_1$. We denote by $\nabla_\tau$ the associated covariant
derivative. We assume that $\nabla_\tau$ is uniformly locally tame.

Denote by $L^2([a,b];\CE_i)$ the space of $L^2$-sections $\zeta$ of $\CE_i$ over
$[a,b]$, i.e., those satisfying
$$
\int_a^b |\zeta(\tau)|_{\CE_i}^2\, dt < \infty.
$$
where $|\zeta(\tau)|_{\CE_i}$ is the norm with respect to the given Hilbert bundle structure of $\CE_i$.

\begin{thm}[Three-Interval Method]\label{thm:three-interval}
Assume $\CE_2\subset\CE_1$ is a pair of Hilbert bundles over $[0, \infty)$ with fibers $\E_2$ and $\E_1$,
and $\E_2\subset \E_1$ is dense.
Let $B$ be a section of the associated bundle $\CL(\CE_2, \CE_1)$ and
$L \in \Gamma(\CE_1)$.
We assume the following:
\begin{enumerate}
\item There exists a covariant derivative $\nabla_\tau$ 
that preserves the Hilbert structure;
\item  $\CE_i$ for $i=1, \, 2$ are uniformly locally tame;
\item  $B$ is precompact, uniformly locally coercive and asymptotically cylindrical;
\item  Every subsequence limit $B_{\infty}$ is a self-adjoint unbounded operator on
$\E_1$ with its domain $\E_2$, and satisfies $\ker B_{\infty} = \{0\}$;
\item There exists some positive number $\delta$
such that any subsequence limiting operator $B_{\infty}$
of the above mentioned pre-compact family has their first non-zero eigenvalues $|\lambda_1| >\delta$;
\item
There exists some $R_0 > 0$, $C_0 > 0$ and $\delta_0 > \delta$ such that
$$
\|L(\tau)\|_{\CE_{1,\tau}} \leq C_0 e^{-\delta_0 \tau}
$$
for all $\tau \geq R_0$.
\end{enumerate}
Then for any (smooth) section $\zeta \in \Gamma(\CE_2)$ with
\be\label{eq:sup-bound}
\sup_{\tau \in [R_0,\infty)} \|\zeta(\tau,\cdot)\|_{\CE_{2,\tau}} < \infty
\ee
and satisfying the equation
\be\label{eq:nabla=L}
\nabla_\tau \zeta + B(\tau) \zeta(\tau) = L(\tau),
\ee
there exist some constants $R$, $C>0$ such that for any $\tau>R$,
$$
\|\zeta(\tau)\|_{\CE_{1,\tau}}\leq C e^{-\delta \tau }.
$$
\end{thm}

\bibliographystyle{amsalpha}

\bibliography{biblio2}

\end{document}